\documentclass{amsart}

\usepackage{booktabs} 
\usepackage{array}
\usepackage{float}

\usepackage{multirow}

\usepackage{amssymb}
\usepackage{hyperref}

\usepackage{xcolor}
\hypersetup{
	colorlinks=true,
	linkcolor=blue,
	filecolor=blue,      
	urlcolor=blue,
	citecolor=cyan,
}
\usepackage{mathtools}

\newtheorem{theorem}{Theorem}[section]
\newtheorem{lemma}[theorem]{Lemma}

\newtheorem{definition}[theorem]{Definition}
\newtheorem{example}[theorem]{Example}
\newtheorem{prop}[theorem]{Proposition}
\newtheorem{coro}[theorem]{Corollary}

\newenvironment{Proof}{\noindent{\bf Proof of Lemma \ref{the first lemma}:}}{\hfill\fbox{}\vspace*{1mm}}
\newenvironment{Proof1}{\noindent{\bf Proof of Lemma \ref{the second lemma}:}}{\hfill\fbox{}\vspace*{1mm}}

\theoremstyle{remark}
\newtheorem{remark}[theorem]{Remark}

\numberwithin{equation}{section}
\usepackage[title]{appendix}

\begin{document}
\title[Uncertainty principles for FMT and associated MO]{Uncertainty principles for free
metaplectic transformation and associated metaplectic operators}

\author{Ping Liang}
	\address{School of Computer Science and Engineering, Faculty of Innovation Engineering, Macau University of Science and Technology}
	\curraddr{}
	\email{pliang0208@163.com}
	\thanks{}
	
	\author{Pei Dang}
	\address{Department of Engneering Science, Faculty of Innovation Engineering, Macau University of Science and Technology}
	\curraddr{}
	\email{pdang@must.edu.mo}
	\thanks{}
	
	\author{Weixiong Mai$^{\star}$}
	\address{School of Computer Science and Engineering, Faculty of Innovation Engineering, Macau University of Science and Technology}
	\curraddr{}
	\email{wxmai@must.edu.mo}
	\thanks{$^{\star}$Corresponding author}
\begin{abstract}
    In this paper, we systematically investigate the Heisenberg-Pauli-Weyl uncertainty principle for free metaplectic transformation, as well as metaplectic operators.
    Specifically, we obtain two different types of the uncertainty principle for free metaplectic transformations in terms of the so-called phase derivative, one of which can be generalized to the $L^p$-case with $1\le p\le 2$. The obtained results are valid not only for free metaplectic transformations but also for general metaplectic operators. In particular, we point out that our results are closely related to those given in \cite{Dias-deGosson-Prata}, and the relationship should be new and not exactly given in the existing literature.

\end{abstract}

\maketitle
\date{}
\dedicatory{}
 \section{Introduction}
 It is well-known that the classical Heisenberg-Pauli-Weyl (HPW) uncertainty principle plays an important role in quantum mechanics, which 
 states that the position and the momentum of a particle
cannot be both determined precisely (e.g. \cite{Fefferman}, \cite{Folland-Sitaram}, \cite{Heisenberg}). In 1946, the HPW uncertainty principle was introduced to signal analysis by Gabor (\cite{Gabor}). It states that 
 a signal
 cannot be sharply localized in
both time and Fourier frequency domains, i.e.,
\begin{equation}\label{OSUP}
    \Delta x^{2}\Delta w^{2}\ge \frac{1}{16\pi^{2}},
\end{equation}
where $\Delta x^{2}=\int_{-\infty}^{\infty}|(x-\left<x\right>_f)f(x)|^{2}\mathrm{d}x
	$ and $\Delta w^{2}=\int_{-\infty}^{\infty}|(w-\left<w\right>_{\widehat{f}})\widehat{f}(w)|^{2}\mathrm{d}w$ with $\left<x\right>_f=\int_{-\infty}^{\infty}
	x\left|f(x)\right|^{2}\mathrm{d}x$ and $\left<w\right>_{
\widehat{f}
    }=\int_{-\infty}^{\infty}
	w|\widehat{f}(w)|^{2}\mathrm{d}w 
	$.
Here $\widehat{f}$ is the Fourier transform of $f$ defined by
 \begin{equation*}
\widehat{f}(w)=\int_{-\infty}^{\infty} f(x) e^{-2\pi i xw}\mathrm{d}x
 \end{equation*}
 if $f\in L^{1}(\mathbb{R})$. When $f(x)$ is written as $f(x)=\left|f(x)\right|e^{2\pi i \varphi(x)}$ with $\Vert f\Vert_{2}=1$, Cohen in \cite{Cohen} obtains 
a stronger version of HPW uncertainty principle, i.e.,
\begin{equation}\label{SUP}
		\Delta x^{2} \Delta w^{2}\ge \frac{1}{16\pi^{2}}+\mathrm{Cov}_{x,w}^{2},
	\end{equation}
	where  $\mathrm{Cov}_{x,w}=\int_{-\infty}^{\infty}(x-\left<x\right>_f)(\varphi^{\prime}(x)-\left<w\right>_{\widehat{f}})\left|f(x)\right|^2\mathrm{d}x$.
    Correspondingly, there have also been some developments for uncertainty principle of self-adjoint operators 
    $\widehat{A}$ and $\widehat{B}$ on a Hilbert space $\mathcal H$. In \cite{Folland}, the author gives an uncertainty principle for self-adjoint operators as follows,
    \begin{equation}\label{OSOUP}
    \Vert(\widehat{A}-\alpha)f\Vert_{2}^{2}
    \Vert (\widehat{B}-\beta)f\Vert_{2}^{2}\ge \frac{1}{4}|
\langle [\widehat{A},\widehat{B}]f,f\rangle|^{2}, \quad f\in D(\widehat{A}\widehat{B})\cap D(\widehat{B}\widehat{A}), 
    \end{equation}
where $\alpha,\beta\in \mathbb{C}, [\widehat{A},\widehat{B}]\triangleq
\widehat{A}\widehat{B}-\widehat{B}\widehat{A}$, $\left<\cdot,\cdot\right>$ is the inner product with $\Vert\cdot\Vert _2\triangleq \left<\cdot,\cdot\right>^{\frac{1}{2}}$,  $D(\widehat{A}\widehat{B})$ and $ D(\widehat{B}\widehat{A})$ are the domains of the products of $\widehat{A}\widehat{B}$ and $\widehat{B}\widehat{A}$ (see (\ref{domain of AB}), (\ref{domain of BA}) for details).
In \cite{Cohen}, a stronger uncertainty principle for self-adjoint operators is given as follows,
\begin{align}\label{SOUP}
   & \Vert(\widehat{A}-\alpha)f\Vert_{2}^{2}
    \Vert (\widehat{B}-\beta)f\Vert_{2}^{2}\ge \frac{1}{4}|
\langle [\widehat{A},\widehat{B}]f,f\rangle|^{2} + |\langle[\widehat{A}-\alpha\widehat{I},\widehat{B}-\beta \widehat{I}]_+f,f\rangle|^{2} , \notag\\
&\quad f\in D(\widehat{A}\widehat{B})\cap D(\widehat{B}\widehat{A}), 
    \end{align}
where $\widehat{I}$ is the identity operator and $[\widehat{A}-\alpha\widehat{I},\widehat{B}-\beta \widehat{I}]_+ \triangleq (\widehat{A}-\alpha\widehat{I})(\widehat{B}-\beta \widehat{I})+(\widehat{B}-\beta \widehat{I})(\widehat{A}-\alpha\widehat{I}).$ Note that
(\ref{OSOUP}) gives (\ref{OSUP}) and 
(\ref{SOUP}) gives (\ref{SUP}) if $\widehat{A}f(x) = xf(x)$ and $\widehat{B}f(x) = \frac{1}{2
\pi i}\frac{df(x)}{dx}$. Later, in \cite{Dang-Deng-Qian} Dang, Deng and Qian give the so-called extra-strong uncertainty principle, which is strictly stronger than (\ref{SUP}).
 Similarly, the authors in \cite{Dang-Deng-Qian} also prove the extra-strong uncertainty principle for self-adjoint operators under some additional conditions. In \cite{Cowling-Price} the authors give the $L^p$-type HPW uncertainty principle for the classical Fourier transform with $1\leq p\leq 2$, and later, a sharper $L^p$-type HPW uncertainty principle is given in \cite{Zhang}. In fact,  uncertainty principles have been widely developed and studied in mathematics since the HPW uncertainty principle was proposed (see e.g. \cite{Cazacu-Flynn-Lam, Cohen1, Dias-Luef-Prata, Donoho-Stark, Escauriaza-Kenig-Ponce, Goh-Goodman, Jiang-Liu-Wu, Jin-Zhang, Kristály, Narcowich-Ward} and the references therein).

The above developments of HPW uncertainty principle are to pursue sharper lower bounds. In fact, 
in the existing literature, a lot of developments of HPW uncertainty principle are based on generalizations of Fourier transform (including one and several variables), such as the fractional Fourier transform (FRFT), the linear canonical transform (LCT) and so on (see \cite{Dang-Deng-Qian2, Kou-Xu-Zhang, Sharma-Joshi1, Xu-Wang-Xu3, Zhao-Tao-Li-Wang, Zhao-Tao-Wang}). 
Note that the Fourier transform and FRFT are two special cases of LCT (\cite{Chen-Fu-Grafakos-Wu, Namias}). In the case of several variables, the  free metaplectic transformation (FMT) could be considered as a generalization of LCT, which was first studied by Folland (\cite{Folland}).
However, there are relatively a few results for FMT (see e.g. \cite{Chen-Dang-Mai,Dias-deGosson-Prata,Ding-Pei,Zhang1,Zhang2,Zhang3}). Based on the above developments of HPW uncertainty principle, the initial purpose of this paper is to give some strong types of HPW uncertainty principle for FMT. 

When preparing this paper, we note the results given by Dias, de Gosson and Prata in \cite{Dias-deGosson-Prata}, which gives an interesting study of HPW uncertainty principle from a metaplectic perspective. 
Their results are as follows.
\begin{prop}[{\cite[Corollary 6]{Dias-deGosson-Prata}}]\label{MPUP}
  Let $f\in L^2(\mathbb{R}^{N})$ with $\Vert f\Vert_{2}=1$. There holds
 \begin{align}\label{Up for operator original}
        &\int_{\mathbb{R}^{N}}\left|x\left(\widehat{M_{1}}f\right)(x)\right|^{2}\mathrm{d}x
\int_{\mathbb{R}^{N}}\left|\xi\left(\widehat{M_{2}}f\right)(\xi)\right|^{2}\mathrm{d}\xi\notag\\
&\ge\frac{1}{16\pi^{2}}\left(\sum_{j=1}^{N}\left|\left(M_{1}JM_{2}^{T}\right)_{jj}\right|\right)^{2},
    \end{align} 
where $ \widehat{M}_{j}$ are associated metaplectic operators of $M_{j}\in\mathrm{Sp}(2N,\mathbb{R}), j=1,2$, 
$
    J=\begin{pmatrix}
0 & I_{N}\\ 
-I_{N} & 0
\end{pmatrix}
,$ and $I_{N}$ is the identity matrix.
\end{prop}

\begin{prop}[{\cite[Theorem 7]{Dias-deGosson-Prata}}]\label{Ro-ScUP}
Let $f$ be such that $\Vert f\Vert_2=1$\\
\noindent and $\int_{\mathbb R^{2N}} {(1+|z|^2)\left|W_\sigma f(z)\right|} dz<\infty.$ There holds
\begin{equation*}
    \Upsilon+\frac{i}{4\pi}\Omega\ge 0,
\end{equation*}
for $\Upsilon=D_{1,2}\Sigma\left(D_{1,2}\right)^T$,
$\Omega=D_{1,2}J\left(D_{1,2}\right)^T$ with $D_{1,2}=\begin{pmatrix}
    A_1 &B_1\\
    A_2 &B_2
\end{pmatrix},$ where $A_j$ and $B_j$ are real $N\times N$ matrices from  $M_j\in\mathrm{Sp}(2N,\mathbb{R})$ with $M_j=\begin{pmatrix}
    A_j& B_j\\
    C_j&D_j
\end{pmatrix}, j=1,2.$
\end{prop}
\noindent Here $\Sigma=\left(\Sigma_{\alpha,\beta}\right)$ is the covariance matrix with
\begin{align}\label{SIGMA}
\begin{split}
\Sigma_{\alpha,\beta}=\Big\langle \left(\frac{\widehat{Z}_{\alpha}\widehat{Z}_{\beta}+\widehat{Z}_{\beta}\widehat{Z}_{\alpha}}{2}
\right)f,f\Big\rangle
=\int_{\mathbb{R}^{2N}}z_{\alpha}z_{\beta}W_{\sigma}f(z)dz,\quad \alpha, \beta=1,\cdots,2N,
\end{split}
\end{align}
where the operators $\widehat{Z}_{\alpha}$ are defined in (\ref{fundamental operator}) and 
$W_{\sigma}f(z)$ is the Wigner function of $f\in L^2(\mathbb{R}^{N})$ for $z=(x,w)\in \mathbb{R}^{2N}$ (see \S \ref{section 2} for its definition).

Their results widely generalize the classical HPW uncertainty principle to many integral transformations other than the Fourier transform. In particular, 
Proposition \ref{MPUP} is corresponding to the classical HPW uncertainty principle for metaplectic operators,
while Proposition \ref{Ro-ScUP} is the analogue of Robertson-Schr$\rm \ddot{o}$dinger uncertainty principle for metaplectic operators, which implies a stronger version of HPW uncertainty principle for metaplectic operators.

 Since FMTs are special metaplectic operators, by specific and nontrivial computations, it turns out that the results of our paper are closely related to Propositions \ref{MPUP} and \ref{Ro-ScUP}. More specifically, a part of our main theorems actually give the same results as those implied by Proposition \ref{Ro-ScUP}. Since this connection between our results and Proposition \ref{Ro-ScUP} is not obvious, and  not shown in existing works, our paper gives direct and completely different proofs of those results and the mentioned connection. Nevertheless, to the authors' knowledge, the results presented in this paper should be new and not exactly given in the literature.  

In the rest of this paper, we always assume  $f\in L^{2}(\mathbb{R}^{N})$ with $\Vert f\Vert_{2}=1$. When $f\in L^{2}(\mathbb{R}^{N})$ is expressed in the form $f(x)=\left|f(x)\right|e^{2\pi i \varphi(x)}$, we assume that for any $1\le j\le N$, the classical partial derivatives $\frac{\partial \left|f\right|}{\partial x_{j}}, \frac{\partial \varphi}{\partial x_{j}}$ and $\frac{\partial f}{\partial x_{j}}$ exist for all $x\in \mathbb{R}^{N}$. Let $\mathcal{L}_{M_{j}}[f](u)$ be the FMT of $f$ with 
 $M_{j}=\begin{pmatrix}
A_{j} & B_{j}\\ 
C_{j} & D_{j}
\end{pmatrix} \in \mathrm{Sp}(2N,\mathbb{R}), j=1,2$ (see \S \ref{section 2} for its definition). Our main results are given as follows.
\smallskip

\noindent \textbf{Main Result I:} (Theorem \ref{Up for MO component})
{\it Let $f(x)=\left|f(x)\right|e^{2\pi i\varphi(x)}, xf(x)$ and $w\widehat{f}(w)\in L^{2}(\mathbb{R}^{N})$. There holds
    \begin{align}\label{main result 2}
&\int_{\mathbb{R}^{N}}\left|u\mathcal{L}_{M_{1}}[f](u)\right|^{2}\mathrm{d}u
\int_{\mathbb{R}^{N}}\left|u\mathcal{L}_{M_{2}}[f](u)\right|^{2}\mathrm{d}u\notag\\
&\ge\biggr[\sum_{j=1}^{N}\biggr(\frac{1}{16\pi^{2}}\left|\left(A_{1}B_{2}^{T}-B_{1}A_{2}^{T}\right)_{jj}\right|^{2}+\Bigr|\Bigr(A_{1}XA_{2}^{T}+B_{1}WB_{2}^{T}\notag\\
&\quad+A_{1}\mathrm{Cov}_{X,W}B_{2}^{T}+B_{1}\left(\mathrm{Cov}_{X,W}\right)^{T}A_{2}^{T}\Big)_{jj}
\Big|^{2}
\bigg)^{\frac{1}{2}}
\bigg]^{2},
\end{align}
where $X, W$ and $\mathrm{Cov}_{X,W}$ are given in Definition \ref{def-jihao}. }

The result of \textbf{Main Result I} is essentially based on estimating the product
$
    \int_{\mathbb{R}^{N}}\left|u_j\mathcal{L}_{M_{1}}[f](u)\right|^{2}\mathrm{d}u
\int_{\mathbb{R}^{N}}\left|u_j\mathcal{L}_{M_{2}}[f](u)\right|^{2}\mathrm{d}u
$
for each $j\in \{1,2,...,N\},$ while the  \textbf{Main Result II}  deals with the product $\int_{\mathbb{R}^{N}}\left|u\mathcal{L}_{M_{1}}[f](u)\right|^{2}\mathrm{d}u
\int_{\mathbb{R}^{N}}\left|u\mathcal{L}_{M_{2}}[f](u)\right|^{2}\mathrm{d}u$. 

\smallskip
\noindent \textbf{Main Result II:} (Theorem \ref{UP for L2 M1 M2})
{\it Let $f(x)=\left|f(x)\right|e^{2\pi i \varphi(x)}, xf(x)$ and $ w\widehat{f}(w)\in L^{2}(\mathbb{R}^{N})$. There holds
     \begin{align}\label{main result 3}
&\int_{\mathbb{R}^{N}}\left|u\mathcal{L}_{M_{1}}[f](u)\right|^{2}\mathrm{d}u\int_{\mathbb{R}^{N}}\left|u\mathcal{L}_{M_{2}}[f](u)\right|^{2}\mathrm{d}u\notag\\
&\ge\frac{\left[\mathrm{tr}\left(A_{1}^{T}B_{2}-A_{2}^{T}B_{1}\right)\right]^{2}}{16\pi^{2}}+\biggr[\int_{\mathbb{R}^{N}}x^{T}A_{2}^{T}A_{1}
x\left|f(x)\right|^{2}\mathrm{d}x\notag\\
&+ \int_{\mathbb{R}^{N}}w^{T}B_{1}^{T}B_{2}
w\left|\widehat{f}(w)\right|^{2}\mathrm{d}w+ \int_{\mathbb{R}^{N}}x^{T}\left(A_{1}^{T}B_{2}+A_{2}^{T}B_{1}
\right)\nabla \varphi(x)\left|f(x)\right|^{2}\mathrm{d}x
\bigg]^{2},\notag\\
\end{align}
where 
$\mathrm{tr}(\cdot)$ denotes the trace of a matrix, and $\nabla=\left(\frac{\partial }{\partial x_{1}}, \dots,\frac{\partial }{\partial x_{N}}\right)^{T}$.}

We note that when $n=1$ \textbf{Main Result II} coincides with \textbf{Main Result I}. The best result of HPW uncertainty principle for LCT in the one dimensional case is given in \cite{Dang-Deng-Qian2}, which is stronger than that of \textbf{Main Result II}. For higher dimensional cases, analogous results of that in \cite{Dang-Deng-Qian2} can only be obtained for special matrices $M$ (see e.g. \cite{Zhang1, Zhang2, Zhang3} and also \S \ref{section 5}).  In comparison, the right side of (\ref{main result 3}) is determined by entire matrices, while that of (\ref{main result 2}) is expressed in terms of components of matrices. 

Although in \S \ref{section 3} we can show that the right side of (\ref{main result 2}) is bigger than that of (\ref{main result 3}), the method in proving \textbf{Main Result II} is more general and can be used to obtain the sharper $L^p$-type HPW uncertainty principle for FMTs with $1\leq p\leq 2,$ that is \textbf{Main Result III}.

\smallskip
\noindent \textbf{Main Result III:} (Theorem \ref{UP for Lp-FMT})
{\it Let $f(x)=\left|f(x)\right|e^{2\pi i\varphi(x)}, xf(x)$ and $w\widehat{f}(w)\in L^{2}(\mathbb{R}^{N})$. If 
$u\mathcal{L}_{M_{1}}[f](u)$ and $ u\mathcal{L}_{M_{2}}[f](u) \in L^{p}(\mathbb{R}^{N})$ with $1\le p\le 2, \frac{1}{p}+\frac{1}{q}=1$, there holds
    \begin{align*}
&\left(\int_{\mathbb{R}^{2}}\left|u\mathcal{L}_{M_{1}}[f](u)\right|^{p}\mathrm{d}u\right)^{\frac{2}{p}}\left(\int_{\mathbb{R}^{N}}\left|u\mathcal{L}_{M_{2}}[f](u)\right|^{p}\mathrm{d}u\right)^{\frac{2}{p}}\notag\\
\ge&\left|\mathrm{det}\left(B_{2}A_{1}^{T}-A_{2}B_{1}^{T}\right)\right|^{\frac{2}{p}-1}\bigg[\frac{\left[\mathrm{tr}\left(A_{1}^{T}B_{2}-A_{2}^{T}B_{1}\right)\right]^{2}}{16\pi^{2}}+\bigg(
\int_{\mathbb{R}^{N}}x^{T}A_{2}^{T}A_{1}
x\left|f(x)\right|^{2}\mathrm{d}x\notag\\
&+ \int_{\mathbb{R}^{N}}w^{T}B_{1}^{T}B_{2}
w\left|\widehat{f}(w)\right|^{2}\mathrm{d}w+ \int_{\mathbb{R}^{N}}x^{T}\left(A_{1}^{T}B_{2}+A_{2}^{T}B_{1}
\right)\nabla \varphi(x)\left|f(x)\right|^2\mathrm{d}x
\bigg)^{2}\bigg].
\end{align*}}
\ 

In \S \ref{section 5} we obtain uncertainty principles for metaplectic operators.

\smallskip
\noindent \textbf{Main Result IV:} (Theorem \ref{corollary 1})
{\it Let $ \int_{\mathbb R^{2N}} {(1+|z|^2)\left|W_\sigma f(z)\right|} dz<\infty$. There holds 
     \begin{align*}
    &\int_{\mathbb{R}^{N}}\left|x\left(\widehat{M_{1}}f\right)(x)\right|^{2}\mathrm{d}x
\int_{\mathbb{R}^{N}}\left|\xi\left(\widehat{M_{2}}f\right)(\xi)\right|^{2}\mathrm{d}\xi\notag\\
&\ge\left[\sum_{j=1}^{N}\left(\frac{1}{16\pi^{2}}\left|\left(M_{1}JM_{2}^{T}\right)_{jj}\right|^{2}+\left|\left(M_{1}\Sigma M_{2}^{T}\right)_{jj}
\right|^{2}
\right)^{\frac{1}{2}}
\right]^{2}.
    \end{align*}}

\noindent \textbf{Main Result V:} (Theorem \ref{Up for MO})
{\it Let $ \int_{\mathbb R^{2N}} {(1+|z|^2)\left|W_\sigma f(z)\right|} dz<\infty$. There holds 
      \begin{align*}
&\int_{\mathbb{R}^{N}}\left|x\left(\widehat{M_{1}}f\right)(x)\right|^{2}\mathrm{d}x
\int_{\mathbb{R}^{N}}\left|\xi\left(\widehat{M_{2}}f\right)(\xi)\right|^{2}\mathrm{d}\xi\notag\\
&\ge\frac{1}{16\pi^2}\left[\sum_{j=1}^{N}\left(M_{1}JM_{2}^{T}\right)_{jj}\right]^{2}+\left[\sum_{j=1}^{N}\left(M_{1}\Sigma M_{2}^{T}\right)_{jj}\right]^{2}.
    \end{align*}}
 
\textbf{Main Result IV} gives a stronger version of Proposition \ref{MPUP}. 
In particular, when $\widehat M_1$ and $\widehat M_2$ are FMTs, \textbf{Main Result IV} and \textbf{Main Result V} coincide with \textbf{Main Result I} and \textbf{Main Result II}, respectively. 

The paper is organized as follows. In \S \ref{section 2}, we introduce the basic properties of symplectic matrices, the metaplectic group and the Weyl operator. In \S \ref{section 3}, we directly prove the main results of this paper. In \S \ref{section 4}, some sharper $L^{p}$-type HPW uncertainty principles with $1\le p\le 2$ are proved. In \S \ref{section 5}, we prove the main results from the point of view of metapletic operators.

\section{Preliminaries}\label{section 2}
\subsection{Symplectic geometry}
Let $\mathbb{R}^{2N}=\mathbb{R}^{N}\oplus \mathbb{R}^{N}$.
A bilinear form on $\mathbb{R}^{2N}$ is called a ``symplectic form'' if it is skew-symmetric and non-degenerate.
The standard symplectic form on $\mathbb{R}^{2N}$ is defined by
\begin{equation*}
    \sigma(z,z^{\prime})=z\cdot J^{-1}z^{\prime}=w \cdot x^{\prime}-x \cdot w^{\prime},
\end{equation*}
where 
\begin{equation*}
    J=\begin{pmatrix}
0 & I_{N}\\ 
-I_{N} & 0
\end{pmatrix}
\end{equation*}
is the standard symplectic matrix, $z=(x,w)$ and $ z^{\prime}=(x^{\prime},w^{\prime})\in \mathbb{R}^{2N}.$ Note that $J^{-1}=J^{T}=-J,$ where $J^{T}$ is the transpose of
$J$.
The space $\mathbb{R}^{2N}$ endowed with the symplectic form $\sigma$ is named the standard symplectic space, which is denoted by $(\mathbb{R}^{2N},\sigma)$. 

The symplectic group $\mathrm{Sp}(2N,\mathbb{R})$ is the set 
of all linear automorphisms $m$ of $\mathbb{R}^{2N}$ such that
\begin{equation}\label{form}
    \sigma(m(z),m(z^{\prime}))=\sigma(z,z')
    \end{equation}
for $z,z'\in \mathbb{R}^{2N}$. 
We refer to the matrix $M$ of $m$ in the canonical basis of $\mathbb{R}^{2N}$ as the symplectic transformation,
\begin{equation*}
    m(z)=Mz.
\end{equation*}
According to (\ref{form}), one has that
\begin{equation}\label{symplectic 1}
  M^{T}JM=J.
\end{equation}
Using (\ref{symplectic 1}), 
we can also have that
\begin{equation*}
     MJM^{T}=J,
\end{equation*}
which means that $M^{T}\in \mathrm{Sp}(2N,\mathbb{R})$. It follows that 
\begin{equation}\label{symplectic 3}
    M\in \mathrm{Sp}(2N,\mathbb{R})\iff M^{T}JM=J \iff MJM^{T}=J.
\end{equation}

If we write a matrix $M\in \mathrm{Sp}(2N,\mathbb{R})$ in block-matrix form
\begin{equation*}
    M=\begin{pmatrix}
A & B\\ 
C & D
\end{pmatrix},
\end{equation*}
where $A,B,C$ and $D$ are real $N\times N$ matrices. Then we have that (\ref{symplectic 3}) is equivalent to the following conditions
\begin{equation}\label{conditions 1}
    A^{T}C=C^{T}A, \quad B^{T}D=D^{T}B,\quad  A^{T}D-C^{T}B=I_{N},
\end{equation}
and
\begin{equation}\label{conditions 2}
   AB^{T}=BA^{T},\quad  CD^{T}=DC^{T},\quad AD^{T}-BC^{T}=I_{N}.
\end{equation}
If the matrix $B$ is invertible, the matrix $M$ is said to be a free symplectic matrix.
To each free symplectic matrix $M_{W}$, it is associated a generating function
\begin{equation*}
    W(x,x')=\frac{1}{2}x^{T}DB^{-1}x-x^{T}B^{-1}x^{\prime}+\frac{1}{2}\left(x^{\prime}\right)^{T}B^{-1}Ax^{\prime},
\end{equation*}
which is a quadratic form. From the second equality in (\ref{conditions 1}) and the first equality in (\ref{conditions 2}), we have that
\begin{equation}\label{conditions 3}
    DB^{-1}=B^{-T}D^{T},
\end{equation}and
\begin{equation}\label{conditions 4}
    B^{-1}A=A^{T}B^{-T}.
\end{equation}
One essential property of free symplectic matrices is that they generate the symplectic group $\mathrm{Sp}(2N,\mathbb{R})$. More precisely,
every $M\in \mathrm{Sp}(2N,\mathbb{R})$ can be represented
as the product $M=M_{W}M_{W'}$, where $M_{W}$ and $M_{W'}$ are two free symplectic matrices.
\subsection{The metaplectic group}
The metaplectic group $\mathrm{Mp}(2N,\mathbb{R})$ is a double cover of the symplectic group. To each $M \in \mathrm{Sp}(2N,\mathbb{R})$, we can associate two unitary operators $\widehat{M}, -\widehat{M}\in \mathrm{Mp}(2N,\mathbb{R})$. The elements of $\mathrm{Mp}(2N,\mathbb{R})$ are known as \lq\lq metaplectic operators”.

Particularly, to every free symplectic matrix $M_{W}$, we can associate two operators, which are given by
\begin{equation}\label{M OPERATOR}
    \widehat{M}_{W,n}f(x)=\frac{i^{n-N/2}}{\sqrt{\left|\mathrm{det}(B)\right|}}\int_{\mathbb{R}^{N}}e^{2\pi i W(x,x^{\prime})}f(x^{\prime})\mathrm{d}x^{\prime},
\end{equation}
for $f\in S(\mathbb{R}^{N})$ (the Schwartz space), where $n=0$ $\mathrm{mod}$ 2 if $\mathrm{det}(B)>0$ and $n=1$ $\mathrm{mod}$ 2 if $\mathrm{det}(B)<0$.

It is well known that these operators can be generalized to unitary operators on $L^{2}(\mathbb{R}^{N})$, and each $\widehat{M}\in \mathrm{Mp}(2N,\mathbb{R})$ can be expressed as a product of $\widehat{M}_{W,n}\widehat{M}_{W^{\prime},n^{\prime}}$ (see Leray \cite{Leray}, de Gosson \cite{deGosson}).
The inverse of the operators $\widehat{M}_{W,n}$ is defined by $\widehat{M}_{W,n}^{-1}=\widehat{M}_{W,n}^{*}=\widehat{M}_{W^{*}, n^{*}}$, where $W^{*}(x,x^{\prime})=-W(x^{\prime},x)$ and $n^{*}=N-n$.

\subsection{Weyl quantization on \texorpdfstring{$(\mathbb{R}^{2N}, \sigma)$}{standard symplectic space}}
In this subsection, we recall some properties of Weyl operator (see \cite{deGosson}).
\begin{definition}
    For $f\in L^{1}(\mathbb{R}^{2N})\bigcap L^{2}(\mathbb{R}^{2N})$, the symplectic Fourier transform is defined by
\begin{equation*}
    ({\mathcal F}_{\sigma}f)(\zeta)=\int_{\mathbb{R}^{2N}}f(z)e^{-2\pi i \sigma(\zeta,z)}\mathrm{d}z.
\end{equation*}
\end{definition}
\noindent
Clearly, the symplectic Fourier transform and Fourier transform are related by the formula,
\begin{equation*}
    ({\mathcal F}_{\sigma}f)(\zeta)=\mathcal Ff(J\zeta),
\end{equation*}
where $\mathcal Ff$ is the Fourier transform of $f.$
\begin{definition}
    Let $a^{\sigma}\in S^{\prime}(\mathbb{R}^{2N})$. The Weyl operator with symbol $a^{\sigma}$ is defined as
\begin{equation*}
\widehat{A}\coloneqq 
 \int_{\mathbb{R}^{2N}}({\mathcal F}_\sigma a^{\sigma})(z_{0}){\widehat{T}}^{\sigma}(z_{0})\mathrm{d}z_{0},
\end{equation*}
where 
\begin{equation*}
({\widehat{T}}^{\sigma}(z_{0})\phi)(x)=e^{2\pi i w_{0}\cdot \left(x-\frac{x_{0}}{2}\right)}\phi(x-x_{0})
\end{equation*}
for $z_{0}=(x_{0},w_{0})\in \mathbb{R}^{2N}$ and $\phi\in S(\mathbb{R}^{N})$.
\end{definition}
\noindent  The correspondence between a symbol $a^{\sigma}\in S^{\prime}(\mathbb{R}^{2N})$ and the Weyl operator it defines is called the Weyl correspondence, which can be written as $\widehat{A}\overset{\text{Weyl}}{\longleftrightarrow }a^{\sigma}$ or $a^{\sigma}\overset{\text{Weyl}}{\longleftrightarrow }\widehat{A}$. 
It is well known that the operator $\widehat{A}$ is formally self-adjoint if and only if the symbol $a^{\sigma}$ is real. 

The fundamental operators in Weyl quantization are given as follows,
\begin{align}\label{fundamental operator}
    \left\{
    \begin{aligned}
        &\left(\widehat{X}_{j}f\right)(x)=x_{j}f(x), \quad  &j=1, \dots,N,\\
        & \left(\widehat{P}_{j}f\right)(x)=\frac{1}{2\pi i}\frac{\partial f(x)}{\partial x_{j}}, \quad  &j=1, \dots,N.
    \end{aligned}\right.
\end{align}
In quantum mechanics $\widehat{X}_{j}$ is explained as
the $j$-th component of the position of a particle and $\widehat{P}_{j}$ is explained as the $j$-th component of its momentum. Let $\widehat{Z}=\left(\widehat{X},\widehat{P}\right)$ with $\widehat{Z}_{j}=\widehat{X}_{j}, \widehat{Z}_{N+j}=\widehat{P}_{j}, j=1, \dots,N$. Then the following commutation relations is satisfied, 
\begin{equation}\label{cov property}
    \left[\widehat{Z}_{\alpha}, \widehat{Z}_{\beta}\right]=\frac{i}{2\pi} J_{\alpha,\beta}\widehat{I}, \quad 1\le \alpha,\beta\le 2N,
\end{equation}
where $\left[\widehat{Z}_{\alpha},\widehat{Z}_{\beta}\right]\triangleq
\widehat{Z}_{\alpha}\widehat{Z}_{\beta}-\widehat{Z}_{\beta}\widehat{Z}_{\alpha}$ and 
$J_{\alpha,\beta}$ are the entries of the standard symplectic matrix $J$.

The Weyl operators have the following symplectic covariance property (see \cite{deGosson}, \cite{Folland}). 
\begin{prop}\label{prop 1}
Let $M\in\mathrm{Sp}(2N,\mathbb{R})$ and $\widehat{M}\in \mathrm{Mp}(2N,\mathbb{R})$ be any of the two metaplectic operators that project onto $M$. For each Weyl operator $\widehat{A}\overset{\text{Weyl}}{\longleftrightarrow }a^{\sigma}$, we have the following correspondence
\begin{equation*}
    a^{\sigma}\circ M \overset{\text{Weyl}} {\longleftrightarrow }{\widehat{M}}^{*}\widehat{A}\widehat{M}.
\end{equation*}
That is, the symbol $a^{\sigma}_{M}(z)=a^{\sigma}(Mz)$ corresponds the Weyl operator ${\widehat{M}}^{*}\widehat{A}\widehat{M}$.
\end{prop}
\noindent Using Proposition \ref{prop 1}, we have that 
\begin{equation*}
\widehat{M}^{*}\widehat{Z}_{\alpha}\widehat{M}=\sum_{\beta=1}^{2N}M_{\alpha, \beta}\widehat{Z}_{\beta}, \quad \alpha=1, \dots,2N,
\end{equation*}
where $M_{\alpha,\beta}$ are the entries of the symplectic matrix $M$. 

The Weyl symbol $a^\sigma$ of the operator $\widehat{A}$ 
and its distributional kernel $K_{\widehat{A}}\in S^{\prime}\left(\mathbb{R}^N \times \mathbb{R}^{N}\right)$ are related by the following formulas, 
\begin{equation}\label{symbol and kernel}
    a^\sigma (x,w)=\int_{\mathbb{R}^{N}}K_{\widehat{A}}\left(x+\frac{y}{2},x-\frac{y}{2}\right)e^{-2\pi iw\cdot y}\mathrm{d}y,
\end{equation}
\begin{equation*}
    K_{\widehat{A}}(x,y)=\int_{\mathbb{R}^{N}}a^\sigma\left(\frac{x+y}{2},w\right)e^{2\pi i w \cdot (x-y)}\mathrm{d}w.
    \end{equation*}
    Let $K_{f,g}(x,y)=\left(f\otimes \overline{g}\right)=f(x)\overline{g(y)}$. By (\ref{symbol and kernel}), the associated Weyl symbol is 
    \begin{equation*}
        W_{\sigma}(f,g)(x,w)=\int_{\mathbb{R}^{N}}f\left(x+\frac{y}{2}\right)\overline{g\left(x-\frac{y}{2}\right)}e^{-2\pi i w \cdot y}\mathrm{d}y,
    \end{equation*}
which is known as the cross-Wigner function. If $f=g$, we simply write the Wigner function $W_\sigma (f,f)$ as $W_{\sigma}f$,
\begin{equation*}
    W_{\sigma}f(x,w)=\int_{\mathbb{R}^{N}}f\left(x+\frac{y}{2}\right)\overline{f\left(x-\frac{y}{2}\right)}e^{-2\pi i w \cdot y}\mathrm{d}y.
\end{equation*}
\subsection{Free metaplectic transformation}

\begin{definition}[\cite{Folland}]\label{MLCT}
For any matrix $M=\begin{pmatrix}
A&B\\ 
C&D
\end{pmatrix}\in \mathrm{Sp}(2N, \mathbb{R})$ with $\mathrm{det}(B)\neq 0$, the free metaplectic transformation (FMT) of a function $f\in L^{1}(\mathbb{R}^{N})$ is defined by
   \begin{align}\label{FMT}
    \mathcal{L}_M[f](u)
    =&\frac{1}{i^{\frac{N}{2}}\sqrt{\mathrm{det}{(B)}}}\int_{\mathbb{R}^{N}}f(x)e^{\pi i (u^{T}DB^{-1}u+x^{T}B^{-1}Ax)-2\pi i x^{T}B^{-1}u}\mathrm{d}x,
\end{align}
where $u=(u_{1}, \dots,u_{N})^{T}$, and the real $N\times N$ matrices $A, B, C$ and $D$ satisfy  (\ref{conditions 1}) and 
 (\ref{conditions 2}). 
If $\mathcal{L}_M[f](u)\in L^{1}(\mathbb{R}^{N})$,
the inverse transform is given by $f(x)=\mathcal{L}_{M^{-1}}[\mathcal{L}_M[f]](x)$, where $M^{-1}=\begin{pmatrix}
D^{T} & -B^{T}\\ 
-C^{T} & A^{T}
\end{pmatrix}$.
\end{definition}
By comparing (\ref{M OPERATOR}) and (\ref{FMT}), we can conclude that when $n$ in (\ref{M OPERATOR}) takes some specific values, we have
\begin{equation*}
    \left(\widehat{M}_{W,n}f\right)(u)= \mathcal{L}_{M_{W}}[f](u).
\end{equation*}
 \noindent From Definition \ref{MLCT}, one has the following relationship between FMT and Fourier transform,
\begin{equation}\label{FMT and FT}
    \mathcal{L}_{M}[f](u)=
    \frac{e^{\pi i u^{T}DB^{-1}u}}{i^{\frac{N}{2}}\sqrt{\mathrm{det}(B)}}
    \left[f(x)e^{\pi i x^{T}B^{-1}Ax}\right]^{\wedge}\left(B^{-1}u\right),
\end{equation}
which plays an important role in the proofs of uncertainty principles for FMT in this paper.

For the free symplectic matrix $M$ taking some special values, FMT becomes some classical transformations, which are given in TABLE \ref{Table 1}.
\begin{table}[H]
	\centering
\caption{Examples for FMT.}\label{Table 1}
 \renewcommand{\arraystretch}{1.2}
	\begin{tabular}{ m{2cm}<{\centering}
 m{2cm}<{\centering}
 m{2cm}<{\centering}
 m{2cm}<{\centering}| m{2.8cm}<{\centering}}
		\hline\hline
  A&   B&   C& D&Transformation\\
  \hline
 0& $I_{N}$& -$I_{N}$&  0& Fourier transform\\
  $\mathrm{diag}(\cos\theta_{1},$&  $\mathrm{diag}(\sin\theta_{1},$  &   $-\mathrm{diag}(\sin\theta_{1},$   & $\mathrm{diag}(\cos\theta_{1},$ &\multirow{2}{*}{FRFT}\\
 $
  \dots,\cos\theta_{N})$ & $
   \dots,\sin\theta_{N})$   &   $
   \dots,\sin\theta_{N})$   & $
   \dots,\cos\theta_{N})$  &\\
 \multirow{2}{*}{$I_{N}$}  & $\mathrm{diag}(b_{11},b_{22}$ & \multirow{2}{*}{0} &   \multirow{2}{*}{$I_{N}$}&\multirow{2}{*}{Fresnel transform}\\
& $\dots,b_{NN})$& & & \\
$\mathrm{diag}(\cosh \theta_{1},$& $\mathrm{diag}(\sinh \theta_{1},$& $\mathrm{diag}(\sinh \theta_{1},$& $\mathrm{diag}(\cosh \theta_{1},$&\multirow{2}{*}{Lorentz transform}\\
$\dots,\cosh \theta_{N})$& $\dots,\sinh \theta_{N})$&$\dots,\sinh \theta_{N})$ &$\dots,\cosh \theta_{N})$ &\\
		\hline\hline
	\end{tabular}
\end{table}
\begin{prop}[e.g. \cite{Chen-Dang-Mai}]\label{additive property}
    For $f\in L^{2}(\mathbb{R}^{N})$, then we have
    \begin{equation*}
         \mathcal{L}_{M_{1}}[ \mathcal{L}_{M_{2}}[f]](u)= \mathcal{L}_{M_{1}M_{2}}[f](u),
    \end{equation*}
where 
$M_{j}=\begin{pmatrix}
    A_{j} & B_{j}\\
    C_{j} & D_{j}
\end{pmatrix} \in \mathrm{Sp}(2N, \mathbb{R}), j=1, 2.$
\end{prop}
\begin{definition}\label{def-jihao}
Let $f(x)=\left|f(x)\right|e^{2\pi i\varphi(x)}, xf(x)$ and $w\widehat{f}(w)\in L^{2}(\mathbb{R}^{N})$. For $j,k=1,\cdots, N$, we define
    \begin{align*}
    (\romannumeral1)\quad & \left<x\right>_f=(\left<x_1\right>_f,\cdots, \left<x_N\right>_f)^T,\text{ where }\left<x_j\right>_f=\int_{\mathbb{R}^{N}}x_j\left|f(x)\right|^2\mathrm{d}x,
    \\
 (\romannumeral2)\quad & \left<w\right>_{\widehat{f}}=(\left<w_1\right>_{\widehat{f}},\cdots, \left<w_N\right>_{\widehat{f}})^T,\text{ where }\left<w_j\right>_{\widehat{f}}=\int_{\mathbb{R}^{N}}w_j\left|\widehat{f}(w)\right|^2\mathrm{d}w,\\  
 (\romannumeral3) \quad & \Delta x^2=\int_{\mathbb{R}^{N}}|(x-\left<x\right>_f )f(x)|^2\mathrm{d}x, \\
 (\romannumeral4)\quad &
 \Delta w^2=\int_{\mathbb{R}^{N}}|(w-\left<w\right>_{\widehat{f}})\widehat{f}(w)|^2\mathrm{d}w,\\
(\romannumeral5) \quad & \mathrm{Cov}_{x,w}=\int_{\mathbb{R}^{N}}(x-\left<x\right>_f)^T (\nabla \varphi(x) -\left<w\right>_{\widehat{f}})\left|f(x)\right|^2\mathrm{d}x,\\
(\romannumeral6) \quad &
\mathrm{COV}_{x,w}=\int_{\mathbb{R}^{N}}\left|(x-\left<x\right>_f)^T\right|\left|\nabla \varphi(x) -\left<w\right>_{\widehat{f}}\right|\left|f(x)\right|^2\mathrm{d}x,\\
         (\romannumeral7) \quad &X=\left(\Delta x_{j,k}^{2}\right), \text{ where }\Delta x_{j,k}^{2}=\int_{\mathbb{R}^{N}}(x_{j}-\left<x_j\right>_f)(x_{k}-\left<x_k\right>_f)|f(x)|^{2}\mathrm{d}x,\\
         (\romannumeral8)\quad &W=\left(\Delta w_{j,k}^{2}\right), \text{ where } \Delta w_{j,k}^{2}=\int_{\mathbb{R}^{N}}(w_{j}-\left<w_j\right>_{\widehat{f}})(w_{k}-\left<w_k\right>_{\widehat{f}})|\widehat{f}(w)|^{2}\mathrm{d}w,\\
         (\romannumeral9) \quad &\mathrm{Cov}_{X,W}=\left(\mathrm{Cov}_{x,w}^{j,k}\right), \\
         \quad & \text{ where } \mathrm{Cov}_{x,w}^{j,k}=\int_{\mathbb{R}^{N}}(x_{j}-\left<x_j\right>_f)\left(\frac{\partial \varphi(x)}{\partial x_{k}}-\left<w_k\right>_{\widehat{f}}\right)\left|f(x)\right|^{2}\mathrm{d}x,\\
(\romannumeral10) \quad&\mathrm{COV}_{x,w}^{j,k}=\int_{\mathbb{R}^{N}}\left|(x_{j}-\left<x_j\right>_f) \left(\frac{\partial \varphi(x)}{\partial x_{k}}-\left<w_k\right>_{\widehat{f}}\right)\right|\left|f(x)\right|^2\mathrm{d}x.
    \end{align*}
\end{definition}
Without loss of generality, in this paper we always assume $\left<x_j\right>_f=0$ and $\left<w_j\right>_{\widehat{f}}=0, j=1,\cdots, N$. In \S \ref{section 5}, our discussion is based on the condition 
\begin{equation}\label{condition-sigma}
    \int_{\mathbb R^{2N}} {(1+|z|^2)\left|W_\sigma f(z)\right|} dz<\infty.
\end{equation}
One can easily have that 
if (\ref{condition-sigma}) holds, then $\Sigma_{\alpha,\beta}<\infty$ for $\alpha,\beta=1,\cdots,2N$ (see equation (\ref{SIGMA}) for its definition). In fact, we have
\begin{equation*}
|\Sigma_{\alpha,\beta}|=\left|\int_{\mathbb{R}^{2N}}z_{\alpha}z_{\beta}W_{\sigma}f(z)dz\right|
\le  \int_{\mathbb R^{2N}} {(1+|z|^2)\left|W_\sigma f(z)\right|} dz<\infty.
\end{equation*}
In the following we can show that the condition (\ref{condition-sigma}) is consistent with assumptions in Definition \ref{def-jihao}.
Since
$\int_{\mathbb{R}^{N}}W_{\sigma}f(x,w)\mathrm{d}w=|f(x)|^2$ and  $\int_{\mathbb{R}^{N}}W_{\sigma}f(x,w)\mathrm{d}x=|\widehat{f}(w)|^2,$
we have 
\begin{align*}
    \int_{\mathbb R^{2N}} (1+|z|^2)W_\sigma f(z) dz=& \iint_{\mathbb R^{2N}}(1+|x|^2+|w|^2)W_\sigma f(x,w)\mathrm{d}x\mathrm{d}w\notag\\
=&\Vert f\Vert_2^2+\Delta x^2+\Delta w^2.
\end{align*}
Hence we have that $\int_{\mathbb R^{2N}} (1+|z|^2)\left|W_\sigma f(z)\right| dz<\infty$ if and only if $f, xf(x)$ and $w\widehat{f}(w)\in L^2(\mathbb{R}^{N})$. In this paper, we assume  
$ <\widehat{Z}_{\alpha}>_f=\int_{\mathbb{R}^{N}}\overline{f(x)}(\widehat{Z}_{\alpha}f)(x)\mathrm{d}x=0.$
We have that $<\widehat{Z}_{\alpha}>_f=0$ if and only if $
\left<x_{\alpha}\right>_f=0,\alpha=1,\cdots, N$ and $\left<w_{\alpha-N}\right>_{\widehat{f}}=0,\alpha=N+1,\cdots,2N.$

\begin{prop}\label{prop-LM}
Let $f(x)=\left|f(x)\right|e^{2\pi i\varphi(x)}, xf(x)$ and $w\widehat{f}(w)\in L^{2}(\mathbb{R}^{N})$. Then there holds
    \begin{align}\label{prop-e-LM}
        \int_{\mathbb{R}^{N}}\left|u\mathcal{L}_{M}[f](u)\right|^2\mathrm{d}u=&\int_{\mathbb{R}^{N}}x^T A^TA x\left|f(x)\right|^2\mathrm{d}x+\int_{\mathbb{R}^{N}}w^T B^TB w\left|\widehat{f}(w)\right|^2\mathrm{d}w\notag\\
        &+2\int_{\mathbb{R}^N}x^TA^T B\nabla \varphi(x)\left|f(x)\right|^2\mathrm{d}x.
    \end{align}
\begin{proof}
    Let $g(x)=f(x)e^{\pi i x^{T}B^{-1}Ax}$. Using (\ref{FMT and FT}), we have
\begin{equation*}
    \mathcal{L}_{M}[f](u)=\frac{e^{\pi i u^{T}DB^{-1}u}}{i^{\frac{N}{2}}\sqrt{\mathrm{det}(B)}}\widehat{g}\left(B^{-1}u\right).
\end{equation*}
By $B^{-1}u=w$, Parseval's identity and 
\begin{equation}\label{nabla g}
    \nabla g(x)=\nabla f(x)e^{\pi i x^{T}B^{-1}Ax}+2\pi i B^{-1}Axf(x)e^{\pi i x^{T}B^{-1}Ax},
\end{equation}
one has
\begin{align}\label{eq-LM}
     \int_{\mathbb{R}^{N}}\left|u\mathcal{L}_{M}[f](u)\right|^2\mathrm{d}u=&\frac{1}{\left|\mathrm{det}(B)\right|}\int_{\mathbb{R}^{N}}\left|u\widehat{g}\left(B^{-1}u\right)\right|^2\mathrm{d}u\notag\\
     =& \int_{\mathbb{R}^{N}}\left|Bw\widehat{g}(w)\right|^2\mathrm{d}w\notag\\
    =&\frac{1}{4\pi^2}\int_{\mathbb{R}^{N}}
    \left|B \nabla g(x)\right|^2
    \mathrm{d}x\notag\\
    =&\frac{1}{4\pi^2}\int_{\mathbb{R}^{N}}\left|B\nabla f(x)+2\pi i Axf(x)\right|^2\mathrm{d}x\notag\\
    =&\frac{1}{4\pi^2}\int_{\mathbb{R}^{N}}\left(\nabla f(x)\right)^TB^TB \overline{\nabla f(x)}\mathrm{d}x+\int_{\mathbb{R}^{N}}x^TA^TAx\left|f(x)\right|^2\mathrm{d}x\notag\\
&+\frac{1}{2\pi i}\int_{\mathbb{R}^{N}}x^TA^TB\left(\nabla f(x)\overline{f(x)}-\overline{\nabla f(x)}f(x)\right)\mathrm{d}x.
\end{align}
Since
\begin{equation}\label{nabla f}
    \nabla f(x)=\nabla \left|f(x)\right| e^{2\pi i \varphi(x)}+2\pi i \nabla \varphi(x)\left|f(x)\right|e^{2\pi i \varphi(x)},
\end{equation}
we have
\begin{align*}
    \frac{1}{2\pi i}\int_{\mathbb{R}^{N}}x^TA^TB\left(\nabla f(x)\overline{f(x)}-\overline{\nabla f(x)}f(x)\right)\mathrm{d}x
    =2\int_{\mathbb{R}^N}x^TA^T B\nabla \varphi(x)\left|f(x)\right|^2\mathrm{d}x.
\end{align*}
Note that
\begin{align*}
    \frac{1}{4\pi^2}\int_{\mathbb{R}^{N}}\left(\nabla f(x)\right)^TB^TB \overline{\nabla f(x)}\mathrm{d}x=\int_{\mathbb{R}^{N}}w^T B^TB w\left|\widehat{f}(w)\right|^2\mathrm{d}w.
\end{align*}
Hence, we have (\ref{prop-e-LM}).
\end{proof}
\end{prop}
\section{HPW uncertainty principles for free metaplectic transformation}\label{section 3}
In this section, we establish two uncertainty principles in two FMT domains, and one uncertainty principle in one time and one FMT domains. The first uncertainty principle in two FMT domains obtained is given as follows.
\begin{theorem}\label{Up for MO component}
  Let $f(x)=\left|f(x)\right|e^{2\pi i\varphi(x)}, xf(x)$ and $w\widehat{f}(w)\in L^{2}(\mathbb{R}^{N})$. Then there holds
    \begin{align}\label{Up for FMT2}
&\int_{\mathbb{R}^{N}}\left|u\mathcal{L}_{M_{1}}[f](u)\right|^{2}\mathrm{d}u
\int_{\mathbb{R}^{N}}\left|u\mathcal{L}_{M_{2}}[f](u)\right|^{2}\mathrm{d}u\notag\\
&\ge\biggr[\sum_{j=1}^{N}\biggr(\frac{1}{16\pi^{2}}\left|\left(A_{1}B_{2}^{T}-B_{1}A_{2}^{T}\right)_{jj}\right|^{2}+\Bigr|\Bigr(A_{1}XA_{2}^{T}+B_{1}WB_{2}^{T}\notag\\
&\quad+A_{1}\mathrm{Cov}_{X,W}B_{2}^{T}+B_{1}\left(\mathrm{Cov}_{X,W}\right)^{T}A_{2}^{T}\Big)_{jj}
\Big|^{2}
\bigg)^{\frac{1}{2}}
\bigg]^{2}.
\end{align}
\end{theorem}
\begin{proof}
According to Cauchy-Schwartz’s inequality, we have
\begin{align}\label{Up for second FMT}
&\int_{\mathbb{R}^{N}}\left|u\mathcal{L}_{M_{1}}[f](u)\right|^{2}\mathrm{d}u
\int_{\mathbb{R}^{N}}\left|u\mathcal{L}_{M_{2}}[f](u)\right|^{2}\mathrm{d}u\notag\\
=&\sum_{j=1}^{N}\int_{\mathbb{R}^{N}}\left|u_j\mathcal{L}_{M_{1}}[f](u)\right|^{2}\mathrm{d}u\sum_{j=1}^{N}\int_{\mathbb{R}^{N}}\left|u_j\mathcal{L}_{M_{2}}[f](u)\right|^{2}\mathrm{d}u\notag\\
\ge&\left[\sum_{j=1}^{N}\left(\int_{\mathbb{R}^{N}}\left|u_j\mathcal{L}_{M_{1}}[f](u)\right|^{2}\mathrm{d}u\int_{\mathbb{R}^{N}}\left|u_j\mathcal{L}_{M_{2}}[f](u)\right|^{2}\mathrm{d}u\right)^{\frac{1}{2}}\right]^2.
\end{align}
Let $g(x)=f(x)e^{\pi i x^TB_1^{-1}A_1x}$. 
    From (\ref{FMT and FT}), we have
    \begin{equation*}
     \mathcal{L}_{M_{1}}[f](u)=\frac{e^{\pi i u^{T}D_{1}B_{1}^{-1}u}}{i^{\frac{N}{2}}\sqrt{\mathrm{det}(B_{1})}}\widehat{g}\left(B_{1}^{-1}u\right).
    \end{equation*}
For each $j=1,\cdots, N$, one has that
\begin{equation*}
\int_{\mathbb{R}^{N}}\left|u_j\mathcal{L}_{M_{1}}[f](u)\right|^{2}\mathrm{d}u\notag\\
=\frac{1}{\left|\mathrm{det}\left(B_1\right)\right|}\int_{\mathbb{R}^{N}}\left|u_j \widehat{g}\left(B_1^{-1}u\right)
\right|^{2}\mathrm{d}u.
\end{equation*}
Let $u=B_1w$. Denote by $\left(B_1\right)_{jk}$ the $(j,k)$-th element of $B_1$, which means that
$u_j=\sum_{k=1}^{N}\left(B_1\right)_{jk}w_k$. 
Since $A_1$ and $B_1$ satisfy (\ref{conditions 4}), we have 
\begin{equation*}
    \frac{\partial g(x)}{\partial x_k}=\frac{\partial f(x)}{\partial x_k}e^{\pi ix^TB_1^{-1}A_1x}+2\pi i\sum_{m=1}^{N}\left(B_1^{-1}A_1\right)_{km}x_m f(x)e^{\pi ix^TB_1^{-1}A_1x}.
\end{equation*}
Using Parseval's identity and
\begin{equation*}
    \frac{\partial f(x)}{\partial x_k}=\frac{\partial \left|f(x)\right|}{\partial x_k}e^{2\pi i\varphi(x)}+2\pi i\frac{\partial \varphi(x)}{\partial x_k}\left|f(x)\right|e^{2\pi i\varphi(x)},
\end{equation*}
we have
\begin{align}\label{M_1}
&\int_{\mathbb{R}^{N}}\left|u_j\mathcal{L}_{M_{1}}[f](u)\right|^{2}\mathrm{d}u\notag\\
=&\int_{\mathbb{R}^{N}}\left|\sum_{k=1}^{N}\left(B_1\right)_{jk}w_k\widehat{g}(w)\right|^2\mathrm{d}w\notag\\
=&\int_{\mathbb{R}^{N}}\left|\frac{1}{2\pi i}\sum_{k=1}^{N}\left(B_1\right)_{jk} \frac{\partial g(x)}{\partial x_k}
\right|^2\mathrm{d}x\notag\\
=&\int_{\mathbb{R}^{N}}\left|\frac{1}{2\pi i}\sum_{k=1}^{N}\left(B_1\right)_{jk} \frac{\partial f(x)}{\partial x_k}
+\sum_{k=1}^{N}\left(A_1\right)_{jk}x_k f(x)\right|^2
\mathrm{d}x\notag\\
=&\int_{\mathbb{R}^{N}}\left|\frac{1}{2\pi i}\sum_{k=1}^{N}\left(B_1\right)_{jk} \frac{\partial \left|f(x)\right|}{\partial x_k}+\sum_{k=1}^{N}\left(B_1\right)_{jk} \frac{\partial \varphi(x)}{\partial x_k}\left|f(x)\right|
+\sum_{k=1}^{N}\left(A_1\right)_{jk}x_k \left|f(x)\right|\right|^2
\mathrm{d}x.\notag\\
\end{align}
Similarly, we have
\begin{align}\label{M_2}
&\int_{\mathbb{R}^{N}}\left|u_j\mathcal{L}_{M_{2}}[f](u)\right|^{2}\mathrm{d}u\notag\\
=&\int_{\mathbb{R}^{N}}\left|\frac{1}{2\pi i}\sum_{k=1}^{N}\left(B_2\right)_{jk} \frac{\partial \left|f(x)\right|}{\partial x_k}+\sum_{k=1}^{N}\left(B_2\right)_{jk} \frac{\partial \varphi(x)}{\partial x_k}\left|f(x)\right|
+\sum_{k=1}^{N}\left(A_2\right)_{jk}x_k \left|f(x)\right|\right|^2
\mathrm{d}x.\notag\\
\end{align}
Using Cauchy-Schwartz’s inequality, one has that
\begin{align*}
&\int_{\mathbb{R}^{N}}\left|u_j\mathcal{L}_{M_{1}}[f](u)\right|^{2}\mathrm{d}u\int_{\mathbb{R}^{N}}\left|u_j\mathcal{L}_{M_{2}}[f](u)\right|^{2}\mathrm{d}u\notag\\
\ge&\Biggr|
\int_{\mathbb{R}^{N}}\left(\frac{1}{2\pi i}\sum_{k=1}^{N}\left(B_1\right)_{jk} \frac{\partial \left|f(x)\right|}{\partial x_k}+\sum_{k=1}^{N}\left(B_1\right)_{jk} \frac{\partial \varphi(x)}{\partial x_k}\left|f(x)\right|
+\sum_{k=1}^{N}\left(A_1\right)_{jk}x_k \left|f(x)\right|\right)\notag\\
&\times\left(-\frac{1}{2\pi i}\sum_{l=1}^{N}\left(B_2\right)_{jl} \frac{\partial \left|f(x)\right|}{\partial x_l}+\sum_{l=1}^{N}\left(B_2\right)_{jl} \frac{\partial \varphi(x)}{\partial x_l}\left|f(x)\right|
+\sum_{l=1}^{N}\left(A_2\right)_{jl}x_l \left|f(x)\right|
\right)\mathrm{d}x
\Bigg|^2\notag\\
=&\left|\frac{1}{2\pi i}I_1+I_2+I_3\right|^2,
\end{align*}
where
\begin{align*}
    I_1=&-\sum_{k=1}^{N}\left(A_1\right)_{jk}\sum_{l=1}^{N}\left(B_2\right)_{jl}\int_{\mathbb{R}^{N}}x_k\frac{\partial \left|f(x)\right|}{\partial x_l}\left|f(x)\right|\mathrm{d}x\notag\\
    &+\sum_{k=1}^{N}\left(B_1\right)_{jk}\sum_{l=1}^{N}\left(A_2\right)_{jl}\int_{\mathbb{R}^{N}}x_l\frac{\partial \left|f(x)\right|}{\partial x_k}\left|f(x)\right|\mathrm{d}x,
\end{align*}
\begin{align*}
I_2=&\sum_{k=1}^{N}\left(A_1\right)_{jk}\sum_{l=1}^{N}\left(A_2\right)_{jl}\Delta x_{k,l}^2+\sum_{k=1}^{N}\left(A_1\right)_{jk}\sum_{l=1}^{N}\left(B_2\right)_{jl}\mathrm{Cov}_{x.w}^{k,l}
\notag\\
&+\sum_{k=1}^{N}\left(B_1\right)_{jk}\sum_{l=1}^{N}\left(A_2\right)_{jl}\mathrm{Cov}_{x.w}^{l,k}
\end{align*}
and 
\begin{align*}
    I_3=&\frac{1}{4\pi^2}\sum_{k=1}^{N}\left(B_1\right)_{jk}\sum_{l=1}^{N}\left(B_2\right)_{jl}\int_{\mathbb{R}^{N}}\frac{\partial \left|f(x)\right|}{\partial x_k}\frac{\partial \left|f(x)\right|}{\partial x_l}\mathrm{d}x\notag\\
&+\frac{1}{2\pi i}\sum_{k=1}^{N}\left(B_1\right)_{jk}\sum_{l=1}^{N}\left(B_2\right)_{jl}\int_{\mathbb{R}^{N}}\frac{\partial \left|f(x)\right|}{\partial x_k}\frac{\partial \varphi(x)}{\partial x_l}\left|f(x)\right|\mathrm{d}x\notag\\
&-\frac{1}{2\pi i}\sum_{k=1}^{N}\left(B_1\right)_{jk}\sum_{l=1}^{N}\left(B_2\right)_{jl}\int_{\mathbb{R}^{N}}\frac{\partial \left|f(x)\right|}{\partial x_l}\frac{\partial \varphi(x)}{\partial x_k}\left|f(x)\right|\mathrm{d}x\notag\\
&+\sum_{k=1}^{N}\left(B_1\right)_{jk}\sum_{l=1}^{N}\left(B_2\right)_{jl}\int_{\mathbb{R}^{N}}\frac{\partial \varphi(x)}{\partial x_k}\frac{\partial \varphi(x)}{\partial x_l}\left|f(x)\right|^2\mathrm{d}x.
\end{align*}
A direct computation yields that 
    \begin{equation*}
        I_1=\frac{1}{2}\left(A_1B_2^T-B_1 A_2^T\right)_{jj},
    \end{equation*}
\begin{equation*}
    I_2=\left(A_1 XA_2^T+A_1\mathrm{Cov}_{X,W}B_2^T+B_1\left(\mathrm{Cov}_{X,W}\right)^{T}A_2^T\right)_{jj}
\end{equation*}
and 
\begin{equation*}
    I_3=\left(B_1 W B_2^T\right)_{jj}.
\end{equation*}
Then we have
\begin{align}\label{component FMT}
    &\int_{\mathbb{R}^{N}}\left|u_j\mathcal{L}_{M_{1}}[f](u)\right|^{2}\mathrm{d}u\int_{\mathbb{R}^{N}}\left|u_j\mathcal{L}_{M_{2}}[f](u)\right|^{2}\mathrm{d}u\notag\\
    \ge&\frac{1}{4\pi^2}\left|I_1\right|^2+\left|I_2+I_3\right|^2\notag\\
    =&\frac{1}{16\pi^2}\left|\left(A_1B_2^T-B_1 A_2^T\right)_{jj}\right|^2+\Bigr|
    \Bigr(A_1 XA_2^T+B_1 W B_2^T+A_1\mathrm{Cov}_{X,W}B_2^T\notag\\
&+B_1\left(\mathrm{Cov}_{X,W}\right)^TA_2^T\Big)_{jj}
    \Big|^2.
    \end{align}
    By substituting (\ref{component FMT}) into (\ref{Up for second FMT}), one has the desired inequality (\ref{Up for FMT2}). 
\end{proof}
\begin{remark}
    When $N=1,$ 
    (\ref{Up for FMT2}) reduces to the following inequality,
    \begin{align}\label{one D}
&\int_{\mathbb{R}}\left|u\mathcal{L}_{M_1}[f](u)\right|^{2}\mathrm{d}u\int_{\mathbb{R}}\left|u\mathcal{L}_{M_2}[f](u)\right|^{2}\mathrm{d}u\notag\\
&\ge\frac{\left(a_{1}b_{2}-a_{2}b_{1}\right)^{2}}{16\pi^{2}}+\biggr[a_{1}a_{2}\Delta x^{2}+b_{1}b_{2}\Delta w^{2}+\left(a_{1}b_{2}+a_{2}b_{1}\right)\mathrm{Cov}_{x,w}
\bigg]^{2},
\end{align}
which gives the result of \cite{Xu-Wang-Xu3,Zhao-Tao-Li-Wang}, where $M_{k}=\begin{pmatrix}
a_{k} & b_{k}\\ 
c_{k} & d_{k}
\end{pmatrix}$ for $k=1,2$.
\end{remark}
To obtain the second uncertainty principle in two FMT domains, we need the following technical lemmas.
\begin{lemma}\label{the first lemma}
Let $f(x)=\left|f(x)\right|e^{2\pi i\varphi(x)}, xf(x)$ and $w\widehat{f}(w)\in L^{2}(\mathbb{R}^{N})$. For real $N\times N$
matrices $A_{2}$ and $B_{2}$, there holds
\begin{align}\label{first lemma}
&i\int_{\mathbb{R}^{N}}u^{T}\mathcal{L}_{M_{1}}[f](u)\left(B_{2}A_{1}^{T}-A_{2}B_{1}^{T}\right)\overline{\nabla \mathcal{L}_{M_{1}}[f](u)}\mathrm{d}u\notag\\
=&-\frac{i}{2}\mathrm{tr}\left(A_{1}^{T}B_{2}+A_{1}^{T}B_{2}C_{1}^{T}B_{1}-A_{1}^{T}A_{2}B_{1}^{T}D_{1}-A_{1}^{T}C_{1}B_{2}^{T}B_{1}+C_{1}^{T}B_{1}A_{2}^{T}B_{1}\right)\notag\\
&+2\pi\int_{\mathbb{R}^{N}}w^{T}\left(B_{1}^{T}B_{2}+B_{1}^{T}B_{2}C_{1}^{T}B_{1}-B_{1}^{T}A_{2}B_{1}^{T}D_{1}\right)w\left|\widehat{f}(w)\right|^{2}\mathrm{d}w\notag\\
&+2\pi\int_{\mathbb{R}^{N}}x^{T}\Big(A_{1}^{T}B_{2}B_{1}^{-1}A_{1}+A_{1}^{T}B_{2}C_{1}^{T}A_{1}-A_{1}^{T}A_{2}D_{1}^{T}A_{1}-B_{1}^{-1}A_{1}B_{2}^{T}A_{1}\notag\\
&+A_{2}^{T}A_{1}\Big)x\left|f(x)\right|^{2}\mathrm{d}x+2\pi\int_{\mathbb{R}^{N}}x^{T}\Big(A_{1}^{T}B_{2}+A_{1}^{T}B_{2}C_{1}^{T}B_{1}-A_{1}^{T}A_{2}B_{1}^{T}D_{1}\notag\\
&+A_{1}^{T}C_{1}B_{2}^{T}B_{1}-C_{1}^{T}B_{1}A_{2}^{T}B_{1}\Big)\nabla \varphi(x)\left|f(x)\right|^{2}\mathrm{d}x.
\end{align}
\end{lemma}
\begin{proof}
    The proof is given in Appendix \ref{appendix}.
\end{proof}

\begin{lemma}\label{the second lemma}
Let $f(x)=\left|f(x)\right|e^{2\pi i\varphi(x)}, xf(x)$ and $w\widehat{f}(w)\in L^{2}(\mathbb{R}^{N})$. For real $N\times N$
matrices $A_{2}$ and $B_{2}$, there holds
\begin{align}\label{the third lemma}
&2\pi\int_{\mathbb{R}^{N}}u^{T}\left(A_{2}D_{1}^{T}-B_{2}C_{1}^{T}\right)u\left|\mathcal{L}_{M_{1}}[f](u)\right|^{2}\mathrm{d}u\notag\\
    =&-\frac{i}{2}\mathrm{tr}\left(A_{1}^{T}A_{2}D_{1}^{T}B_{1}-A_{1}^{T}B_{2}C_{1}^{T}B_{1}-A_{2}^{T}B_{1}-C_{1}^{T}B_{1}A_{2}^{T}B_{1}+A_{1}^{T}C_{1}B_{2}^{T}B_{1}\right)\notag\\
    &+2\pi\int_{\mathbb{R}^{N}}x^{T}\left(A_{1}^{T}A_{2}D_{1}^{T}A_{1}-A_{1}^{T}B_{2}C_{1}^{T}A_{1}\right)x\left|f(x)\right|^{2}\mathrm{d}x\notag\\
    &+2\pi\int_{\mathbb{R}^{N}} w^{T}\left(B_{1}^{T}A_{2}D_{1}^{T}B_{1}-B_{1}^{T}B_{2}C_{1}^{T}B_{1}\right)w\left|\widehat{f}(w)\right|^{2}\mathrm{d}w\notag\\
    &+2\pi\int_{\mathbb{R}^{N}}x^{T}\Big(
    A_{1}^{T}A_{2}D_{1}^{T}B_{1}-A_{1}^{T}B_{2}C_{1}^{T}B_{1}+A_{2}^{T}B_{1}
   +C_{1}^{T}B_{1}A_{2}^{T}B_{1}\notag\\
   &-A_{1}^{T}C_{1}B_{2}^{T}B_{1}\Big)\nabla \varphi(x)\left|f(x)\right|^2\mathrm{d}x.
\end{align}
\end{lemma}
\begin{proof}
    The proof is given in Appendix \ref{appendix}.
\end{proof}
Using Lemmas \ref{the first lemma} and \ref{the second lemma}, we have the following main result.
\begin{theorem}\label{UP for L2 M1 M2}
Let $f(x)=\left|f(x)\right|e^{2\pi i\varphi(x)}, xf(x)$ and $w\widehat{f}(w)\in L^{2}(\mathbb{R}^{N})$. Then there holds
    \begin{align}\label{UP for LCT}
&\int_{\mathbb{R}^{N}}\left|u\mathcal{L}_{M_{1}}[f](u)\right|^{2}\mathrm{d}u\int_{\mathbb{R}^{N}}\left|u\mathcal{L}_{M_{2}}[f](u)\right|^{2}\mathrm{d}u\notag\\
&\ge\frac{\left[\mathrm{tr}\left(A_{1}^{T}B_{2}-A_{2}^{T}B_{1}\right)\right]^{2}}{16\pi^{2}}+\biggr[\int_{\mathbb{R}^{N}}x^{T}A_{2}^{T}A_{1}
x\left|f(x)\right|^{2}\mathrm{d}x\notag\\
&\quad+ \int_{\mathbb{R}^{N}}w^{T}B_{1}^{T}B_{2}
w\left|\widehat{f}(w)\right|^{2}\mathrm{d}w+ \int_{\mathbb{R}^{N}}x^{T}\left(A_{1}^{T}B_{2}+A_{2}^{T}B_{1}
\right)\nabla \varphi(x)\left|f(x)\right|^{2}\mathrm{d}x
\bigg]^{2}.\notag\\
\end{align}
\end{theorem}
\begin{proof}
Let $M_{3}=M_{2}M_{1}^{-1}$. Since
$M_1^{-1}=\begin{pmatrix}
D_1^{T} & -B_1^{T}\\ 
-C_1^{T} & A_1^{T}
\end{pmatrix}$, then we have
\begin{align*}
    M_{3}=\begin{pmatrix}
A_{3} & B_{3}\\ 
C_{3} & D_{3}
\end{pmatrix}=\begin{pmatrix}
A_{2}D_{1}^{T}-B_{2}C_{1}^{T} & B_{2}A_{1}^{T}-A_{2}B_{1}^{T}\\ 
C_{2}D_{1}^{T}-D_{2}C_{1}^{T} & D_{2}A_{1}^{T}-C_{2}B_{1}^{T}
\end{pmatrix}.
\end{align*}
By Proposition \ref{additive property}, we have
\begin{equation}\label{the additive property}
    \mathcal{L}_{M_{2}}[f](u)=\mathcal{L}_{M_{3}}[\mathcal{L}_{M_{1}}[f]](u).
\end{equation}
Let $H(x)=\mathcal{L}_{M_{1}}[f](x)e^{\pi i x^{T}B_{3}^{-1}A_{3}x}$. Using (\ref{FMT and FT}), we have that
\begin{equation}\label{FMT3 and ft}
    \mathcal{L}_{M_{3}}[\mathcal{L}_{M_{1}}[f]](u)=\frac{e^{\pi i u^{T}D_{3}B_{3}^{-1}u}}{i^{\frac{N}{2}}\sqrt{\mathrm{det}(B_{3})}}\widehat{H}\left(B_{3}^{-1}u\right).
\end{equation}
By $B_{3}^{-1}u=w$ and Parseval's identity,
we have
\begin{align*}
\int_{\mathbb{R}^{N}}\left|u\mathcal{L}_{M_{2}}[f](u)\right|^{2}\mathrm{d}u
=&\int_{\mathbb{R}^{N}}\left|u\mathcal{L}_{M_{3}}[\mathcal{L}_{M_{1}}[f]](u)\right|^{2}\mathrm{d}u\\
=&\frac{1}{\left|\mathrm{det}\left(B_{3}\right)\right|}\int_{\mathbb{R}^{N}}\left|u\widehat{H}\left(B_{3}^{-1}u\right)\right|^{2}\mathrm{d}u
\notag\\
=&\int_{\mathbb{R}^{N}}\left|B_{3}w\widehat{H}(w)\right|^2\mathrm{d}w
\notag\\
=&\frac{1}{4\pi^{2}}\int_{\mathbb{R}^{N}}\left|
B_{3}\nabla H(u)\right|^{2}\mathrm{d}u.
\end{align*}
Since $A_3$ and $B_3$ satisfy (\ref{conditions 4}), we have
\begin{equation}\label{nabla H}
    \nabla H(u)=\nabla \mathcal{L}_{M_{1}}[f](u)e^{\pi i u^{T}B_{3}^{-1}A_{3}u}+2\pi i B_{3}^{-1}A_{3}ue^{\pi i u^{T}B_{3}^{-1}A_{3}u} \mathcal{L}_{M_{1}}[f](u).
\end{equation}
Applying Cauchy-Schwartz’s inequality, we have
\begin{align*}
&\int_{\mathbb{R}^{N}}\left|u\mathcal{L}_{M_{1}}[f](u)\right|^{2}\mathrm{d}u\int_{\mathbb{R}^{N}}\left|u\mathcal{L}_{M_{2}}[f](u)\right|^{2}\mathrm{d}u\notag\\
    =&\frac{1}{4\pi^{2}}\int_{\mathbb{R}^{N}}\left|u\mathcal{L}_{M_{1}}[f](u)\right|^{2}\mathrm{d}u\int_{\mathbb{R}^{N}}\left|
B_{3}\nabla H(u)\right|^{2}\mathrm{d}u\notag\\
\ge& \frac{1}{4\pi^{2}}\left|\int_{\mathbb{R}^{N}}ie^{\pi i u^{T}B_{3}^{-1}A_{3}u}u^{T}\mathcal{L}_{M_{1}}[f](u)
B_{3}\overline{\nabla H(u)}\mathrm{d}u\right|^{2}\notag\\
=&\frac{1}{4\pi^{2}}\left|I_{1}+I_{2}
\right|^{2},
\end{align*}
where 
\begin{equation}\label{J_{1}}
I_{1}=i\int_{\mathbb{R}^{N}}u^{T}\mathcal{L}_{M_{1}}[f](u)B_{3}\overline{\nabla \mathcal{L}_{M_{1}}[f](u)}\mathrm{d}u,
\end{equation}
and
\begin{equation}\label{J_{2}}
I_{2}=2\pi \int_{\mathbb{R}^{N}}u^{T}A_{3}u\left|\mathcal{L}_{M_{1}}[f](u)\right|^{2}\mathrm{d}u.
\end{equation}
Using Lemmas \ref{the first lemma} and \ref{the second lemma}, one has the desired inequality (\ref{UP for LCT}). 
\end{proof}
\begin{remark}
When $N=1$ and $
        M_{k}=\begin{pmatrix}
a_{k} & b_{k}\\ 
c_{k} & d_{k}
\end{pmatrix}
    $
    for $k=1,2$, we have that (\ref{UP for LCT}) also becomes (\ref{one D}).
\end{remark}
\begin{remark}
When $A_{k}, B_{k}, k=1,2$ take some special values in (\ref{UP for LCT}), we can obtain better forms of HPW uncertainty principle in two FMT domains. Note that in \cite{Zhang1, Zhang2, Zhang3}, the author obtains some versions of HPW uncertainty principle in two FMT domains.
In general, we cannot compare the lower bounds of (\ref{UP for LCT}) with those in \cite{Zhang1, Zhang2, Zhang3} in the following cases.

\noindent (\romannumeral1)
Let $A_{k}=\mathrm{diag}(a_{11}^{(k)},\dots,a_{NN}^{(k)}),
B_{k}=\mathrm{diag}(b_{11}^{(k)},\dots,b_{NN}^{(k)}), k=1,2$. Then (\ref{UP for LCT}) becomes the following inequality,
\begin{align*}
&\int_{\mathbb{R}^{N}}\left|u\mathcal{L}_{M_{1}}[f](u)\right|^{2}\mathrm{d}u\int_{\mathbb{R}^{N}}\left|u\mathcal{L}_{M_{2}}[f](u)\right|^{2}\mathrm{d}u\notag\\
&\ge\frac{\left[\sum_{j=1}^{N}\left(a_{jj}^{(1)}b_{jj}^{(2)}-a_{jj}^{(2)}b_{jj}^{(1)}\right)\right]^{2}}{16\pi^{2}}+\biggr[\sum_{j=1}^{N}a_{jj}^{(1)}a_{jj}^{(2)}\Delta x_{j,j}^{2}+\sum_{j=1}^{N}b_{jj}^{(1)}b_{jj}^{(2)}\Delta w_{j,j}^{2}\notag\\
&\quad +\sum_{j=1}^{N}\left(a_{jj}^{(1)}b_{jj}^{(2)}+a_{jj}^{(2)}b_{jj}^{(1)}\right)\mathrm{Cov}_{x,w}^{j,j}
\bigg]^{2}.
\end{align*} 
In particular, when $A_{k}=a_{k}I_{N}, B_{k}=b_{k}I_{N}, k=1,2$, we have
\begin{align*}
&\int_{\mathbb{R}^{N}}\left|u\mathcal{L}_{M_{1}}[f](u)\right|^{2}\mathrm{d}u\int_{\mathbb{R}^{N}}\left|u\mathcal{L}_{M_{2}}[f](u)\right|^{2}\mathrm{d}u\notag\\
&\ge\frac{\left(a_{1}b_{2}-a_{2}b_{1}\right)^{2}N^{2}}{16\pi^{2}}+\biggr[a_{1}a_{2}\Delta x^{2}+b_{1}b_{2}\Delta w^{2} +\left(a_{1}b_{2}+a_{2}b_{1}\right)\mathrm{Cov}_{x,w}
\bigg]^{2}.
\end{align*}
\noindent(\romannumeral2)
Let $A_{2}^{T}A_{1}=A_{1}^{T}A_{2}$, $B_{1}^{T}B_{2}=B_{2}^{T}B_{1}$, $A_{1}^{T}B_{2}=\frac{1}{2}I_{N}$ and $A_{2}^{T}B_{1}=-\frac{1}{2}I_{N}$. Then (\ref{UP for LCT}) reduces to
\begin{align*}
&\int_{\mathbb{R}^{N}}\left|u\mathcal{L}_{M_{1}}[f](u)\right|^{2}\mathrm{d}u\int_{\mathbb{R}^{N}}\left|u\mathcal{L}_{M_{2}}[f](u)\right|^{2}\mathrm{d}u\notag\\
&\ge \frac{N^{2}}{16\pi^{2}}+\biggr[\mu_{\mathrm{min}}(A_{1}^{T}A_{2})\Delta x^{2}+\mu_{\mathrm{min}}(B_{1}^{T}B_{2})\Delta w^{2}\bigg]^{2},
\end{align*}
where $\mu_{\mathrm{min}}(A_{1}^{T}A_{2})$ and $\mu_{\mathrm{min}}(B_{1}^{T}B_{2})$ are the minimum singular values of $A_{1}^{T}A_{2}$ and $B_{1}^{T}B_{2}$, respectively.
However, in a special case, we can show that the lower bound given above is larger than that in \cite{Zhang1}. 
When $f$ is a real function,  $A_{1}=A_{2}$ and $B_{1}=B_{2}$, the inequality 
 (\ref{UP for LCT}) reduces to the following inequality
    \begin{align}\label{Up for real 2}
        &\int_{\mathbb{R}^{N}}\left|u\mathcal{L}_{M_{1}}[f](u)\right|^{2}\mathrm{d}u\int_{\mathbb{R}^{N}}\left|u\mathcal{L}_{M_{2}}[f](u)\right|^{2}\mathrm{d}u\notag\\
&\ge\biggr( \int_{\mathbb{R}^{N}}x^{T}A_{1}^{T}A_{1}
x\left|f(x)\right|^{2}\mathrm{d}x+ \int_{\mathbb{R}^{N}}w^{T}B_{1}^{T}B_{1}
w\left|\widehat{f}(w)\right|^{2}\mathrm{d}w
\bigg)^{2}.
    \end{align}
 The result in \cite[Theorem 1.1]{Zhang1} becomes
\begin{align}\label{Up for real}
&\int_{\mathbb{R}^{N}}\left|u\mathcal{L}_{M_{1}}[f](u)\right|^{2}\mathrm{d}u\int_{\mathbb{R}^{N}}\left|u\mathcal{L}_{M_{2}}[f](u)\right|^{2}\mathrm{d}u\notag\\
&\ge\biggr(\mu_{\mathrm{min}}^{2}(A_{1})\Delta x^{2}+\mu_{\mathrm{min}}^{2}(B_{1})\Delta w^{2}\bigg)^{2}.
\end{align}
 Clearly, the lower bound of (\ref{Up for real 2}) is sharper than that of (\ref{Up for real}).
\end{remark}
\begin{remark}
In the following, we show that the lower bound of (\ref{Up for FMT2}) is larger than that of (\ref{UP for LCT}). By the Minkowski inequality, one has that
\begin{align}\label{first inequality}
&\biggr[\sum_{j=1}^{N}\biggr(\frac{1}{16\pi^{2}}\left|\left(A_{1}B_{2}^{T}-B_{1}A_{2}^{T}\right)_{jj}\right|^{2}+\Bigr|\Bigr(A_{1}XA_{2}^{T}+B_{1}WB_{2}^{T}\notag\\
&\quad+A_{1}\mathrm{Cov}_{X,W}B_{2}^{T}+B_{1}\left(\mathrm{Cov}_{X,W}\right)^{T}A_{2}^{T}\Big)_{jj}
\Big|^{2}
\bigg)^{\frac{1}{2}}
\bigg]^{2}\notag\\
\ge&\frac{1}{16\pi^{2}}\left[\sum_{j=1}^{N}\left|\left(A_1 B_2^T-B_1 A_2^T\right)_{jj}\right|\right]^2+\biggr[\sum_{j=1}^{N}\Bigr|
    \Bigr(A_1 XA_2^T+B_1 W B_2^T\notag\\
&+A_1\mathrm{Cov}_{X,W}B_2^T+B_1\left(\mathrm{Cov}_{X,W}\right)^TA_2^T\Big)_{jj}
    \Big|\bigg]^2\notag\\
    \ge&\frac{\left[\mathrm{tr}\left(A_{1}^{T}B_{2}-A_{2}^{T}B_{1}\right)\right]^{2}}{16\pi^{2}}+\biggr[
    \mathrm{tr}\Bigr(A_1 XA_2^T+B_2 W B_1^T+A_1\mathrm{Cov}_{X,W}B_2^T\notag\\
    &+A_2\mathrm{Cov}_{X,W}B_1^T\Big)
    \bigg]^2.\notag\\
    \end{align}
For $j=1,\cdots, N$, we denote by $e_j^T$ a unit row vector with the $j$-th element being 1. One can calculate that
\begin{align}\label{first equality}
\mathrm{tr}\left(A_{1}XA_{2}^{T}\right)=&\sum_{j=1}^{N}e_{j}^{T}A_{1}XA_{2}^{T}e_{j}\notag\\
=&\sum_{j=1}^{N}\int_{{\mathbb{R}}^{N}}e_{j}^{T}A_{1}xx^{T}A_{2}^{T}e_{j}\left|f(x)\right|^{2}\mathrm{d}x
\notag\\
=&\int_{{\mathbb{R}}^{N}}x^{T}A_{2}^{T}A_{1}x\left|f(x)\right|^{2}\mathrm{d}x.
\end{align}
Similarly, we have
\begin{align}\label{second equality}
\mathrm{tr}\left(B_{2}WB_{1}^{T}\right)
=&\int_{{\mathbb{R}}^{N}}w^{T}B_{1}^{T}B_{2}w\left|\widehat{f}(w)\right|^{2}\mathrm{d}w
\end{align}
and 
\begin{align}\label{third equality}
&\mathrm{tr}\left(A_{1}\mathrm{Cov}_{X,W}B_{2}^{T}\right)+\mathrm{tr}\left(A_{2}\mathrm{Cov}_{X,W}B_{1}^{T}\right)\notag\\
=&\int_{{\mathbb{R}}^{N}}x^{T}\left(A_{1}^{T}B_{2}+A_{2}^{T}B_{1}\right)\nabla \varphi(x)\left|f(x)\right|^2\mathrm{d}x.
\end{align}
Combining (\ref{first inequality})-(\ref{third equality}), we can conclude that the lower bound of (\ref{Up for FMT2}) is stronger than that of (\ref{UP for LCT}).
\end{remark}
Now, we give an example to demonstrate that the lower bound of (\ref{Up for FMT2}) is larger than that of (\ref{UP for LCT}).
\begin{example}
    Let
    \begin{equation*}
        f(x)=\frac{1}{\pi^{\frac{N}{4}}\left(\prod_{k=1}^{N}\zeta_{k}\right)^{\frac{1}{4}}}e^{-\sum_{k=1}^{N}\frac{1}{2\zeta_{k}}x_{k}^{2}}e^{2\pi i\left(\frac{1}{2\varepsilon }|x|^{2}+\beta\right)},
    \end{equation*}
    where $\zeta_{k}>0$,  $\zeta_{j}\neq\zeta_{k}$ for $j, k=1,\cdots,N$, $\varepsilon>0$ and $\beta\in \mathbb{R}$. Then, we have
\begin{align*}
   \Delta x^2
    =\frac{1}{\pi^{\frac{N}{2}}\left(\prod_{k=1}^{N}\zeta_{k}\right)^{\frac{1}{2}}}\sum_{j=1}^{N}\int_{\mathbb{R}^{N}}x_{j}^{2}e^{-\sum_{k=1}^{N}\frac{1}{\zeta_{k}}x_{k}^{2}}\mathrm{d}x
    =\frac{1}{2}\sum_{j=1}^{N}\zeta_{j},
\end{align*}
\begin{align*}
  \Delta w^2
    =&\frac{1}{4\pi^{2}}\sum_{j=1}^{N}\int_{\mathbb{R}^{N}}\left|\frac{\partial f(x)}{\partial x_{j}}\right|^{2}\mathrm{d}x\notag\\
    =&\sum_{j=1}^{N}\left(\frac{1}{4\pi^{2}\zeta_{j}^{2}}+\frac{1}{\varepsilon^{2}}\right)\frac{1}{\pi^{\frac{N}{2}}\left(\prod_{k=1}^{N}\zeta_{k}\right)^{\frac{1}{2}}}\int_{\mathbb{R}^{N}}x_{j}^{2}e^{-\sum_{k=1}^{N}\frac{1}{\zeta_{k}}x_{k}^{2}}\mathrm{d}x\notag\\
    =&\sum_{j=1}^{N}\left(\frac{1}{8\pi^{2}\zeta_{j}}+\frac{\zeta_{j}}{2\varepsilon^{2}}\right)
\end{align*}
and
\begin{align*}
\mathrm{Cov}_{x,w}
    =\frac{1}{\varepsilon}\frac{1}{\pi^{\frac{N}{2}}\left(\prod_{k=1}^{N}\zeta_{k}\right)^{\frac{1}{2}}}\sum_{j=1}^{N}\int_{\mathbb{R}^{N}}x_{j}^{2}e^{-\sum_{k=1}^{N}\frac{1}{\zeta_{k}}x_{k}^{2}}\mathrm{d}x
    =\frac{1}{2\varepsilon}\sum_{j=1}^{N}\zeta_{j}.
\end{align*}
Let $N=2$, $\zeta_{1}=1, \zeta_{2}=2$, $\varepsilon=1$ and 
$M_{1}=\begin{pmatrix}
I_{N} & -I_{N}\\ 
  0   & I_{N}
\end{pmatrix}, M_{2}=\begin{pmatrix}
I_{N} & I_{N}\\ 
I_{N} & 0
\end{pmatrix}$. Using (\ref{prop-e-LM}), we can obtain
\begin{align*}
&\int_{\mathbb{R}^{N}}\left|u\mathcal{L}_{M_{1}}[f](u)\right|^{2}\mathrm{d}u \int_{\mathbb{R}^{N}}\left|u\mathcal{L}_{M_{2}}[f](u)\right|^{2}\mathrm{d}u\notag\\
=&\left(\Delta x^{2}+\Delta w^{2}-2\mathrm{Cov}_{x,w}\right)\left(\Delta x^{2}+\Delta w^{2}+2\mathrm{Cov}_{x,w}\right)\notag\\
=&0.114347245036271.
\end{align*}
Moreover, we have
\begin{align}\label{1}
&\frac{\left[\mathrm{tr}\left(A_{1}^{T}B_{2}-A_{2}^{T}B_{1}\right)\right]^{2}}{16\pi^{2}}+\biggr[\int_{\mathbb{R}^{N}}x^{T}A_{2}^{T}A_{1}
x\left|f(x)\right|^{2}\mathrm{d}x\notag\\
&+ \int_{\mathbb{R}^{N}}w^{T}B_{1}^{T}B_{2}
w\left|\widehat{f}(w)\right|^{2}\mathrm{d}w+ \int_{\mathbb{R}^{N}}x^{T}\left(A_{1}^{T}B_{2}+A_{2}^{T}B_{1}
\right)\nabla \varphi(x)\left|f(x)\right|^2\mathrm{d}x
\bigg]^{2}\notag\\
=&\frac{N^{2}}{4\pi^{2}}+\left(\Delta x^{2}-\Delta w^{2}\right)^{2}
=0.101682097080979\notag\\
\end{align}
and 
\begin{align}\label{2}
    &\biggr[\sum_{j=1}^{N}\biggr(\frac{1}{16\pi^{2}}\left|\left(A_{1}B_{2}^{T}-B_{1}A_{2}^{T}\right)_{jj}\right|^{2}+\Bigr|\Bigr(A_{1}XA_{2}^{T}+B_{1}WB_{2}^{T}\notag\\
&\quad+A_{1}\mathrm{Cov}_{X,W}B_{2}^{T}+B_{1}\left(\mathrm{Cov}_{X,W}\right)^{T}A_{2}^{T}\Big)_{jj}
\Big|^{2}
\bigg)^{\frac{1}{2}}
\bigg]^{2}\notag\\
=&\biggr[\sum_{j=1}^{N}\biggr(\frac{1}{4\pi^{2}}+\Bigr|\Delta x_{j,j}^{2}-\Delta w_{j,j}^{2}
\Big|^{2}
\bigg)^{\frac{1}{2}}
\bigg]^{2}
=0.101722056292651.
\end{align}
Clearly, the
lower bound of (\ref{2}) is bigger than that of (\ref{1}).
\end{example}
We give the HPW uncertainty principle in one time and one FMT domains, which is stated as follows.
\begin{theorem}\label{first theorem}
Let $f(x)=\left|f(x)\right|e^{2\pi i\varphi(x)}, xf(x)$ and $w\widehat{f}(w)\in L^{2}(\mathbb{R}^{N})$. There holds
    \begin{align}\label{Up for f and FMT}
        &\int_{\mathbb{R}^{N}}\left|xf(x)\right|^{2}\mathrm{d}x\int_{\mathbb{R}^{N}}\left|u\mathcal{L}_{M}[f](u)\right|^{2}\mathrm{d}u\notag\\
        &\ge\frac{\left[\mathrm{tr}(B)\right]^{2}}{16\pi^{2}}+\left(\int_{\mathbb{R}^{N}}x^{T}Ax|f(x)|^{2}\mathrm{d}x+\int_{\mathbb{R}^{N}}x^{T}B\nabla\varphi(x)\left|f(x)\right|^2\mathrm{d}x\right)^{2}.
    \end{align}
\end{theorem}
\begin{proof}
We denote $g(x)=f(x)e^{\pi i x^{T}B^{-1}Ax}$. Using (\ref{FMT and FT}), we know that
\begin{equation*}
    \mathcal{L}_{M}[f](u)=
    \frac{e^{\pi i u^{T}DB^{-1}u}}{i^{\frac{N}{2}}\sqrt{\mathrm{det}(B)}}
    \widehat{g}\left(B^{-1}u\right).
\end{equation*}
From proof of (\ref{eq-LM}) and (\ref{nabla f}), we have
\begin{align*}
&\int_{\mathbb{R}^{N}}\left|u\mathcal{L}_{M}[f](u)\right|^{2}\mathrm{d}u\notag\\
    =&\frac{1}{4\pi^{2}}\int_{\mathbb{R}^{N}}\left|B\nabla f(x)+2\pi i Ax f(x)\right|^{2}\mathrm{d}x\notag\\
    =&\frac{1}{4\pi^{2}}\int_{\mathbb{R}^{N}}\left|B\nabla \left|f(x)\right|\right|^{2}\mathrm{d}x+\int_{\mathbb{R}^{N}}\left|Ax\left|f(x)\right|+B\nabla \varphi(x)\left|f(x)\right|
    \right|^{2}\mathrm{d}x.
    \end{align*}
Applying Cauchy-Schwartz’s inequality, we have
\begin{align}\label{inequality 1}
&\frac{1}{4\pi^{2}}\int_{\mathbb{R}^{N}}\left|xf(x)\right|^{2}\mathrm{d}x\int_{\mathbb{R}^{N}}\left|B\nabla \left|f(x)\right|\right|^{2}\mathrm{d}x\notag\\
\ge &\frac{1}{4\pi^{2}}\left[\int_{\mathbb{R}^{N}}x^{T}B\nabla \left|f(x)\right|\left|f(x)\right|\mathrm{d}x\right]^{2}\notag\\
=&\frac{\left[\mathrm{tr}(B)\right]^{2}}{16\pi^{2}}
\end{align}
and
\begin{align}\label{inequality 2}
&\int_{\mathbb{R}^{N}}\left|xf(x)\right|^{2}\mathrm{d}x\int_{\mathbb{R}^{N}}\left|Ax\left|f(x)\right|+B\nabla \varphi(x)\left|f(x)\right|
    \right|^{2}\mathrm{d}x\notag\\
    \ge&\left(\int_{\mathbb{R}^{N}}x^{T}Ax|f(x)|^{2}\mathrm{d}x+x^{T}B\nabla \varphi(x)\left|f(x)\right|^2\mathrm{d}x
    \right)^{2}
\end{align}
Combining (\ref{inequality 1}) and (\ref{inequality 2}), one has (\ref{Up for f and FMT}). 
\end{proof}
\begin{remark}
Here we point out that when $M=J=\begin{pmatrix}
0 & I_{N}\\ 
-I_{N} & 0
\end{pmatrix}$, 
the inequality (\ref{Up for f and FMT}) reduces to the sharper $N$-dimensional HPW uncertainty principle, that is
    \begin{equation*}
        \Delta x^{2} \Delta w^{2}\ge \frac{N^{2}}{16\pi^{2}}+\mathrm{Cov}_{x,w}^{2}.
    \end{equation*}
\end{remark}
\section{\texorpdfstring{$L^p$}{Lp}-type HPW uncertainty principles for free metaplectic transformation}\label{section 4}
In this section, we study $L^p$-type HPW uncertainty principles for FMT with $1\le p\le 2$. The results are stated as follows.
\begin{theorem}\label{UP for Lp-FMT}
Let $f(x)=\left|f(x)\right|e^{2\pi i\varphi(x)}, xf(x)$ and $w\widehat{f}(w)\in L^{2}(\mathbb{R}^{N})$. If 
$u\mathcal{L}_{M_{1}}[f](u)$ and $ u\mathcal{L}_{M_{2}}[f](u) \in L^{p}(\mathbb{R}^{N})$ with $1\le p\le 2, \frac{1}{p}+\frac{1}{q}=1$, then
    \begin{align}\label{UP for LCT M1 M2}
&\left(\int_{\mathbb{R}^{2}}\left|u\mathcal{L}_{M_{1}}[f](u)\right|^{p}\mathrm{d}u\right)^{\frac{2}{p}}\left(\int_{\mathbb{R}^{N}}\left|u\mathcal{L}_{M_{2}}[f](u)\right|^{p}\mathrm{d}u\right)^{\frac{2}{p}}\notag\\
\ge&\left|\mathrm{det}\left(B_{2}A_{1}^{T}-A_{2}B_{1}^{T}\right)\right|^{\frac{2}{p}-1}\bigg[\frac{\left[\mathrm{tr}\left(A_{1}^{T}B_{2}-A_{2}^{T}B_{1}\right)\right]^{2}}{16\pi^{2}}+\bigg(
\int_{\mathbb{R}^{N}}x^{T}A_{2}^{T}A_{1}
x\left|f(x)\right|^{2}\mathrm{d}x\notag\\
&+ \int_{\mathbb{R}^{N}}w^{T}B_{1}^{T}B_{2}
w\left|\widehat{f}(w)\right|^{2}\mathrm{d}w+ \int_{\mathbb{R}^{N}}x^{T}\left(A_{1}^{T}B_{2}+A_{2}^{T}B_{1}
\right)\nabla \varphi(x)\left|f(x)\right|^2\mathrm{d}x
\bigg)^{2}\bigg].\notag\\
\end{align}
\end{theorem}
\begin{proof}
Let $M_{3}=M_{2}M_{1}^{-1}$. Since
$M_1^{-1}=\begin{pmatrix}
D_1^{T} & -B_1^{T}\\ 
-C_1^{T} & A_1^{T}
\end{pmatrix}$, we have
\begin{align*}
    M_{3}=\begin{pmatrix}
A_{3} & B_{3}\\ 
C_{3} & D_{3}
\end{pmatrix}=\begin{pmatrix}
A_{2}D_{1}^{T}-B_{2}C_{1}^{T} & B_{2}A_{1}^{T}-A_{2}B_{1}^{T}\\ 
C_{2}D_{1}^{T}-D_{2}C_{1}^{T} & D_{2}A_{1}^{T}-C_{2}B_{1}^{T}
\end{pmatrix}.
\end{align*}
Let $H(x)=\mathcal{L}_{M_{1}}[f](x)e^{\pi i x^{T}B_{3}^{-1}A_{3}x}$.
By (\ref{the additive property}), 
 (\ref{FMT3 and ft}), $B_3^{-1}u=w$ and 
the Hausdorff–Young inequality, it follows that
\begin{align}\label{ieq-LM}
\left(\int_{\mathbb{R}^{N}}\left|u\mathcal{L}_{M_{2}}[f](u)\right|^{p}\mathrm{d}u\right)^{\frac{2}{p}}
=&\left(\int_{\mathbb{R}^{N}}\left|u\mathcal{L}_{M_{3}}[\mathcal{L}_{M_{1}}[f]](u)\right|^{p}\mathrm{d}u\right)^{\frac{2}{p}}\notag\\
=&\left|\mathrm{det}(B_{3})\right|^{-1}\left(\int_{\mathbb{R}^{N}}\left|u\widehat{H}\left(B_{3}^{-1}u\right)\right|^{p}\mathrm{d}u\right)^{\frac{2}{p}}\notag\\
=&\left|\mathrm{det}(B_{3})\right|^{\frac{2}{p}-1}\left(\int_{\mathbb{R}^{N}}\left|B_{3}w\widehat{H}(w)\right|^{p}\mathrm{d}w\right)^{\frac{2}{p}}\notag\\
=&\frac{\left|\mathrm{det}(B_{3})\right|^{\frac{2}{p}-1}}{4\pi^{2}}\left(\int_{\mathbb{R}^{N}}\left|
B_{3}\left[\nabla H\right]^{\wedge}(w)\right|^{p}\mathrm{d}w\right)^{\frac{2}{p}}\notag\\
\ge&\frac{\left|\mathrm{det}(B_{3})\right|^{\frac{2}{p}-1}}{4\pi^{2}}\left(\int_{\mathbb{R}^{N}}\left|
B_{3}\nabla H(u)\right|^{q}\mathrm{d}u\right)^{\frac{2}{q}}.
\end{align}
Using Hölder's inequality and (\ref{nabla H}), we have that
\begin{align*}
&\left(\int_{\mathbb{R}^{N}}\left|u\mathcal{L}_{M_{1}}[f](u)\right|^{p}\mathrm{d}u\right)^{\frac{2}{p}}\left(\int_{\mathbb{R}^{N}}\left|u\mathcal{L}_{M_{2}}[f](u)\right|^{p}\mathrm{d}u\right)^{\frac{2}{p}}\notag\\
=&\left(\int_{\mathbb{R}^{2}}\left|u\mathcal{L}_{M_{1}}[f](u)\right|^{p}\mathrm{d}u\right)^{\frac{2}{p}}\left(\int_{\mathbb{R}^{N}}\left|u\mathcal{L}_{M_{3}}[\mathcal{L}_{M_{1}}[f]](u)\right|^{p}\mathrm{d}u\right)^{\frac{2}{p}}\\
    \ge&\frac{\left|\mathrm{det}(B_{3})\right|^{\frac{2}{p}-1}}{4\pi^{2}}\left(\int_{\mathbb{R}^{N}}\left|u\mathcal{L}_{M_{1}}[f](u)\right|^{p}\mathrm{d}u\right)^{\frac{2}{p}}\left(\int_{\mathbb{R}^{N}}\left|
B_{3}\nabla H(u)\right|^{q}\mathrm{d}u\right)^{\frac{2}{q}}\\
\ge& \frac{\left|\mathrm{det}(B_{3})\right|^{\frac{2}{p}-1}}{4\pi^{2}}\left|\int_{\mathbb{R}^{N}}ie^{\pi i u^{T}B_{3}^{-1}A_{3}u}u^{T}\mathcal{L}_{M_{1}}[f](u)
B_{3}\overline{\nabla H(u)}\mathrm{d}u\right|^{2}\\
=&\frac{\left|\mathrm{det}(B_{3})\right|^{\frac{2}{p}-1}}{4\pi^{2}}\left|I_{1}+I_{2}\right|^{2},
\end{align*}
where $I_{1}$ and $I_{2}$ are given by (\ref{J_{1}}) and (\ref{J_{2}}) respectively.
From Lemmas \ref{the first lemma} and \ref{the second lemma},
one has the desired inequality (\ref{UP for LCT M1 M2}). 
\end{proof}
\begin{theorem}
Let $f(x)=\left|f(x)\right|e^{2\pi i\varphi(x)}$ and $w\widehat{f}(w)\in L^2(\mathbb{R}^{N})$. If $ xf(x)\in L^{2}(\mathbb{R}^{N})\cap L^{p}(\mathbb{R}^{N})$ and $u\mathcal{L}_{M}[f](u) \in L^{p}(\mathbb{R}^{N})$ with $1\le p\le 2,\frac{1}{p}+\frac{1}{q}=1,$ then
    \begin{align*}
        &\left(\int_{\mathbb{R}^{N}}\left|xf(x)\right|^{p}\mathrm{d}x\right)^{\frac{2}{p}}
        \left(\int_{\mathbb{R}^{N}}\left|u \mathcal{L}_{M}[f](u)\right|^{p}\mathrm{d}u\right)^{\frac{2}{p}}\notag\\
        &\ge\left|\mathrm{det}(B)\right|^{\frac{2}{p}-1}\left[\frac{\left[\mathrm{tr}(B)\right]^{2}}{16\pi^{2}}+\left(\int_{\mathbb{R}^{N}}x^{T}Ax|f(x)|^{2}\mathrm{d}x+\int_{\mathbb{R}^{N}}x^{T}B\nabla \varphi(x)\left|f(x)\right|^2\mathrm{d}x\right)^{2}\right].
    \end{align*}
\end{theorem}
\begin{proof}
Let $g(x)=f(x)e^{\pi i x^{T}B^{-1}A}$. From (\ref{FMT and FT}), we know that
\begin{equation*}
    \mathcal{L}_{M}[f](u)=
    \frac{e^{\pi i u^{T}DB^{-1}u}}{i^{\frac{N}{2}}\sqrt{\mathrm{det}(B)}}
    \widehat{g}\left(B^{-1}u\right).
\end{equation*}
Similar to (\ref{ieq-LM}), we have
\begin{equation*}
\left(\int_{\mathbb{R}^{N}}\left|u \mathcal{L}_{M}[f](u)\right|^{p}\mathrm{d}u\right)^{\frac{2}{p}}
        \ge\frac{\left|\mathrm{det}(B)\right|^{\frac{2}{p}-1}}{4\pi^{2}}\left(\int_{\mathbb{R}^{N}}\left|B\nabla g(x)\right|^{q}\mathrm{d}x\right)^{\frac{2}{q}}.
\end{equation*}
Using Hölder's inequality, (\ref{nabla g}) and (\ref{nabla f}), we have
\begin{align*}
      &\left(\int_{\mathbb{R}^{N}}\left|xf(x)\right|^{p}\mathrm{d}x\right)^{\frac{2}{p}}
        \left(\int_{\mathbb{R}^{N}}\left|u \mathcal{L}_{M}[f](u)\right|^{p}\mathrm{d}u\right)^{\frac{2}{p}}\notag\\
        \ge&\frac{\left|\mathrm{det}(B)\right|^{\frac{2}{p}-1}}{4\pi^{2}}\left(\int_{\mathbb{R}^{N}}\left|xf(x)\right|^{p}\mathrm{d}x\right)^{\frac{2}{p}}\left(\int_{\mathbb{R}^{N}}\left|B\nabla g(x)\right|^{q}\mathrm{d}x\right)^{\frac{2}{q}}\notag\\
        \ge&\frac{\left|\mathrm{det}(B)\right|^{\frac{2}{p}-1}}{4\pi^{2}}\left|\int_{\mathbb{R}^{N}}ie^{\pi i x^{T}B^{-1}Ax}x^{T}f(x)B\overline{\nabla g(x)}\mathrm{d}x\right|^{2}\notag\\
        =&\frac{\left|\mathrm{det}(B)\right|^{\frac{2}{p}-1}}{4\pi^{2}}\left|\int_{\mathbb{R}^{N}}x^{T}B\overline{\nabla f(x)}f(x)\mathrm{d}x+2\pi \int_{\mathbb{R}^{N}}x^{T}Ax\left|f(x)\right|^{2}\mathrm{d}x
        \right|^{2}\notag\\
       =&\frac{\left|\mathrm{det}(B)\right|^{\frac{2}{p}-1}}{4\pi^{2}}\bigg[\left(\int_{\mathbb{R}^{2}}x^{T}B\nabla\left|f(x)\right|\left|f(x)\right|\mathrm{d}x\right)^{2}+4\pi^{2}\bigg(\int_{\mathbb{R}^{N}}x^{T}Ax|f(x)|^{2}\mathrm{d}x\notag\\
       &+\int_{\mathbb{R}^{N}}x^{T}B\nabla \varphi(x)\left|f(x)\right|^2\mathrm{d}x\bigg)^{2}
        \bigg]\notag\\
        =&\left|\mathrm{det}(B)\right|^{\frac{2}{p}-1}\left[\frac{\left[\mathrm{tr}(B)\right]^{2}}{16\pi^{2}}+\left(\int_{\mathbb{R}^{N}}x^{T}Ax|f(x)|^{2}\mathrm{d}x+\int_{\mathbb{R}^{N}}x^{T}B\nabla \varphi(x)\left|f(x)\right|^2\mathrm{d}x\right)^{2}\right].
\end{align*}
\end{proof}
\begin{theorem}
Let $f(x)=\left|f(x)\right|e^{2\pi i\varphi(x)}$ and $w\widehat{f}(w)\in L^2(\mathbb{R}^{N})$. If $ xf(x)\in L^{2}(\mathbb{R}^{N})\cap L^{p}(\mathbb{R}^{N})$ and $u\mathcal{L}_{M}[f](u) \in L^{p}(\mathbb{R}^{N})$ with $1\le p\le 2,\frac{1}{p}+\frac{1}{q}=1,$ then
    \begin{align}\label{UP for fmt}
&\left(\int_{\mathbb{R}^{N}}\left|xf(x)\right|^{p}\mathrm{d}x\right)^{\frac{q}{p}}
        \left(\int_{\mathbb{R}^{N}}\left|u \mathcal{L}_{M}[f](u)\right|^{p}\mathrm{d}\xi\right)^{\frac{q}{p}}\notag\\
        &\ge\left|\mathrm{det}(B)\right|^{\frac{q}{p}-\frac{q}{2}}\left[
\frac{\left|\mathrm{tr}(B)\right|^{q}}{(4\pi)^{q}}
+\left|\int_{\mathbb{R}^{N}}x^{T}Ax|f(x)|^{2}\mathrm{d}x+\int_{\mathbb{R}^{N}}x^{T}B\nabla \varphi(x)\left|f(x)\right|^2\mathrm{d}x
\right|^{q}\right].\notag\\
\end{align}
\end{theorem}
\begin{proof}
Let $g(x)=f(x)e^{\pi i x^{T}B^{-1}A}$. Similar to (\ref{ieq-LM}), by (\ref{nabla g}) and (\ref{nabla f}), a direct computation gives
    \begin{align*}
        &\left(\int_{\mathbb{R}^{N}}\left|u\mathcal{L}_{M}[f](u)\right|^{p}\mathrm{d}u\right)^{\frac{q}{p}}\notag\\
        \ge&\frac{\left|\mathrm{det}(B)\right|^{\frac{q}{p}-\frac{q}{2}}}{(2\pi)^{q}}\int_{\mathbb{R}^{N}}\left|B\nabla g(x)\right|^{q}\mathrm{d}x\notag\\
\ge&\frac{\left|\mathrm{det}(B)\right|^{\frac{q}{p}-\frac{q}{2}}}{(2\pi)^{q}}\int_{\mathbb{R}^{N}}\left|B\nabla \left|f(x)\right|\right|^{q}\mathrm{d}x+\left|\mathrm{det}(B)\right|^{\frac{q}{p}-\frac{q}{2}}\int_{\mathbb{R}^{N}}\left|Ax\left|f(x)\right|+B\nabla \varphi(x)\left|f(x)\right|\right|^{q}\mathrm{d}x,
    \end{align*}
    where the last inequality follows from the fact that
\begin{align*}
     &\left[\left|B\nabla \left|f(x)\right|\right|^{2}+4\pi^{2}\left|Ax\left|f(x)\right|+B\nabla \varphi(x)\left|f(x)\right|\right|^{2}\right]^{\frac{q}{2}}\notag\\
     \ge&\left|B\nabla \left|f(x)\right|\right|^{q}+(2\pi)^{q}\left|Ax\left|f(x)\right|+B\nabla \varphi(x)\left|f(x)\right|\right|^{q}.
     \end{align*}
Applying Hölder's inequality, we have
\begin{align}\label{the first inequality}
   & \frac{\left|\mathrm{det}(B)\right|^{\frac{q}{p}-\frac{q}{2}}}{(2\pi)^{q}}\left(\int_{\mathbb{R}^{N}}\left|xf(x)\right|^{p}\mathrm{d}x\right)^{\frac{q}{p}}
\int_{\mathbb{R}^{N}}\left|B\nabla \left|f(x)\right|\right|^{q}\mathrm{d}x\notag\\
\ge&\frac{\left|\mathrm{det}(B)\right|^{\frac{q}{p}-\frac{q}{2}}}{(2\pi)^{q}}\left|\int_{\mathbb{R}^{N}}x^{T}B\nabla \left|f(x)\right|\left|f(x)\right|\mathrm{d}x\right|^{q}\notag\\
=&\frac{\left|\mathrm{det}(B)\right|^{\frac{q}{p}-\frac{q}{2}}\left|\mathrm{tr}(B)\right|^{q}}{(4\pi)^{q}}
\end{align}
and
\begin{align}\label{the second inequality}
   & \left|\mathrm{det}(B)\right|^{\frac{q}{p}-\frac{q}{2}}\left(\int_{\mathbb{R}^{N}}\left|xf(x)\right|^{p}\mathrm{d}x\right)^{\frac{q}{p}}\int_{\mathbb{R}^{N}}\left|Ax\left|f(x)\right|+B\nabla \varphi(x)\left|f(x)\right|\right|^{q}\mathrm{d}x\notag\\
    \ge&\left|\mathrm{det}(B)\right|^{\frac{q}{p}-\frac{q}{2}}\left|\int_{\mathbb{R}^{N}}x^{T}Ax|f(x)|^{2}\mathrm{d}x+\int_{\mathbb{R}^{N}}x^{T}B\nabla \varphi(x)\left|f(x)\right|^2\mathrm{d}x
\right|^{q}
\end{align}
Combining (\ref{the first inequality}) and (\ref{the second inequality}), one has (\ref{UP for fmt}). 
\end{proof}
\section{HPW uncertainty principles for metaplectic operators}\label{section 5}In this section, we consider the HPW uncertainty principle for general metaplectic operators. In particular, we obtain two versions of uncertainty principles, where the first version corresponds to the result of Theorem \ref{Up for MO component}, and the second one corresponds to Theorem  \ref{UP for L2 M1 M2}.

Let $\mathcal H$ be a Hilbert space
with inner product $\left<\cdot,\cdot\right>$ and with norm $\Vert\cdot\Vert \triangleq \left<\cdot,\cdot\right>^{\frac{1}{2}}$. Suppose that $\widehat{A}$ and $\widehat{B}$ are two self-adjoint operators with domains $D(\widehat{A})$ and $D(\widehat{B})$, respectively. Consequently, the domains of the products $\widehat{A}\widehat{B}$ and $\widehat{B}\widehat{A}$ are given by
\begin{equation}\label{domain of AB}
D(\widehat{A}\widehat{B})=\left\{f\in D(\widehat{B}):\widehat{B}f\in D(\widehat{A})\right\}
\end{equation}
and
\begin{equation}\label{domain of BA}
D(\widehat{B}\widehat{A})=\left\{f\in D(\widehat{A}):\widehat{A}f\in D(\widehat{B})\right\}.
\end{equation}
The commutator and anticommutator are, respectively, defined as
\begin{equation*}
[\widehat{A},\widehat{B}]\triangleq
\widehat{A}\widehat{B}-\widehat{B}\widehat{A}\quad \text{on}\quad D([\widehat{A}, \widehat{B}])=D(\widehat{A}\widehat{B})\cap D(\widehat{B}\widehat{A})
\end{equation*}
and
\begin{equation*}
[\widehat{A},\widehat{B}]_{+}\triangleq
\widehat{A}\widehat{B}+\widehat{B}\widehat{A}\quad \text{on}\quad D([\widehat{A}, \widehat{B}]_{+})=D(\widehat{A}\widehat{B})\cap D(\widehat{B}\widehat{A}).
\end{equation*}
\begin{prop}[\cite{Cohen}]\label{prop}
    Let $\widehat{A}, \widehat{B}$ be two self-adjoint operators on $\mathcal H $. Then
    \begin{equation}\label{operator inequality}
    \Vert\widehat{A}f\Vert_{2}^{2}
    \Vert \widehat{B}f\Vert_{2}^{2}\ge \frac{1}{4}|
\langle [\widehat{A},\widehat{B}]f,f\rangle|^{2}+\frac{1}{4}|
\langle [\widehat{A},\widehat{B}]_{+}f,f\rangle|^{2},
    \end{equation}
    for all $f\in D(\widehat{A}\widehat{B})\cap D(\widehat{B}\widehat{A})$.
    Moreover, the equality in (\ref{operator inequality}) holds if and only if 
    \begin{equation*}
    \widehat{A}f=il\widehat{B}f,
    \end{equation*}
for some $l\in \mathbb{R}$.
\end{prop}
\begin{remark}
   Let $f(x)=\left|f(x)\right|e^{2\pi i\varphi(x)}$,  $x_{j}f(x)$ and $ \frac{\partial f(x)}{\partial x_j}\in L^{2}(\mathbb{R}^{N})$. If we set $\left(\widehat{A}f\right)(x)=\left(\widehat{X}_{j}f\right)(x)=x_{j}f(x)$, $\left(\widehat{B}f\right)(x)=\left(\widehat{P}_{j}f\right)(x)=\frac{1}{2\pi i}\frac{\partial f(x)}{\partial x_j}$, for $j=1,\cdots,N$,  then we have that
    \begin{equation}\label{component HPW}
        \Delta x_{j,j}^{2}\Delta w_{j,j}^{2}\ge\frac{1}{16\pi^{2}}+{\mathrm{Cov}_{x,w}^{j,j}}^{2}.
    \end{equation}
    The equality of (\ref{component HPW}) is satisfied if and only if
$f(x)=e^{-x^{T}Lx+d_{0}}$,
where $L=\mathrm{diag}(l_{1},\cdots,l_{N})$ with $l_{j}>0$, $d_{0}\in \mathbb{R}$ and $
\frac{{\frac{\pi}{2}}^{\frac{N}{2}}e^{2d_{0}}}{\left(\prod_{j=1}^{N} l_{j}\right)
^{\frac{1}{2}}}=1$. 
 Applying the Cauchy-Schwartz inequality, we obtain the following result.
        \begin{equation}\label{N-dim HPW}
       \Delta x^{2}\Delta w^{2}
\ge\left[\sum_{j=1}^{N}\left(\frac{1}{16\pi^{2}}+{\mathrm{Cov}_{x,w}^{j,j}}^{2}\right)^{\frac{1}{2}}\right]^{2}.
    \end{equation}
    The equality in (\ref{N-dim HPW}) holds if and only if 
        $f(x)=e^{-L\left|x\right|^{2}+d_{1}}$,
    where $L>0$, $d_{1}\in \mathbb{R}$ and $\left(\frac{\pi}{2L}\right)^{\frac{N}{2}}e^{2d_{1}}=1$.
\end{remark}

\smallskip
In the following, we consider uncertainty principles for general metaplectic operators under the assumptions $(\widehat{M}f)(x)\in L^2(\mathbb{R}^{N})$ and 
$\left<x\right>_{\widehat{ Mf}}=0$. 
We first recall that $\Sigma=\left(\Sigma_{\alpha,\beta}\right)$ is the covariance matrix with
\begin{equation*}
\Sigma_{\alpha,\beta}=\Big\langle \left(\frac{\widehat{Z}_{\alpha}\widehat{Z}_{\beta}+\widehat{Z}_{\beta}\widehat{Z}_{\alpha}}{2}
\right)f,f\Big\rangle=\int_{\mathbb{R}^{2N}}z_{\alpha}z_{\beta}W_{\sigma}f(z)dz,\quad \alpha, \beta=1,\cdots,2N,
\end{equation*}
where the operators $\widehat{Z}_{\alpha}$ are defined in (\ref{fundamental operator}) and 
$W_{\sigma}f(z)$ is the Wigner function of $f\in L^2(\mathbb{R}^{N})$ for $z=(x,w)\in \mathbb{R}^{2N}$ (see \S \ref{section 2} for its definition).
In \S \ref{section 2}, 
we have that $\int_{\mathbb R^{2N}} {(1+|z|^2)\left|W_\sigma f(z)\right|} dz<\infty$ ensures that $\Sigma_{\alpha,\beta}<\infty$ for $\alpha,\beta=1,\cdots,2N$. In addition, there holds
\begin{equation*}
    \int_{\mathbb R^{2N}} {(1+|z|^2)\left|W_\sigma f(z)\right|} dz<\infty \iff f, xf(x) \text{ and } w\widehat{f}(w)\in L^2(\mathbb{R}^{N}).
\end{equation*}
Let $\widehat{M}$ be associated metaplectic operators of $M$. We have that \begin{equation}\label{symplectic covariance}
\widehat{M}^{*}\widehat{Z}_{\alpha}\widehat{M}=\sum_{\beta=1}^{2N}M_{\alpha, \beta}\widehat{Z}_{\beta}, \quad \alpha=1, \dots,2N,
\end{equation}

To obtain uncertainty principles for general metaplectic operators, we need the following technical lemma. 
\begin{lemma}\label{lemma operators}
Let $\int_{\mathbb R^{2N}} {(1+|z|^2)\left|W_\sigma f(z)\right|} dz<\infty$. For $j, k=1,\cdots,N$, there holds
    \begin{align}\label{UP for operator spatial}
        &\int_{\mathbb{R}^{N}}\left|x_{j}\left(\widehat{M_{1}}f\right)(x)\right|^{2}\mathrm{d}x
\int_{\mathbb{R}^{N}}\left|\xi_{k}\left(\widehat{M_{2}}f\right)(\xi)\right|^{2}\mathrm{d}\xi\notag\\
&\ge\frac{1}{16\pi^{2}}\left|\left(M_{1}JM_{2}^{T}\right)_{jk}\right|^{2}+\left|\left(M_{1}\Sigma M_{2}^{T}\right)_{jk}
\right|^{2}.
    \end{align}
\end{lemma}
\begin{proof}
If either $\int_{\mathbb{R}^{N}}\left|x_{j}\left(\widehat{M_{1}}f\right)(x)\right|^{2}\mathrm{d}x=\infty$ or $\int_{\mathbb{R}^{N}}\left|\xi_{k}\left(\widehat{M_{2}}f\right)(\xi)\right|^{2}\mathrm{d}\xi=\infty$, then inequality (\ref{UP for operator spatial}) obviously holds. Hence, we assume that $\int_{\mathbb{R}^{N}}\left|x_{j}\left(\widehat{M_{1}}f\right)(x)\right|^{2}\mathrm{d}x$ and  $\int_{\mathbb{R}^{N}}\left|\xi_{k}\left(\widehat{M_{2}}f\right)(\xi)\right|^{2}\mathrm{d}\xi$ are finite.
To prove (\ref{UP for operator spatial}), we apply (\ref{operator inequality}) to the operators
\begin{equation*}
    \widehat{A}=\sum_{\alpha=1}^{2N}M_{j,\alpha}^{(1)}\widehat{Z}_{\alpha},\quad \quad \widehat{B}=\sum_{\beta=1}^{2N}M_{k,\beta}^{(2)}\widehat{Z}_{\beta},
\end{equation*}
where $j,k=1,...,N$ are fixed. 
Using (\ref{cov property}), we have 
\begin{align*}
[\widehat{A},\widehat{B}]=&\sum_{1\le \alpha,\beta\le 2N}M_{j,\alpha}^{(1)}M_{k,\beta}^{(2)}[\widehat{Z}_{\alpha},\widehat{Z}_{\beta}]\notag\\
=&\frac{i}{2\pi}\sum_{1\le \alpha,\beta\le 2N}M_{j,\alpha}^{(1)}M_{k,\beta}^{(2)}J_{\alpha,\beta}\notag\\
=&\frac{i}{2\pi}\left(M_{1}JM_{2}^{T}\right)_{j,k}.
\end{align*}
Hence, we have
\begin{equation}\label{operator inequality 1}
    \frac{1}{4}|
\langle [\widehat{A},\widehat{B}]f,f\rangle|^{2}=\frac{1}{16\pi^{2}}\left|\left(M_{1}JM_{2}^{T}\right)_{jk}\right|^{2}.
\end{equation}
Note that
\begin{align*}
   \frac{1}{2} \langle [\widehat{A},\widehat{B}]_{+}f,f\rangle=&\sum_{1\le \alpha,\beta\le 2N}M_{j,\alpha}^{(1)}M_{k,\beta}^{(2)}\langle \left(\frac{\widehat{Z}_{\alpha}\widehat{Z}_{\beta}+\widehat{Z}_{\beta}\widehat{Z}_{\alpha}}{2}
    \right)f,f\rangle \notag\\
   =&\sum_{1\le \alpha,\beta\le 2N}M_{j,\alpha}^{(1)}M_{k,\beta}^{(2)}\Sigma_{\alpha,\beta}\notag\\
   =&\left(M_{1}\Sigma M_{2}^{T}\right)_{jk}.
\end{align*}
Then we have
\begin{equation}\label{operator inequality 2}
     \frac{1}{4}|
\langle [\widehat{A},\widehat{B}]_{+}f,f\rangle|^{2}=\left|\left(M_{1}\Sigma M_{2}^{T}\right)_{jk}
\right|^{2}.
\end{equation}
Using (\ref{symplectic covariance}), we have
\begin{align*}
    \widehat{A}=\widehat{M_{1}}^{*}\widehat{X}_{j}\widehat{M_{1}}\text{ and }\widehat{B}=\widehat{M_{2}}^{*}\widehat{X}_{k}\widehat{M_{2}}.
\end{align*}
Since $\widehat{M_{1}}$ and $\widehat{M_{2}}$ are unitary operators, we have 
\begin{equation}\label{operator inequality 3}
    \Vert\widehat{A}f\Vert_{2}^{2}=\Vert(\widehat{M_{1}}^{*}\widehat{X}_{j}\widehat{M_{1}})f\Vert_{2}^{2}
    =\Vert(\widehat{X}_{j}\widehat{M_{1}})f\Vert_{2}^{2}
    =\int_{\mathbb{R}^{N}}\left|x_{j}\left(\widehat{M_{1}}f\right)(x)\right|^{2}\mathrm{d}x
\end{equation}
and
\begin{equation}\label{operator inequality 4}
    \Vert\widehat{B}f\Vert_{2}^{2}=\Vert(\widehat{M_{2}}^{*}\widehat{X}_{k}\widehat{M_{2}})f\Vert_{2}^{2}
    =\Vert(\widehat{X}_{k}\widehat{M_{2}})f\Vert_{2}^{2}
    =\int_{\mathbb{R}^{N}}\left|\xi_{k}\left(\widehat{M_{2}}f\right)(\xi)\right|^{2}\mathrm{d}\xi.
\end{equation}
Combining (\ref{operator inequality}) and 
(\ref{operator inequality 1})-(\ref{operator inequality 4}), we obtain
(\ref{UP for operator spatial}). 
\end{proof}
\begin{theorem}\label{corollary 1}
Let $\int_{\mathbb R^{2N}} {(1+|z|^2)\left|W_\sigma f(z)\right|} dz<\infty$. There holds 
    \begin{align}\label{UP for operator directional}
        &\int_{\mathbb{R}^{N}}\left|x\left(\widehat{M_{1}}f\right)(x)\right|^{2}\mathrm{d}x
\int_{\mathbb{R}^{N}}\left|\xi\left(\widehat{M_{2}}f\right)(\xi)\right|^{2}\mathrm{d}\xi\notag\\
&\ge\left[\sum_{j=1}^{N}\left(\frac{1}{16\pi^{2}}\left|\left(M_{1}JM_{2}^{T}\right)_{jj}\right|^{2}+\left|\left(M_{1}\Sigma M_{2}^{T}\right)_{jj}
\right|^{2}
\right)^{\frac{1}{2}}
\right]^{2}.
    \end{align}
\end{theorem}
\begin{proof}
    Applying Cauchy-Schwartz's inequality and (\ref{UP for operator spatial}), we have
    \begin{align*}
&\int_{\mathbb{R}^{N}}\left|x\left(\widehat{M_{1}}f\right)(x)\right|^{2}\mathrm{d}x
\int_{\mathbb{R}^{N}}\left|\xi\left(\widehat{M_{2}}f\right)(\xi)\right|^{2}\mathrm{d}\xi\notag\\
\ge&\left[\sum_{j=1}^{N}\left(\int_{\mathbb{R}^{N}}\left|x_{j}\left(\widehat{M_{1}}f\right)(x)\right|^{2}\mathrm{d}x\int_{\mathbb{R}^{N}}\left|\xi_{j}\left(\widehat{M_{2}}f\right)(\xi)\right|^{2}\mathrm{d}\xi
\right)^{\frac{1}{2}}
\right]^{2}\notag\\
\ge&\left[\sum_{j=1}^{N}\left(
\frac{1}{16\pi^{2}}\left|\left(M_{1}JM_{2}^{T}\right)_{jj}\right|^{2}+\left|\left(M_{1}\Sigma M_{2}^{T}\right)_{jj}
\right|^{2}
\right)^{\frac{1}{2}}
\right]^{2}.
    \end{align*}
    
\end{proof}
\begin{remark}
Let $\widehat{M_{1}}f=\mathcal{L}_{M_{1}}[f]$ and $\widehat{M_{2}}f=\mathcal{L}_{M_{2}}[f]$ in Theorem \ref{corollary 1}. Then one has
\begin{align*}
      &\int_{\mathbb{R}^{N}}\left|u\mathcal{L}_{M_{1}}[f](u)\right|^{2}\mathrm{d}u
\int_{\mathbb{R}^{N}}\left|u\mathcal{L}_{M_{2}}[f](u)\right|^{2}\mathrm{d}u\notag\\
&\ge\left[\sum_{j=1}^{N}\left(\frac{1}{16\pi^{2}}\left|\left(M_{1}JM_{2}^{T}\right)_{jj}\right|^{2}+\left|\left(M_{1}\Sigma M_{2}^{T}\right)_{jj}
\right|^{2}
\right)^{\frac{1}{2}}
\right]^{2}.
\end{align*}
If we write $f(x)=\left|f(x)\right| e^{2\pi i\varphi(x)}$,  one can calculate that
\begin{equation}\label{Sigma}
        \Sigma=\begin{pmatrix}
X  & \mathrm{Cov}_{X,W}\\
{\left(\mathrm{Cov}_{X,W}\right)}^{T}  & W
\end{pmatrix}.
    \end{equation}
Consequently, we have
\begin{align*}
&\int_{\mathbb{R}^{N}}\left|u\mathcal{L}_{M_{1}}[f](u)\right|^{2}\mathrm{d}u
\int_{\mathbb{R}^{N}}\left|u\mathcal{L}_{M_{2}}[f](u)\right|^{2}\mathrm{d}u\notag\\
&\ge\biggr[\sum_{j=1}^{N}\biggr(\frac{1}{16\pi^{2}}\left|\left(A_{1}B_{2}^{T}-B_{1}A_{2}^{T}\right)_{jj}\right|^{2}+\Bigr|\Bigr(A_{1}XA_{2}^{T}+B_{1}WB_{2}^{T}\notag\\
&\quad+A_{1}\mathrm{Cov}_{X,W}B_{2}^{T}+B_{1}\left(\mathrm{Cov}_{X,W}\right)^{T}A_{2}^{T}\Big)_{jj}
\Big|^{2}
\bigg)^{\frac{1}{2}}
\bigg]^{2},
\end{align*}
which is the result of Theorem \ref{Up for MO component}.
\end{remark}
\begin{theorem}\label{Up for MO}
Let $\int_{\mathbb R^{2N}} {(1+|z|^2)\left|W_\sigma f(z)\right|} dz<\infty$. There holds 
      \begin{align*}
&\int_{\mathbb{R}^{N}}\left|x\left(\widehat{M_{1}}f\right)(x)\right|^{2}\mathrm{d}x
\int_{\mathbb{R}^{N}}\left|\xi\left(\widehat{M_{2}}f\right)(\xi)\right|^{2}\mathrm{d}\xi\notag\\
&\ge\frac{1}{16\pi^2}\left[\sum_{j=1}^{N}\left(M_{1}JM_{2}^{T}\right)_{jj}\right]^{2}+\left[\sum_{j=1}^{N}\left(M_{1}\Sigma M_{2}^{T}\right)_{jj}\right]^{2}.
    \end{align*}
\end{theorem}
\begin{proof}
    By (\ref{UP for operator directional}) and 
the Minkowski inequality, we have
    \begin{align*}
        &\int_{\mathbb{R}^{N}}\left|x\left(\widehat{M_{1}}f\right)(x)\right|^{2}\mathrm{d}x
\int_{\mathbb{R}^{N}}\left|\xi\left(\widehat{M_{2}}f\right)(\xi)\right|^{2}\mathrm{d}\xi\notag\\
\ge&\left[\sum_{j=1}^{N}\left(
\frac{1}{16\pi^{2}}\left|\left(M_{1}JM_{2}^{T}\right)_{jj}\right|^{2}+\left|\left(M_{1}\Sigma M_{2}^{T}\right)_{jj}
\right|^{2}
\right)^{\frac{1}{2}}
\right]^{2}\notag\\
\ge&\frac{1}{16\pi^{2}}\left[\sum_{j=1}^{N}\left|\left(M_{1}JM_{2}^{T}\right)_{jj}\right|\right]^{2}+\left[\sum_{j=1}^{N}\left|\left(M_{1}\Sigma M_{2}^{T}\right)_{jj}\right|\right]^{2}\notag\\
\ge&\frac{1}{16\pi^{2}}\left[\sum_{j=1}^{N}\left(M_{1}JM_{2}^{T}\right)_{jj}\right]^{2}+\left[\sum_{j=1}^{N}\left(M_{1}\Sigma M_{2}^{T}\right)_{jj}\right]^{2}.
    \end{align*}
\end{proof}
\begin{remark}
    Let 
$\widehat{M_{1}}f=\mathcal{L}_{M_{1}}[f]$ and $\widehat{M_{2}}f=\mathcal{L}_{M_{2}}[f]$ for $f(x)=\left|f(x)\right| e^{2\pi i\varphi(x)}$ in Theorem \ref{Up for MO}. We can obtain that
 \begin{align*}
&\int_{\mathbb{R}^{N}}\left|u\mathcal{L}_{M_{1}}[f](u)\right|^{2}\mathrm{d}u\int_{\mathbb{R}^{N}}\left|u\mathcal{L}_{M_{2}}[f](u)\right|^{2}\mathrm{d}u\notag\\
&\ge\frac{\left[\mathrm{tr}\left(A_{1}^{T}B_{2}-A_{2}^{T}B_{1}\right)\right]^{2}}{16\pi^{2}}+\biggr[\int_{\mathbb{R}^{N}}x^{T}A_{2}^{T}A_{1}
x\left|f(x)\right|^{2}\mathrm{d}x\notag\\
&\quad+ \int_{\mathbb{R}^{N}}w^{T}B_{1}^{T}B_{2}
w\left|\widehat{f}(w)\right|^{2}\mathrm{d}w+ \int_{\mathbb{R}^{N}}x^{T}\left(A_{1}^{T}B_{2}+A_{2}^{T}B_{1}
\right)\nabla \varphi(x)\left|f(x)\right|^2\mathrm{d}x
\bigg]^{2},
\end{align*}
which coincides with the result of Theorem \ref{UP for L2 M1 M2}. 
In fact, we have
\begin{align*}
    &\int_{\mathbb{R}^{N}}\left|u\mathcal{L}_{M_{1}}[f](u)\right|^{2}\mathrm{d}u
\int_{\mathbb{R}^{N}}\left|u\mathcal{L}_{M_{2}}[f](u)\right|^{2}\mathrm{d}u\notag\\
&\ge\frac{1}{16\pi^2}\left[\sum_{j=1}^{N}\left(M_{1}JM_{2}^{T}\right)_{jj}\right]^{2}+\left[\sum_{j=1}^{N}\left(M_{1}\Sigma M_{2}^{T}\right)_{jj}\right]^{2}.
\end{align*}
Notice that
\begin{align*}
\frac{1}{16\pi^{2}}\left[\sum_{j=1}^{N}\left(M_{1}JM_{2}^{T}\right)_{jj}\right]^{2}
=&\frac{\left[\mathrm{tr}\left(A_{1}^{T}B_{2}-A_{2}^{T}B_{1}\right)\right]^2}{16\pi^{2}}.
\end{align*}
Using (\ref{Sigma}), we have
\begin{align*}
&\left[\sum_{j=1}^{N}\left(M_{1}\Sigma M_{2}^{T}\right)_{jj}\right]^{2}\notag\\
=&\left[\sum_{j=1}^{N}\left(
A_{1}XA_{2}^{T}+B_{1}WB_{2}^{T}+A_{1}\mathrm{Cov}_{X,W}B_{2}^{T}+B_{1}\left(\mathrm{Cov}_{X,W}\right)^{T}A_{2}^{T}
\right)_{jj}\right]^{2}\notag\\
=&\left[\mathrm{tr}\left(A_{1}XA_{2}^{T}\right)+\mathrm{tr}\left(B_{2}WB_{1}^{T}\right)+\mathrm{tr}\left(A_{1}\mathrm{Cov}_{X,W}B_{2}^{T}\right)+\mathrm{tr}\left(A_{2}\mathrm{Cov}_{X,W}B_{1}^{T}
\right)\right]^{2}.
\end{align*}
From (\ref{first equality})-(\ref{third equality}), we have
\begin{align*}
    \left[\sum_{j=1}^{N}\left(M_{1}\Sigma M_{2}^{T}\right)_{jj}\right]^{2}=&\biggr[\int_{\mathbb{R}^{N}}x^{T}A_{2}^{T}A_{1}
x\left|f(x)\right|^{2}\mathrm{d}x+ \int_{\mathbb{R}^{N}}w^{T}B_{1}^{T}B_{2}
w\left|\widehat{f}(w)\right|^{2}\mathrm{d}w\notag\\
&+\int_{\mathbb{R}^{N}}x^{T}\left(A_{1}^{T}B_{2}+A_{2}^{T}B_{1}
\right)\nabla \varphi(x)\left|f(x)\right|^2\mathrm{d}x
\bigg]^{2}.
\end{align*}

\end{remark}
To obtain stronger types of HPW uncertainty principle for metaplectic operators, we need the following technical lemmas.
\begin{lemma}\label{last lemma}
Let $f(x)=\left|f(x)\right|e^{2\pi i\varphi(x)}$ and $\int_{\mathbb R^{2N}} {(1+|z|^2)\left|W_\sigma f(z)\right|} dz<\infty$. For $j=1,...,N$, there holds
\begin{equation}\label{operator M spatial}
\int_{\mathbb{R}^{N}}\left|x_j\left(\widehat{M}f\right)(x)\right|^{2}\mathrm{d}x=\left(M\Sigma M^T\right)_{jj}. 
\end{equation}
\end{lemma}
    \begin{proof}
        Applying (\ref{symplectic covariance}) and (\ref{operator inequality 3}), we have
\begin{align*}
    \int_{\mathbb{R}^{N}}\left|x_{j}\left(\widehat{M}f\right)(x)\right|^{2}\mathrm{d}x
    =&\int_{\mathbb{R}^{N}}\left|\left(\widehat{M}^{*}\widehat{X}_{j}\widehat{M}f\right)(x)\right|^2\mathrm{d}x\notag\\
=&\int_{\mathbb{R}^{N}}\left|\sum_{\alpha=1}^{2N}M_{j,\alpha}\widehat{Z}_{\alpha}\right|^2\mathrm{d}x\notag\\
=&\int_{\mathbb{R}^{N}}\left|\sum_{\alpha=1}^{N}M_{j,\alpha}x_{\alpha}f(x)+\frac{1}{2\pi i}
\sum_{\alpha=N+1}^{2N}M_{j,\alpha}\frac{\partial f(x)}{\partial x_{\alpha-N}}
\right|^2\mathrm{d}x.
\end{align*}
Similar to proof of (\ref{M_1}), we have
\begin{align*}
    &\int_{\mathbb{R}^{N}}\left|x_{j}\left(\widehat{M}f\right)(x)\right|^{2}\mathrm{d}x\notag\\
=&\int_{\mathbb{R}^{N}}\left|\sum_{\alpha=1}^{N}M_{j,\alpha}x_{\alpha}\left|f(x)\right|+\frac{1}{2\pi i}
\sum_{\alpha=N+1}^{2N}M_{j,\alpha}\frac{\partial \left|f(x)\right|}{\partial x_{\alpha-N}}+\sum_{\alpha=N+1}^{2N}M_{j,\alpha}\frac{\partial \varphi(x)}{\partial x_{\alpha-N}}\left|f(x)\right|
\right|^2\mathrm{d}x\notag\\
=&\sum_{\alpha=1}^{N}M_{j,\alpha}\sum_{\beta=1}^{N}M_{j,\beta}\Delta x_{\alpha,\beta}^2+2\sum_{\alpha=1}^{N}M_{j,\alpha}\sum_{\beta=N+1}^{2N}M_{j,\beta}\mathrm{Cov}_{x,w}^{\alpha, \beta-N}
\notag\\
&+\left[\int_{\mathbb{R}^{N}}\left(\sum_{\alpha=N+1}^{2N}M_{j,\alpha}\frac{\partial \varphi(x)}{\partial x_{\alpha-N}}\left|f(x)\right|
\right)^2\mathrm{d}x+\frac{1}{4\pi^2}\int_{\mathbb{R}^{N}}\left(\sum_{\alpha=N+1}^{2N}M_{j,\alpha}\frac{\partial \left|f(x)\right|}{\partial x_{\alpha-N}}\right)^2\mathrm{d}x
\right].
\end{align*}
One can directly calculate that
\begin{align*}
&\int_{\mathbb{R}^{N}}\left(\sum_{\alpha=N+1}^{2N}M_{j,\alpha}\frac{\partial \varphi(x)}{\partial x_{\alpha-N}}\left|f(x)\right|
\right)^2\mathrm{d}x+\frac{1}{4\pi^2}\int_{\mathbb{R}^{N}}\left(\sum_{\alpha=N+1}^{2N}M_{j,\alpha}\frac{\partial \left|f(x)\right|}{\partial x_{\alpha-N}}\right)^2\mathrm{d}x\notag\\
=&\sum_{\alpha=N+1}^{2N}M_{j,\alpha}\sum_{\beta=N+1}^{2N}M_{j,\beta}\Delta w_{\alpha-N,\beta-N}^2.
\end{align*}
Therefore we have (\ref{operator M spatial}).
    \end{proof}
    \begin{lemma}
    \label{lemma general MO}
 Let $f(x)=\left|f(x)\right|e^{2\pi i\varphi(x)}$ and $\int_{\mathbb R^{2N}} {(1+|z|^2)\left|W_\sigma f(z)\right|} dz<\infty$. Then
\begin{align}\label{Up for general MO}
    &\int_{\mathbb{R}^{N}}\left|x\left(\widehat{M_{1}}f\right)(x)\right|^{2}\mathrm{d}x
\int_{\mathbb{R}^{N}}\left|\xi\left(\widehat{M_{2}}f\right)(\xi)\right|^{2}\mathrm{d}\xi\notag\\
\ge&\left[\sum_{j=1}^{N}\left(
\prod_{k=1}^{2}\left(M_k \Sigma M_k^T\right)_{jj}
\right)^{\frac{1}{2}}
\right]^2.
\end{align}
\end{lemma}
\begin{proof}
     From (\ref{operator M spatial}), we have that for $k=1,2$,
\begin{equation*}
\int_{\mathbb{R}^{N}}\left|x_{j}\left(\widehat{M_k}f\right)(x)\right|^{2}\mathrm{d}x=\left(M_k\Sigma M_k^T\right)_{jj}.
\end{equation*}
        Using Cauchy-Schwartz’s inequality, we have
        \begin{align*}
&\int_{\mathbb{R}^{N}}\left|x\left(\widehat{M_{1}}f\right)(x)\right|^{2}\mathrm{d}x
\int_{\mathbb{R}^{N}}\left|\xi\left(\widehat{M_{2}}f\right)(\xi)\right|^{2}\mathrm{d}\xi\notag\\
\ge&\left[\sum_{j=1}^{N}\left(\int_{\mathbb{R}^{N}}\left|x_j\left(\widehat{M_{1}}f\right)(x)\right|^{2}\mathrm{d}x\int_{\mathbb{R}^{N}}\left|\xi_{j}\left(\widehat{M_{2}}f\right)(\xi)\right|^{2}\mathrm{d}\xi
\right)^{\frac{1}{2}}
\right]^{2}\notag\\
=&\left[\sum_{j=1}^{N}\left(
\prod_{k=1}^{2}\left(M_k \Sigma M_k^T\right)_{jj}
\right)^{\frac{1}{2}}
\right]^2.
        \end{align*}
\end{proof}
    \begin{lemma}\label{Lemma-MO}
   Let $f(x)=\left|f(x)\right|e^{2\pi i\varphi(x)}$ and $\int_{\mathbb R^{2N}} {(1+|z|^2)\left|W_\sigma f(z)\right|} dz<\infty$. Then
        \begin{equation}\label{operator M}
            \int_{\mathbb{R}^{N}}\left|x\left(\widehat{M}f\right)(x)\right|^{2}\mathrm{d}x=\sum_{j=1}^{N}\left(M\Sigma M^T\right)_{jj}.
        \end{equation}
    \end{lemma}
For $f(x)=\left|f(x)\right|e^{2\pi i\varphi(x)}$, the authors in \cite{Dang-Deng-Qian} give the so-called extra-strong uncertainty principle, which is stated as follows,
\begin{equation}\label{COV}
		\Delta x^{2} \Delta w^{2}\ge \frac{1}{16\pi^{2}}+\mathrm{COV}_{x,w}^{2},
	\end{equation}
	where $\mathrm{COV}_{x,w}=\int_{-\infty}^{\infty}|x\varphi^{\prime}(x) |\left|f(x)\right|^2\mathrm{d}x\ge \int_{-\infty}^{\infty}x\varphi^{\prime}(x) \left|f(x)\right|^2\mathrm{d}x=\mathrm{Cov}_{x,w}$.
   In \cite{Dang-Deng-Qian2}, the authors obtain 
   the extra-strong uncertainty principle for LCT based on the result of (\ref{COV}), i.e.,
    \begin{align}\label{extra-stronger}
&\int_{\mathbb{R}}\left|u\mathcal{L}_{M_{1}}[f](u)\right|^{2}\mathrm{d}u\int_{\mathbb{R}}\left|u\mathcal{L}_{M_{2}}[f](u)\right|^{2}\mathrm{d}u\notag\\
&\ge\left(\frac{1}{16\pi^{2}}+\mathrm{COV}_{x,w}^2-\mathrm{Cov}_{x,w}^2\right)\left(a_{1}b_{2}-a_{2}b_{1}\right)^{2}\notag\\
&\quad+\biggr[a_{1}a_{2}\Delta x^{2}+b_{1}b_{2}\Delta w^{2}+\left(a_{1}b_{2}+a_{2}b_{1}\right)\mathrm{Cov}_{x,w}
\bigg]^{2},
\end{align}
where $M_{k}=\begin{pmatrix}
a_{k} & b_{k}\\ 
c_{k} & d_{k}
\end{pmatrix}$ for $k=1,2$. To the authors' best knowledge, it is the best result so far. Since there is also a generalization of (\ref{COV}) in higher dimensional Euclidean spaces (see \cite{Dang-Mai}), it is significant and interesting to expect analogous result of (\ref{extra-stronger}) for metaplectic operators. In the following we prove some generalizations of (\ref{extra-stronger}) for a special class of metaplectic operators.
    \begin{theorem}
    Let $f(x)=\left|f(x)\right|e^{2\pi i\varphi(x)}$ and $\int_{\mathbb R^{2N}} {(1+|z|^2)\left|W_\sigma f(z)\right|} dz<\infty$. If $ M_{1}=\begin{pmatrix}
A_1& B_1\\ 
C_{1} & D_{1}
\end{pmatrix}$ and $M_{2}=\begin{pmatrix}
A_2& B_2\\ 
C_{2} & D_{2}
\end{pmatrix}$, where
$A_{k}=\mathrm{diag}(a_{11}^{(k)},\dots,a_{NN}^{(k)})$ and $
B_{k}=\mathrm{diag}(b_{11}^{(k)},\dots,b_{NN}^{(k)}), k=1,2$, then
\begin{align}\label{Up for diag MO}
    &\int_{\mathbb{R}^{N}}\left|x\left(\widehat{M_{1}}f\right)(x)\right|^{2}\mathrm{d}x
\int_{\mathbb{R}^{N}}\left|\xi\left(\widehat{M_{2}}f\right)(\xi)\right|^{2}\mathrm{d}\xi\notag\\
&\ge \biggr[\sum_{j=1}^{N}\biggr(
\left(\frac{1}{16\pi^2}+{\mathrm{COV}_{x,w}^{j,j}}^2-{\mathrm{Cov}_{x,w}^{j,j}}^2\right)\left|\left(M_1 J M_2^T\right)_{jj}\right|^2+\left|\left(M_{1}\Sigma M_{2}^{T}\right)_{jj}
\right|^2
\bigg)^{\frac{1}{2}}
\bigg]^2.\notag\\
\end{align}
\end{theorem}
\begin{proof}
       Notice that
        \begin{align*}
            \left(M_k \Sigma M_k^T\right)_{jj}=&\left(A_k X A_k^T\right)_{jj}+\left(B_k X B_k^T\right)_{jj}+2\left(A_k\mathrm{Cov}_{X,W}B_k^T\right)_{jj}\notag\\
            =&\left(a_{jj}^{(k)}\right)^2\Delta x_{j,j}^2+\left(b_{jj}^{(k)}\right)^2\Delta w_{j,j}^2+2a_{jj}^{(k)}b_{jj}^{(k)}\mathrm{Cov}_{x,w}^{j,j}.
        \end{align*}
By the method in \cite{Dang-Deng-Qian}, one has 
        \begin{equation*}
            \Delta x_{j,j}^2 \Delta w_{j,j}^2\ge \frac{1}{16\pi^2}+{\mathrm{COV}_{x,w}^{j,j}}^2.
        \end{equation*}
        Thus we have
        \begin{align*}
           & \prod_{k=1}^{2}\left(M_k \Sigma M_k^T\right)_{jj}\notag\\
    =&\left(\Delta x_{j,j}^2 \Delta w_{j,j}^2-{\mathrm{Cov}_{x,w}^{j,j}}^2\right)\left(a_{jj}^{(1)}b_{jj}^{(2)}-a_{jj}^{(2)}b_{jj}^{(1)}\right)^2\notag\\
&+\left[a_{jj}^{(1)}a_{jj}^{(2)}\Delta x_{j,j}^2+b_{jj}^{(1)}b_{jj}^{(2)}\Delta w_{j,j}^2+\left(a_{jj}^{(1)}b_{jj}^{(2)}+a_{jj}^{(2)}b_{jj}^{(1)}\right)\mathrm{Cov}_{x,w}^{j,j}
\right]^2\notag\\
\ge&\left(\frac{1}{16\pi^2}+{\mathrm{COV}_{x,w}^{j,j}}^2-{\mathrm{Cov}_{x,w}^{j,j}}^2\right)\left(a_{jj}^{(1)}b_{jj}^{(2)}-a_{jj}^{(2)}b_{jj}^{(1)}\right)^2\notag\\
&+\left[a_{jj}^{(1)}a_{jj}^{(2)}\Delta x_{j,j}^2+b_{jj}^{(1)}b_{jj}^{(2)}\Delta w_{j,j}^2+\left(a_{jj}^{(1)}b_{jj}^{(2)}+a_{jj}^{(2)}b_{jj}^{(1)}\right)\mathrm{Cov}_{x,w}^{j,j}
\right]^2.
        \end{align*}
         One can calculate that
        \begin{equation*}
            \left(a_{jj}^{(1)}b_{jj}^{(2)}-a_{jj}^{(2)}b_{jj}^{(1)}\right)^2=\left|\left(M_1 J M_2^T\right)_{jj}\right|^{2}
        \end{equation*}
    and
    \begin{equation*}
        \left[a_{jj}^{(1)}a_{jj}^{(2)}\Delta x_{j,j}^2+b_{jj}^{(1)}b_{jj}^{(2)}\Delta w_{j,j}^2+\left(a_{jj}^{(1)}b_{jj}^{(2)}+a_{jj}^{(2)}b_{jj}^{(1)}\right)\mathrm{Cov}_{x,w}^{j,j}
\right]^2=\left|\left(M_{1}\Sigma M_{2}^{T}\right)_{jj}
\right|^2.
    \end{equation*}
        Consequently, by invoking (\ref{Up for general MO}), we have (\ref{Up for diag MO}). 
\end{proof}
\begin{coro}
Let $f(x)=\left|f(x)\right|e^{2\pi i\varphi(x)}, xf(x)$ and $w\widehat{f}(w)\in L^{2}(\mathbb{R}^{N})$. If $A_{k}=\mathrm{diag}(a_{11}^{(k)},\dots,a_{NN}^{(k)})$ and $
B_{k}=\mathrm{diag}(b_{11}^{(k)},\dots,b_{NN}^{(k)})$ for $k=1,2$. Then
    \begin{align}\label{Up for component FMT}
&\int_{\mathbb{R}^{N}}\left|u\mathcal{L}_{M_{1}}[f](u)\right|^{2}\mathrm{d}u\int_{\mathbb{R}^{N}}\left|u\mathcal{L}_{M_{2}}[f](u)\right|^{2}\mathrm{d}u\notag\\
        &\ge \biggr[\sum_{j=1}^{N}\biggr(
\left(\frac{1}{16\pi^2}+{\mathrm{COV}_{x,w}^{j,j}}^2-{\mathrm{Cov}_{x,w}^{j,j}}^2\right)\left(a_{jj}^{(1)}b_{jj}^{(2)}-a_{jj}^{(2)}b_{jj}^{(1)}\right)^2\notag\\
&\quad+\left[a_{jj}^{(1)}a_{jj}^{(2)}\Delta x_{j,j}^2+b_{jj}^{(1)}b_{jj}^{(2)}\Delta w_{j,j}^2+\left(a_{jj}^{(1)}b_{jj}^{(2)}+a_{jj}^{(2)}b_{jj}^{(1)}\right)\mathrm{Cov}_{x,w}^{j,j}
\right]^2
\bigg)^{\frac{1}{2}}
\bigg]^2.
    \end{align}
\end{coro}
\begin{remark}
    When $N=1$, $M_{1}=\begin{pmatrix}
a_{1} & b_{1}\\ 
c_{1} & d_{1}
\end{pmatrix}$ and  $M_{2}=\begin{pmatrix}
a_{2} & b_{2}\\ 
c_{2} & d_{2}
\end{pmatrix}$, we obtain that (\ref{Up for component FMT}) reduces to (\ref{extra-stronger}).
\end{remark}
    We denote by $I_{N}^{+,-}$ any one of the class of $N\times N$ matrices of $1$ or $-1$ for all diagonal
elements and 0 otherwise. Using Lemma \ref{Lemma-MO}, we have the following result.
\begin{theorem}
Let $f(x)=\left|f(x)\right|e^{2\pi i\varphi(x)}$ and $\int_{\mathbb R^{2N}} {(1+|z|^2)\left|W_\sigma f(z)\right|} dz<\infty$. If $ 
M_{1}=\begin{pmatrix}
A_1& B_1\\ 
C_{1} & D_{1}
\end{pmatrix}$ and $M_{2}=\begin{pmatrix}
A_2& B_2\\ 
C_{2} & D_{2}
\end{pmatrix}$, where 
$A_k=a_k I_{N}^{+,-}$ and $B_k=b_k I_{N}^{+,-}$ for $k=1,2$, then
    \begin{align}\label{UP for metaplectic operator special}
        &\int_{\mathbb{R}^{N}}\left|x\left(\widehat{M_{1}}f\right)(x)\right|^{2}\mathrm{d}x
\int_{\mathbb{R}^{N}}\left|\xi\left(\widehat{M_{2}}f\right)(\xi)\right|^{2}\mathrm{d}\xi\notag\\
&\ge \left(\frac{1}{16\pi^2}+\frac{\mathrm{COV}_{x,w}^2-\mathrm{Cov}_{x,w}^2}{N^2}\right)\left[\sum_{j=1}^{N}\left(M_{1}JM_{2}^{T}\right)_{jj}\right]^{2}+\left[\sum_{j=1}^{N}\left(M_{1}\Sigma M_{2}^{T}\right)_{jj}\right]^{2}.\notag\\
    \end{align}
    \end{theorem}
    \begin{proof}
      Using (\ref{operator M}), we have 
\begin{equation*}
    \int_{\mathbb{R}^{N}}\left|x\left(\widehat{M_{1}}f\right)(x)\right|^{2}\mathrm{d}x=
a_1^2\Delta x^2+b_1^2\Delta w^2+2a_1b_1\mathrm{Cov}_{x,w}
\end{equation*}
and
\begin{equation*}
\int_{\mathbb{R}^{N}}\left|\xi\left(\widehat{M_{2}}f\right)(\xi)\right|^{2}\mathrm{d}\xi=a_2^2\Delta x^2+b_2^2\Delta w^2+2a_2b_2\mathrm{Cov}_{x,w}.
\end{equation*}
Thus we have
  \begin{align*}
&\int_{\mathbb{R}^{N}}\left|x\left(\widehat{M_{1}}f\right)(x)\right|^{2}\mathrm{d}x\int_{\mathbb{R}^{N}}\left|\xi\left(\widehat{M_{2}}f\right)(\xi)\right|^{2}\mathrm{d}\xi\notag\\
=& \left(\Delta x^{2}\Delta w^{2}-\mathrm{Cov}_{x,w}^{2}\right)\left(a_1b_2-a_2b_1\right)^{2}\notag\\
&\quad+\bigg[a_1a_2\Delta x^{2}+b_1b_2\Delta w^{2}+\left(a_1b_2+a_2b_1\right)\mathrm{Cov}_{x,w}\bigg]^{2}\notag\\
\ge& \left(\frac{N^{2}}{16\pi^{2}}+\mathrm{COV}_{x,w}^{2}-\mathrm{Cov}_{x,w}^{2}\right)\left(a_1b_2-a_2b_1\right)^{2}\notag\\
&\quad+\bigg[a_1a_2\Delta x^{2}+b_1b_2\Delta w^{2}+\left(a_1b_2+a_2b_1\right)\mathrm{Cov}_{x,w}\bigg]^{2},
    \end{align*}
    where the last inequality follows from the fact that
\begin{equation*}
    \Delta x^{2}\Delta w^{2}\ge \frac{N^{2}}{16\pi^{2}}+\mathrm{COV}_{x,w}^{2}
\end{equation*}
(\cite{Dang-Mai}).
Since
\begin{equation*}
   N^2 \left(a_1b_2-a_2b_1\right)^2=\left[\sum_{j=1}^{N}\left(M_{1}JM_{2}^{T}\right)_{jj}\right]^{2}
\end{equation*}
and
\begin{equation*}
    \bigg[a_1a_2\Delta x^{2}+b_1b_2\Delta w^{2}+\left(a_1b_2+a_2b_1\right)\mathrm{Cov}_{x,w}\bigg]^{2}
=\left[\sum_{j=1}^{N}\left(M_{1}\Sigma M_{2}^{T}\right)_{jj}\right]^{2},
\end{equation*}
we have the desired inequality (\ref{UP for metaplectic operator special}). 
    \end{proof}
     \begin{remark}
    In the following we compare the lower bounds of (\ref{Up for diag MO}) and (\ref{UP for metaplectic operator special}) for $A_k=a_k I_{N}^{+,-}, B_k=b_k I_{N}^{+,-}, k=1,2$. Then (\ref{Up for diag MO}) and (\ref{UP for metaplectic operator special}) can be reduced to 
      \begin{align}\label{Up for special MO}
          &\int_{\mathbb{R}^{N}}\left|x\left(\widehat{M_{1}}f\right)(x)\right|^{2}\mathrm{d}x
\int_{\mathbb{R}^{N}}\left|\xi\left(\widehat{M_{2}}f\right)(\xi)\right|^{2}\mathrm{d}\xi\notag\\
&\ge \biggr[\sum_{j=1}^{N}\biggr(
\left(\frac{1}{16\pi^2}+{\mathrm{COV}_{x,w}^{j,j}}^2-{\mathrm{Cov}_{x,w}^{j,j}}^2\right)\left(a_{1}b_{2}-a_{2}b_{1}\right)^2\notag\\
&\quad+\Bigr[a_{1}a_{2}\Delta x_{j,j}^2+b_{1}b_{2}\Delta w_{j,j}^2+\left(a_{1}b_{2}+a_{2}b_{1}\right)\mathrm{Cov}_{x,w}^{j,j}
\Big]^2
\bigg)^{\frac{1}{2}}
\bigg]^2
      \end{align}
      and 
       \begin{align}\label{Up for signal diag MO}
&\int_{\mathbb{R}^{N}}\left|x\left(\widehat{M_{1}}f\right)(x)\right|^{2}\mathrm{d}x\int_{\mathbb{R}^{N}}\left|\xi\left(\widehat{M_{2}}f\right)(\xi)\right|^{2}\mathrm{d}\xi\notag\\
\ge& \left(\frac{N^{2}}{16\pi^{2}}+\mathrm{COV}_{x,w}^{2}-\mathrm{Cov}_{x,w}^{2}\right)\left(a_1b_2-a_2b_1\right)^{2}\notag\\
&\quad+\bigg[a_1a_2\Delta x^{2}+b_1b_2\Delta w^{2}+\left(a_1b_2+a_2b_1\right)\mathrm{Cov}_{x,w}\bigg]^{2}\notag\\
=&J_1\left(a_{1}b_{2}-a_{2}b_{1}\right)^2+J_2
    \end{align}
      Using the Minkowski inequality, we have
      \begin{align*}
      &\biggr[\sum_{j=1}^{N}\biggr(
\left(\frac{1}{16\pi^2}+{\mathrm{COV}_{x,w}^{j,j}}^2-{\mathrm{Cov}_{x,w}^{j,j}}^2\right)\left(a_{1}b_{2}-a_{2}b_{1}\right)^2\notag\\
&\quad+\Bigr[a_{1}a_{2}\Delta x_{j,j}^2+b_{1}b_{2}\Delta w_{j,j}^2+\left(a_{1}b_{2}+a_{2}b_{1}\right)\mathrm{Cov}_{x,w}^{j,j}
\Big]^2
\bigg)^{\frac{1}{2}}
\bigg]^2\notag\\
          \ge&\biggr[\sum_{j=1}^{N}
\left(\frac{1}{16\pi^2}+{\mathrm{COV}_{x,w}^{j,j}}^2-{\mathrm{Cov}_{x,w}^{j,j}}^2\right)^{\frac{1}{2}}
\bigg]^2\left(a_{1}b_{2}-a_{2}b_{1}\right)^2\notag\\
&+\biggr[\sum_{j=1}^{N}\Bigr|a_1a_2 \Delta x_{j,j}^2+b_1b_2 \Delta w_{j,j}^2+\left(a_1b_2+a_2b_1\right)\mathrm{Cov}_{x,w}^{j,j}\Big|\bigg]^2\notag\\
=&K_1\left(a_{1}b_{2}-a_{2}b_{1}\right)^2+K_2.
      \end{align*}
          Clearly, one has that
\begin{equation*}
K_2
\ge\biggr[a_1a_2 \Delta x^2+b_1b_2 \Delta w^2+\left(a_1b_2+a_2b_1\right)\mathrm{Cov}_{x,w}
\bigg]^2=J_2.
\end{equation*}
So we just need to compare $J_1$ and $K_1$.
Note that
$\mathrm{Cov}_{x,w}=\sum_{j=1}^{N}\mathrm{Cov}_{x,w}^{j,j},$
\begin{align*}
\mathrm{COV}_{x,w}
=&\int_{\mathbb{R}^{N}}\left(\sum_{j=1}^{N}x_j^2\right)^{\frac{1}{2}}\left(\sum_{j=1}^{N}\left(\frac{\partial \varphi(x)}{\partial x_j}\right)^2\right)^{\frac{1}{2}}\left|f(x)\right|^2\mathrm{d}x\notag\\
    \ge&\sum_{j=1}^{N}\int_{\mathbb{R}^{N}}\left|x_j \frac{\partial \varphi(x)}{\partial x_j}\right|\left|f(x)\right|^2\mathrm{d}x=\sum_{j=1}^{N}\mathrm{COV}_{x,w}^{j,j}
\end{align*}
and 
\begin{equation*}
    K_1\ge\frac{N^2}{16\pi^2}+\left[\sum_{j=1}^{N}\left({\mathrm{COV}_{x,w}^{j,j}}^2-{\mathrm{Cov}_{x,w}^{j,j}}^2\right)^{\frac{1}{2}}\right]^2.
\end{equation*}
Hence we have
\begin{align*}
     K_1-J_1\ge&\left[\sum_{j=1}^{N}\left({\mathrm{COV}_{x,w}^{j,j}}^2-{\mathrm{Cov}_{x,w}^{j,j}}^2\right)^{\frac{1}{2}}\right]^2-\left(\mathrm{COV}_{x,w}^{2}-\mathrm{Cov}_{x,w}^{2}\right)\notag\\
=&\sum_{\substack{j,k=1\\j\neq k}}^N\left({\mathrm{COV}_{x,w}^{j,j}}^2-{\mathrm{Cov}_{x,w}^{j,j}}^2\right)^{\frac{1}{2}}\left({\mathrm{COV}_{x,w}^{k,k}}^2-{\mathrm{Cov}_{x,w}^{k,k}}^2\right)^{\frac{1}{2}}\notag\\
&+\sum_{j=1}^{N}{\mathrm{COV}_{x,w}^{j,j}}^2
-\mathrm{COV}_{x,w}^2+\sum_{\substack{j,k=1\\j\neq k}}^N\mathrm{Cov}_{x,w}^{j,j}\mathrm{Cov}_{x,w}^{k,k}.
\end{align*}
It seems that (\ref{Up for special MO}) and (\ref{Up for signal diag MO}) are not comparable.
    \end{remark}
    \begin{coro}\label{Corollary-up for FMT}
Let $f(x)=\left|f(x)\right|e^{2\pi i \varphi(x)}, xf(x)$ and $w\widehat{f}(w)
\in L^{2}(\mathbb{R}^{N})$. If $A_k=a_k I_{N}^{+,-}$ and $B_k=b_k I_{N}^{+,-}$ for $k=1,2$, then
    \begin{align}\label{result for special case}
&\int_{\mathbb{R}^{N}}\left|u\mathcal{L}_{M_{1}}[f](u)\right|^{2}\mathrm{d}u\int_{\mathbb{R}^{N}}\left|u\mathcal{L}_{M_{2}}[f](u)\right|^{2}\mathrm{d}u\notag\\
&\ge \left(\frac{N^{2}}{16\pi^{2}}+\mathrm{COV}_{x,w}^{2}-\mathrm{Cov}_{x,w}^{2}\right)\left(a_1 b_2-a_2 b_1\right)^{2}\notag\\
&\quad+\biggr[a_1a_2\Delta x^{2}+b_1 b_2\Delta w^{2}+\left(a_1b_2+a_2b_1\right)\mathrm{Cov}_{x,w}\bigg]^{2}.
    \end{align}
\end{coro}
\begin{remark}
    When $N=1$, $M_{1}=\begin{pmatrix}
a_{1} & b_{1}\\ 
c_{1} & d_{1}
\end{pmatrix}$ and  $M_{2}=\begin{pmatrix}
a_{2} & b_{2}\\ 
c_{2} & d_{2}
\end{pmatrix}$, (\ref{result for special case}) can be reduced to (\ref{extra-stronger}).
\end{remark}
\begin{remark}
Here we denote by $\mu({A_{k}})$ and $\mu({B_{k}})$ the singular values of $A_{k}$ and $B_{k}$ for $k=1,2$. Let
$A_k=\mu({A_k})I_{N}^{+,-}$ and $B_k=\mu({B_k})I_{N}^{+,-}$ in Corollary \ref{Corollary-up for FMT}. Then (\ref{result for special case}) becomes
 \begin{align}\label{Up for LCT singular}
&\int_{\mathbb{R}^{N}}\left|u\mathcal{L}_{M_{1}}[f](u)\right|^{2}\mathrm{d}u\int_{\mathbb{R}^{N}}\left|u\mathcal{L}_{M_{2}}[f](u)\right|^{2}\mathrm{d}u\notag\\
&\ge \left(\frac{N^{2}}{16\pi^{2}}+\mathrm{COV}_{x,w}^{2}-\mathrm{Cov}_{x,w}^{2}\right)\left(\mu({A_{1}})\mu({B_{2}})-\mu({A_{2}})\mu({B_{1}})\right)^{2}\notag\\
&\quad+\biggr[\mu({A_{1}})\mu({A_{2}})\Delta x^{2}+\mu({B_{1}})\mu({B_{2}})\Delta w^{2}+\left(\mu({A_{1}})\mu({B_{2}})+\mu({A_{2}})\mu({B_{1}})\right)\mathrm{Cov}_{x,w}\bigg]^{2}.\notag\\
    \end{align}
In the following, one can verify that the lower bound of (\ref{Up for LCT singular}) is sharper than that in \cite{Zhang2}.
     In \cite{Zhang2}, the author obtains the following result,
\begin{align}\label{result for special case 2}
    &\int_{\mathbb{R}^{N}}\left|u\mathcal{L}_{M_{1}}[f](u)\right|^{2}\mathrm{d}u\int_{\mathbb{R}^{N}}\left|u\mathcal{L}_{M_{2}}[f](u)\right|^{2}\mathrm{d}u\notag\\
&\ge \left(\frac{N^{2}}{16\pi^{2}}+\mathrm{COV}_{x,w}^{2}-|\mathrm{Cov}|_{x,w}^{2}\right)\left(\mu({A_{1}})\mu({B_{2}})-\mu({A_{2}})\mu({B_{1}})\right)^{2}\notag\\
&\quad+\bigg[\mu({A_{1}})\mu({A_{2}})\Delta x^{2}+\mu({B_{1}})\mu({B_{2}})\Delta w^{2}-\left(\mu({A_{1}})\mu({B_{2}})+\mu({A_{2}})\mu({B_{1}})\right)|\mathrm{Cov}|_{x,w}\bigg]^{2},\notag\\
\end{align}
    where 
    \begin{equation*}
        |\mathrm{Cov}|_{x,w}=\sum_{j=1}^{N}\left|\int_{\mathbb{R}^{N}}x_{j}
        \frac{\partial \varphi(x)}{\partial x_{j}}
        \left|f(x)\right|^2\mathrm{d}x\right|.
    \end{equation*}
It is obvious that $\mathrm{COV}_{x,w}\ge |\mathrm{Cov}|_{x,w}\ge|\mathrm{Cov}_{x,w}|$. Therefore one has that the lower bounds of (\ref{Up for LCT singular}) is stronger than that of (\ref{result for special case 2}).
\end{remark}
\begin{remark}
Based on subsequent research and combined with the results calculated in the paper, we find that Proposition \ref{Ro-ScUP} implies the previous results. Recall that the result of Proposition \ref{Ro-ScUP} is
\begin{equation}\label{RSUP}
    \Upsilon+\frac{i}{4\pi}\Omega\ge 0,
\end{equation}
where $\Upsilon=D_{1,2}\Sigma\left(D_{1,2}\right)^T$,
$\Omega=D_{1,2}J\left(D_{1,2}\right)^T$ with $D_{1,2}=\begin{pmatrix}
    A_1 &B_1\\
    A_2 &B_2
\end{pmatrix},$ and where $A_j$ and $B_j$ are real $N\times N$ matrices from  $M_j\in\mathrm{Sp}(2N,\mathbb{R})$ with $M_j=\begin{pmatrix}
    A_j& B_j\\
    C_j&D_j
\end{pmatrix}, j=1,2.$ Here the inequality (\ref{RSUP}) means that $\Upsilon+\frac{i}{4\pi}\Omega$ is positive semi-definite.
For $f(x)=\left|f(x)\right|e^{2\pi i \varphi(x)}$, by (\ref{Sigma}) one has that 
\begin{equation*}
    \Upsilon+\frac{i}{4\pi}\Omega =\begin{pmatrix}
        P&Q\\
        S&T
    \end{pmatrix},
\end{equation*}
where
\begin{align*}
    P=&A_1XA_1^T+B_1WB_1^T+A_1\mathrm{Cov}_{X,W}B_1^T+B_1\left(\mathrm{Cov}_{X,W}\right)^TA_1^T,\notag\\
    Q=&A_1XA_2^T+B_1WB_2^T+A_1\mathrm{Cov}_{X,W}B_2^T+B_1\left(\mathrm{Cov}_{X,W}\right)^TA_2^T+\frac{i}{4\pi}\left(A_1B_2^T-B_1A_2^T\right),\notag\\
    S=&A_2XA_1^T+B_2WB_1^T+A_2\mathrm{Cov}_{X,W}B_1^T+B_2\left(\mathrm{Cov}_{X,W}\right)^TA_1^T-\frac{i}{4\pi}\left(B_2A_1^T-A_2B_1^T\right)\\
\text{and} \\
T=&A_2XA_2^T+B_2WB_2^T+A_2\mathrm{Cov}_{X,W}B_2^T+B_2\left(\mathrm{Cov}_{X,W}\right)^TA_2^T.
\end{align*}
Let $\left(P\right)_{jj}, \left(Q\right)_{jj}, \left(S\right)_{jj}$ and $\left(T\right)_{jj}$ be the $(j,j)$-th diagonal element of $P, Q, S$ and $T$ for $j=1,\cdots N$, respectively.
If the matrix inequality (\ref{RSUP}) holds, we have
\begin{equation*}
   \Gamma_j = \begin{pmatrix}
        \left(P\right)_{jj}&\left(Q\right)_{jj}\\
        \left(S\right)_{jj}&\left(T\right)_{jj}
    \end{pmatrix}\ge 0.
\end{equation*}
Notice that
\begin{align*}
    \left(P\right)_{jj}=&
    \left(A_1XA_1^T+B_1WB_1^T+A_1\mathrm{Cov}_{X,W}B_1^T+B_1\left(\mathrm{Cov}_{X,W}\right)^TA_1^T\right)_{jj}
    \notag\\
=&\sum_{k=1}^{N}\left(A_1\right)_{jk} \sum_{l=1}^{N}\left(A_1\right)_{jl} \Delta x_{k,l}^2+\sum_{k=1}^{N}\left(B_1\right)_{jk} \sum_{l=1}^{N}\left(B_1\right)_{jl} \Delta w_{k,l}^2\notag\\
&+2\sum_{k=1}^{N}\left(A_1\right)_{jk} \sum_{l=1}^{N}\left(B_1\right)_{jl} \mathrm{Cov}_{x,w}^{k,l}.
\end{align*}
A direct computation yields
\begin{align*}
    \sum_{k=1}^{N}\left(B_1\right)_{jk} \sum_{l=1}^{N}\left(B_1\right)_{jl} \Delta w_{k,l}^2
    =&\frac{1}{4\pi^2}\sum_{k=1}^{N}\left(B_1\right)_{jk} \sum_{l=1}^{N}\left(B_1\right)_{jl}\int_{\mathbb{R}^N}\frac{\partial \left|f(x)\right|}{\partial x_k}\frac{\partial \left|f(x)\right|}{\partial x_l}\mathrm{d}x\notag\\
    &+\sum_{k=1}^{N}\left(B_1\right)_{jk} \sum_{l=1}^{N}\left(B_1\right)_{jl}\int_{\mathbb{R}^N}\frac{\partial \varphi(x)}{\partial x_k}\frac{\partial \varphi(x)}{\partial x_l}\left|f(x)\right|^2\mathrm{d}x
\end{align*}
From (\ref{M_1}), we have
\begin{equation*}
\left(P\right)_{jj}=\int_{\mathbb{R}^{N}}\left|u_j\mathcal{L}_{M_{1}}[f](u)\right|^{2}\mathrm{d}u.
\end{equation*}
Similarly, by (\ref{M_2}) we have
\begin{equation*}
    \left(T\right)_{jj}
    =\int_{\mathbb{R}^{N}}\left|u_j\mathcal{L}_{M_{2}}[f](u)\right|^{2}\mathrm{d}u.
\end{equation*}
Since 
\begin{equation*}
    \left(P\right)_{jj}+\left(T\right)_{jj}=\int_{\mathbb{R}^{N}}\left|u_j\mathcal{L}_{M_{1}}[f](u)\right|^{2}\mathrm{d}u+\int_{\mathbb{R}^{N}}\left|u_j\mathcal{L}_{M_{2}}[f](u)\right|^{2}\mathrm{d}u\ge 0,
\end{equation*}
we have that the matrix inequality $\Gamma _{j}\ge 0$ satisfy if and only if
\begin{equation*}
\left(P\right)_{jj}\left(T\right)_{jj}\ge\left(Q\right)_{jj}\left(S\right)_{jj},
\end{equation*}
that is 
\begin{align*}
&\int_{\mathbb{R}^{N}}\left|u_j\mathcal{L}_{M_{1}}[f](u)\right|^{2}\mathrm{d}u\int_{\mathbb{R}^{N}}\left|u_j\mathcal{L}_{M_{2}}[f](u)\right|^{2}\mathrm{d}u\notag\\
    \ge&\frac{1}{16\pi^2}\left|\left(A_1B_2^T-B_1 A_2^T\right)_{jj}\right|^2+\Bigr|
    \Bigr(A_1 XA_2^T+B_1 W B_2^T+A_1\mathrm{Cov}_{X,W}B_2^T\notag\\
&+B_1\left(\mathrm{Cov}_{X,W}\right)^TA_2^T\Big)_{jj}
    \Big|^2.
    \end{align*}
Combining the Cauchy-Schwartz inequality, we have
\begin{align*}
&\int_{\mathbb{R}^{N}}\left|u\mathcal{L}_{M_{1}}[f](u)\right|^{2}\mathrm{d}u
\int_{\mathbb{R}^{N}}\left|u\mathcal{L}_{M_{2}}[f](u)\right|^{2}\mathrm{d}u\notag\\
&\ge\biggr[\sum_{j=1}^{N}\biggr(\frac{1}{16\pi^{2}}\left|\left(A_{1}B_{2}^{T}-B_{1}A_{2}^{T}\right)_{jj}\right|^{2}+\Bigr|\Bigr(A_{1}XA_{2}^{T}+B_{1}WB_{2}^{T}\notag\\
&\quad+A_{1}\mathrm{Cov}_{X,W}B_{2}^{T}+B_{1}\left(\mathrm{Cov}_{X,W}\right)^{T}A_{2}^{T}\Big)_{jj}
\Big|^{2}
\bigg)^{\frac{1}{2}}
\bigg]^{2},
\end{align*}
which coincide with the result of Theorem \ref{Up for MO component}.
Furthermore, we have
    $\sum_{j=1}^{N} \Gamma_{j}\ge0$.
    Hence we have
    \begin{align*}
        &\int_{\mathbb{R}^{N}}\left|u\mathcal{L}_{M_{1}}[f](u)\right|^{2}\mathrm{d}u\int_{\mathbb{R}^{N}}\left|u\mathcal{L}_{M_{2}}[f](u)\right|^{2}\mathrm{d}u\notag\\
    \ge&\frac{\left[\mathrm{tr}\left(A_{1}B_{2}^{T}-B_{1}A_{2}^{T}\right)\right]^2}{16\pi^{2}}\notag\\
        &+\left[\mathrm{tr}\left(A_{1}XA_{2}^{T}+B_{1}WB_{2}^{T}+A_{1}\mathrm{Cov}_{X,W}B_{2}^{T}+B_{1}\left(\mathrm{Cov}_{X,W}\right)^TA_{2}^{T}
\right)\right]^{2}\notag\\
       =&\frac{\left[\mathrm{tr}\left(A_{1}^{T}B_{2}-A_{2}^{T}B_{1}\right)\right]^2}{16\pi^{2}}\notag\\
        &+\left[\mathrm{tr}\left(A_{1}XA_{2}^{T}\right)+\mathrm{tr}\left(B_{2}WB_{1}^{T}\right)+\mathrm{tr}\left(A_{1}\mathrm{Cov}_{X,W}B_{2}^{T}\right)+\mathrm{tr}\left(A_{2}\mathrm{Cov}_{X,W}B_{1}^{T}
\right)\right]^{2}.
    \end{align*}
Combining (\ref{first equality})-(\ref{third equality}), we have 
 \begin{align*}
&\int_{\mathbb{R}^{N}}\left|u\mathcal{L}_{M_{1}}[f](u)\right|^{2}\mathrm{d}u\int_{\mathbb{R}^{N}}\left|u\mathcal{L}_{M_{2}}[f](u)\right|^{2}\mathrm{d}u\notag\\
&\ge\frac{\left[\mathrm{tr}\left(A_{1}^{T}B_{2}-A_{2}^{T}B_{1}\right)\right]^{2}}{16\pi^{2}}+\biggr[\int_{\mathbb{R}^{N}}x^{T}A_{2}^{T}A_{1}
x\left|f(x)\right|^{2}\mathrm{d}x\notag\\
&\quad+ \int_{\mathbb{R}^{N}}w^{T}B_{1}^{T}B_{2}
w\left|\widehat{f}(w)\right|^{2}\mathrm{d}w+ \int_{\mathbb{R}^{N}}x^{T}\left(A_{1}^{T}B_{2}+A_{2}^{T}B_{1}
\right)\nabla \varphi(x)\left|f(x)\right|^2\mathrm{d}x
\bigg]^{2},
\end{align*}
    which is the result of Theorem \ref{UP for L2 M1 M2}.
\end{remark}

\subsection*{Acknowledgment. }
W. X. Mai was supported by the Science and Technology Development Fund, Macau SAR (No. 0133/2022/A). P. Dang was supported by the Science and Technology Development Fund, Macau SAR (No. 0067/2024/RIA1).

\subsection*{Data availability statement. }
Data sharing is not applicable to this article as no new data were created or analyzed in this study.
\subsection*{Conflict of interest}
The authors declared that they have no competing interest regarding this research work.

\appendix
	\section{Proof of Lemmas \ref{the first lemma} and \ref{the second lemma}}\label{appendix}
    \begin{Proof}
Let $g(x)=f(x)e^{\pi i x^{T}B_{1}^{-1}A_{1}x}$.
According to (\ref{FMT and FT}), we have 
\begin{equation*}
    \mathcal{L}_{M_{1}}[f](u)=\frac{e^{\pi i u^{T}D_{1}B_{1}^{-1}u}}{i^{\frac{N}{2}}\sqrt{\mathrm{det}(B_{1})}}\widehat{g}\left(B_{1}^{-1}u\right).
\end{equation*}
Since $B_1$ and $D_1$ satisfy (\ref{conditions 3}),
one has 
\begin{align*}
&\nabla \mathcal{L}_{M_{1}}[f](u)\\
=&\frac{1}{i^{\frac{N}{2}}\sqrt{\mathrm{det}{(B_{1})}}}\left[2\pi i D_{1}B_{1}^{-1} ue^{\pi i u^{T}D_{1}B_{1}^{-1}u}\widehat{g}\left(B_{1}^{-1}u\right)+e^{\pi i u^{T}D_{1}B_{1}^{-1}u} \nabla \widehat{g}\left(B_{1}^{-1}u\right)
\right].
\end{align*}
Let $B_{3}=B_{2}A_{1}^{T}-A_{2}B_{1}^{T}$.
Therefore we have
\begin{align}\label{I1+I2}
&i\int_{\mathbb{R}^{N}}u^{T}\mathcal{L}_{M_{1}}[f](u)\left(B_{2}A_{1}^{T}-A_{2}B_{1}^{T}\right)\overline{\nabla \mathcal{L}_{M_{1}}[f](u)}\mathrm{d}u\notag\\
=&i\int_{\mathbb{R}^{N}}u^{T}\mathcal{L}_{M_{1}}[f](u)B_{3}\overline{\nabla \mathcal{L}_{M_{1}}[f](u)}\mathrm{d}u\notag\\
=&I_{1}+I_{2},
\end{align}
where 
\begin{equation*}
    I_{1}=\frac{2\pi}{\left|\mathrm{det}(B_{1})\right|}\int_{\mathbb{R}^{N}}u^{T}B_{3}D_{1}B_{1}^{-1}u\left|\widehat{g}\left(B_{1}^{-1}u\right)\right|^{2}\mathrm{d}u
\end{equation*}
and 
\begin{equation*}
     I_{2}=\frac{i}{\left|\mathrm{det}(B_{1})\right|}\int_{\mathbb{R}^{N}}u^{T}B_{3}\widehat{g}\left(B_{1}^{-1}u\right)\overline{\nabla \widehat{g}\left(B_{1}^{-1}u\right)}\mathrm{d}u.
\end{equation*}
 Note that
\begin{equation*}
I_1=2\pi\int_{\mathbb{R}^{N}}u^{T}B_{3}D_{1}B_{1}^{-1}u\left|\mathcal{L}_{M_{1}}[f](u)\right|^{2}\mathrm{d}u.
\end{equation*}
and 
\begin{equation}\label{grident g}
     \nabla g(x)=\nabla f(x)e^{\pi i x^{T}B_{1}^{-1}A_{1}x}+2\pi i B_{1}^{-1}A_{1}xf(x)e^{\pi i x^{T}B_{1}^{-1}A_{1}x},
\end{equation}
Similar to proof of Proposition \ref{prop-LM}, one has
\begin{align*}
    I_{1}=&\frac{1}{2\pi}\int_{\mathbb{R}^{N}}\left(\nabla f(x)\right)^{T}Q \overline{\nabla f(x)}\mathrm{d}x+2\pi \int_{\mathbb{R}^{N}}x^{T}B_{1}^{-1}A_{1}QB_{1}^{-1}A_{1}x\left|f(x)\right|^{2}\mathrm{d}x\\
    &-i\int_{\mathbb{R}^{N}}x^{T}B_{1}^{-1}A_{1}\left[Q^{T} \nabla f(x) \overline{f(x)}-Q\overline{\nabla f(x)}f(x)\right]\mathrm{d}x,
\end{align*}
where $Q=B_{1}^{T}B_{3}D_{1}$. 
Applying (\ref{nabla f}), we have
\begin{align*}
   &i\int_{\mathbb{R}^{N}}x^{T}B_{1}^{-1}A_{1}\left[Q^{T} \nabla f(x) \overline{f(x)}-Q\overline{\nabla f(x)}f(x)\right]\mathrm{d}x\notag\\
=&-i\int_{\mathbb{R}^{N}}x^{T}B_{1}^{-1}A_{1}(Q-Q^{T})\nabla \left|f(x)\right| \left|f(x)\right|\mathrm{d}x\notag\\
&-2\pi \int_{\mathbb{R}^{N}}x^{T}B_{1}^{-1}A_{1}(Q+Q^{T})\nabla \varphi(x)\left|f(x)\right|^{2}\mathrm{d}x\notag\\
=&\frac{i}{2}\mathrm{tr}\left(B_{1}^{-1}A_{1}\left(Q-Q^{T}\right)\right)-2\pi \int_{\mathbb{R}^{N}}x^{T}B_{1}^{-1}A_{1}(Q+Q^{T})\nabla \varphi(x)\left|f(x)\right|^{2}\mathrm{d}x.
\end{align*}
Since
\begin{equation*}
    \frac{1}{2\pi}\int_{\mathbb{R}^{N}}\left(\nabla f(x)\right)^{T}Q \overline{\nabla f(x)}\mathrm{d}x
=2\pi \int_{\mathbb{R}^{N}}w^{T}Q w\left|\widehat{f}(w)\right|^{2}\mathrm{d}w,
\end{equation*}
we have
\begin{align}\label{I1}
     I_1=&2\pi \int_{\mathbb{R}^{N}}w^{T}Q w\left|\widehat{f}(w)\right|^{2}\mathrm{d}w+2\pi \int_{\mathbb{R}^{N}}x^{T}B_{1}^{-1}A_{1}QB_{1}^{-1}A_{1}x\left|f(x)\right|^{2}\mathrm{d}x\notag\\
    &-\frac{i}{2}\mathrm{tr}\left(B_{1}^{-1}A_{1}\left(Q-Q^{T}\right)\right)+2\pi \int_{\mathbb{R}^{N}}x^{T}B_{1}^{-1}A_{1}(Q+Q^{T})\nabla \varphi(x)\left|f(x)\right|^{2}\mathrm{d}x.
\end{align}
By $B_1^{-1}u=w$ and (\ref{grident g}), one has that
\begin{align*}
    I_{2}
=&i\int_{\mathbb{R}^{N}}w^{T}\widehat{g}(w)B_{1}^{T}B_{3}B_{1}^{-T}\overline{\nabla \widehat{g}(w)}\mathrm{d}w\notag\\
    =&\frac{1}{2\pi}\int_{\mathbb{R}^{N}}\left(\nabla g(x)\right)^{T} B_{1}^{T}B_{3}B_{1}^{-T} \overline{-2\pi ixg(x)}\mathrm{d}x\\
    =&i\int_{\mathbb{R}^{N}}\left(\nabla g(x)\right)^{T} B_{1}^{T}B_{3}B_{1}^{-T} x\overline{g(x)}\mathrm{d}x
    \\
=&i\int_{\mathbb{R}^{N}} x^{T}B_{1}^{-1}B_{3}^{T}B_{1}\nabla g(x) \overline{g(x)}\mathrm{d}t\notag\\
=&i\int_{\mathbb{R}^{N}}x^{T}B_{1}^{-1}B_{3}^{T}B_{1}\nabla f(x) \overline{f(x)}\mathrm{d}x-2\pi \int_{\mathbb{R}^{N}}x^{T}B_{1}^{-1}B_{3}^{T}A_{1}x\left|f(x)\right|^{2}\mathrm{d}x.
\end{align*}
Using (\ref{nabla f}),
we have
\begin{align}\label{I2}
    I_{2}
=&i\int_{\mathbb{R}^{N}}x^{T}B_{1}^{-1}B_{3}^{T}B_{1}\nabla\left|f(x)\right| \left|f(x)\right|\mathrm{d}x-2\pi \int_{\mathbb{R}^{N}}x^{T}B_{1}^{-1}B_{3}^{T}B_{1}\nabla \varphi(x)\left|f(x)\right|^{2}\mathrm{d}x\notag\\
&-2\pi \int_{\mathbb{R}^{N}}x^{T}B_{1}^{-1}B_{3}^{T}A_{1}x\left|f(x)\right|^{2}\mathrm{d}x\notag\\
=&-\frac{i}{2}\mathrm{tr}\left(B_{1}^{-1}B_{3}^{T}B_{1}\right)-2\pi \int_{\mathbb{R}^{N}}x^{T}B_{1}^{-1}B_{3}^{T}B_{1}\nabla \varphi(x)\left|f(x)\right|^{2}\mathrm{d}x\notag\\
&-2\pi \int_{\mathbb{R}^{N}}x^{T}B_{1}^{-1}B_{3}^{T}A_{1}x\left|f(x)\right|^{2}\mathrm{d}x.\notag\\
\end{align}
Combining (\ref{I1+I2}), (\ref{I1}) and (\ref{I2}), we have
\begin{align}\label{I} 
&i\int_{\mathbb{R}^{N}}u^{T}\mathcal{L}_{M_{1}}[f](u)B_{3}\overline{\nabla \mathcal{L}_{M_{1}}[f](u)}\mathrm{d}u\notag\\
=&-\frac{i}{2}\mathrm{tr}\left(B_{1}^{-1}A_{1}\left(Q-Q^{T}\right)+B_{1}^{-1}B_{3}^{T}B_{1}\right)+2\pi \int_{\mathbb{R}^{N}}w^{T}Q w\left|\widehat{f}(w)\right|^{2}\mathrm{d}w\notag\\
&+2\pi \int_{\mathbb{R}^{N}}x^{T}\left(B_{1}^{-1}A_{1}QB_{1}^{-1}A_{1}-B_{1}^{-1}B_{3}^{T}A_{1}\right)x\left|f(x)\right|^{2}\mathrm{d}x\notag\\
&+2\pi \int_{\mathbb{R}^{N}}x^{T}\left(B_{1}^{-1}A_{1}\left(Q+Q^{T}\right)-B_{1}^{-1}B_{3}^{T}B_{1}\right)\nabla \varphi(x)\left|f(x)\right|^{2}\mathrm{d}x.
\end{align}
Note that $A_1$, $B_1$, $C_1$ and $D_1$ satisfy (\ref{conditions 1}), (\ref{conditions 2}), (\ref{conditions 3}) and (\ref{conditions 4}).
Since $B_{3}=B_{2}A_{1}^{T}-A_{2}B_{1}^{T}$, we have
\begin{equation*}
B_{1}^{-1}B_{3}^{T}B_{1}=B_{1}^{-1}A_{1}B_{2}^{T}B_{1}-A_{2}^{T}B_{1},
\end{equation*}
\begin{equation*}
B_{1}^{-1}B_{3}^{T}A_{1}
    =B_{1}^{-1}A_{1}B_{2}^{T}A_{1}-A_{2}^{T}A_{1}
\end{equation*}
and
\begin{align*}
   Q=&B_{1}^{T}B_{3}D_{1}\notag\\
   =&B_{1}^{T}B_{2}A_{1}^{T}D_{1}-B_{1}^{T}A_{2}B_{1}^{T}D_{1}\notag\\
   =&B_{1}^{T}B_{2}\left(\mathbf{I}_{N}+C_{1}^{T}B_{1}\right)-B_{1}^{T}A_{2}B_{1}^{T}D_{1}\notag\\
   =&B_{1}^{T}B_{2}+B_{1}^{T}B_{2}C_{1}^{T}B_{1}-B_{1}^{T}A_{2}B_{1}^{T}D_{1}.
\end{align*}
Then we have
\begin{align*}
B_{1}^{-1}A_{1}Q=&B_{1}^{-1}A_{1}B_{1}^{T}B_{2}+B_{1}^{-1}A_{1}B_{1}^{T}B_{2}C_{1}^{T}B_{1}-B_{1}^{-1}A_{1}B_{1}^{T}A_{2}B_{1}^{T}D_{1}\notag\\
    =&A_{1}^{T}B_{1}^{-T}B_{1}^{T}B_{2}+A_{1}^{T}B_{1}^{-T}B_{1}^{T}B_{2}C_{1}^{T}B_{1}-A_{1}^{T}B_{1}^{-T}B_{1}^{T}A_{2}B_{1}^{T}D_{1}\notag\\
    =&A_{1}^{T}B_{2}+A_{1}^{T}B_{2}C_{1}^{T}B_{1}-A_{1}^{T}A_{2}B_{1}^{T}D_{1},
\end{align*}
\begin{align*}
    B_{1}^{-1}A_{1}Q^{T}
=&B_{1}^{-1}A_{1}B_{2}^{T}B_{1}+B_{1}^{-1}A_{1}B_{1}^{T}C_{1}B_{2}^{T}B_{1}-B_{1}^{-1}A_{1}D_{1}^{T}B_{1}A_{2}^{T}B_{1}\notag\\
    =&B_{1}^{-1}A_{1}B_{2}^{T}B_{1}+A_{1}^{T}B_{1}^{-T}B_{1}^{T}C_{1}B_{2}^{T}B_{1}-B_{1}^{-1}\left(\mathbf{I}_{N}+B_{1}C_{1}^{T}\right)B_{1}A_{2}^{T}B_{1}\notag\\
    =&B_{1}^{-1}A_{1}B_{2}^{T}B_{1}+A_{1}^{T}C_{1}B_{2}^{T}B_{1}-A_{2}^{T}B_{1}-C_{1}^{T}B_{1}A_{2}^{T}B_{1}
\end{align*}
and 
\begin{align*}
    B_{1}^{-1}A_{1}QB_{1}^{-1}A_{1}
=&A_{1}^{T}B_{2}B_{1}^{-1}A_{1}+A_{1}^{T}B_{2}C_{1}^{T}A_{1}-A_{1}^{T}A_{2}D_{1}^{T}B_{1}B_{1}^{-1}A_{1}\notag\\
=&A_{1}^{T}B_{2}B_{1}^{-1}A_{1}+A_{1}^{T}B_{2}C_{1}^{T}A_{1}-A_{1}^{T}A_{2}D_{1}^{T}A_{1}.
`\end{align*}
Therefore we have
\begin{align}\label{B1A1Q-Q}
    &B_{1}^{-1}A_{1}\left(Q-Q^{T}\right)+B_{1}^{-1}B_{3}^{T}B_{1}\notag\\
    =&A_{1}^{T}B_{2}+A_{1}^{T}B_{2}C_{1}^{T}B_{1}-A_{1}^{T}A_{2}B_{1}^{T}D_{1}-A_{1}^{T}C_{1}B_{2}^{T}B_{1}+C_{1}^{T}B_{1}A_{2}^{T}B_{1},
\end{align}
\begin{align}\label{B1A1QB1A1}
&B_{1}^{-1}A_{1}QB_{1}^{-1}A_{1}-B_{1}^{-1}B_{3}^{T}A_{1}\notag\\
=&A_{1}^{T}B_{2}B_{1}^{-1}A_{1}+A_{1}^{T}B_{2}C_{1}^{T}A_{1}-A_{1}^{T}A_{2}D_{1}^{T}A_{1}-B_{1}^{-1}A_{1}B_{2}^{T}A_{1}+A_{2}^{T}A_{1}
\end{align}
and
\begin{align}\label{B1A1Q+Q}
&B_{1}^{-1}A_{1}\left(Q+Q^{T}\right)-B_{1}^{-1}B_{3}^{T}B_{1}\notag\\
=&A_{1}^{T}B_{2}+A_{1}^{T}B_{2}C_{1}^{T}B_{1}-A_{1}^{T}A_{2}B_{1}^{T}D_{1}+A_{1}^{T}C_{1}B_{2}^{T}B_{1}-C_{1}^{T}B_{1}A_{2}^{T}B_{1}.
\end{align}
Combining (\ref{I})-(\ref{B1A1Q+Q}), one has the desired equality (\ref{first lemma}). 
\end{Proof}

\begin{Proof1}
Let $g(x)=f(x)e^{\pi i x^{T}B_{1}^{-1}A_{1}x}$. By (\ref{FMT and FT}), we have
\begin{equation*}
    \mathcal{L}_{M_{1}}[f](u)=\frac{e^{\pi i u^{T}D_{1}B_{1}^{-1}u}}{i^{\frac{N}{2}}\sqrt{\mathrm{det}(B_{1})}}\widehat{g}\left(B_{1}^{-1}u\right).
\end{equation*}
Let $A_{3}=A_{2}D_{1}^{T}-B_{2}C_{1}^{T}$. Similar to proof of Proposition \ref{prop-LM}, by (\ref{grident g}) we have
   \begin{align*}
       &2\pi\int_{\mathbb{R}^{N}}u^{T}\left(A_{2}D_{1}^{T}-B_{2}C_{1}^{T}\right)u\left|\mathcal{L}_{M_{1}}[f](u)\right|^{2}\mathrm{d}u\notag\\
       =&\frac{2\pi}{\left|\mathrm{det}(B_{1})\right|}\int_{\mathbb{R}^{N}}u^{T}A_{3}u\left|\widehat{g}\left(B_{1}^{-1}u\right)\right|^{2}\mathrm{d}u\notag\\
       =&\frac{1}{2\pi}\int_{\mathbb{R}^{N}}\left(\nabla f(x)\right)^{T}B_{1}^{T}A_{3}B_{1}\overline{\nabla f(x)}\mathrm{d}x+2\pi\int_{\mathbb{R}^{N}}x^{T}A_{1}^{T}A_{3}A_{1}x\left|f(x)\right|^{2}\mathrm{d}x,\notag\\
&-i\int_{\mathbb{R}^{N}}x^{T}A_{1}^{T}A_{3}^{T}B_{1}\nabla f(x)\overline{f(x)}\mathrm{d}x+i\int_{\mathbb{R}^{N}}x^{T}A_{1}^{T}A_{3}B_{1}\overline{\nabla f(x)}f(x)\mathrm{d}x
   \end{align*}
     By (\ref{nabla f}), we have
\begin{align*}
  &-i\int_{\mathbb{R}^{N}}x^{T}A_{1}^{T}A_{3}^{T}B_{1}\nabla f(x)\overline{f(x)}\mathrm{d}x+i\int_{\mathbb{R}^{N}}x^{T}A_{1}^{T}A_{3}B_{1}\overline{\nabla f(x)}f(x)\mathrm{d}x\notag\\
=&i\int_{\mathbb{R}^{N}}x^{T}A_{1}^{T}\left(A_{3}-A_{3}^{T}\right)B_{1}\nabla \left|f(x)\right|\left|f(x)\right|\mathrm{d}x\notag\\
&+2\pi\int_{\mathbb{R}^{N}}x^{T}A_{1}^{T}\left(A_{3}+A_{3}^{T}\right)B_{1}\nabla \varphi(x)\left|f(x)\right|^2\mathrm{d}x\notag\\
=&-\frac{i}{2}\mathrm{tr}\left(A_{1}^{T}(A_{3}-A_{3}^{T})B_{1}\right)+2\pi\int_{\mathbb{R}^{N}}x^{T}A_{1}^{T}\left(A_{3}+A_{3}^{T}\right)B_{1}\nabla \varphi(x)\left|f(x)\right|^2\mathrm{d}x.
    \end{align*} 
Since
\begin{equation*}
   \frac{1}{2\pi}\int_{\mathbb{R}^{N}}\left(\nabla f(x)\right)^{T}B_{1}^{T}A_{3}B_{1}\overline{\nabla f(x)}\mathrm{d}x
    =2\pi\int_{\mathbb{R}^{N}} w^{T}B_{1}^{T}A_{3}B_{1}w\left|\widehat{f}(w)\right|^{2}\mathrm{d}w,
    \end{equation*}
we have
\begin{align}\label{J3}
 &2\pi\int_{\mathbb{R}^{N}}u^{T}\left(A_{2}D_{1}^{T}-B_{2}C_{1}^{T}\right)u\left|\mathcal{L}_{M_{1}}[f](u)\right|^{2}\mathrm{d}u\notag\\
    =&-\frac{i}{2}\mathrm{tr}\left(A_{1}^{T}(A_{3}-A_{3}^{T})B_{1}\right)+2\pi\int_{\mathbb{R}^{N}}x^{T}A_{1}^{T}A_{3}A_{1}x\left|f(x)\right|^{2}\mathrm{d}x\notag\\
    &+2\pi\int_{\mathbb{R}^{N}} w^{T}B_{1}^{T}A_{3}B_{1}w\left|\widehat{f}(w)\right|^{2}\mathrm{d}w+2\pi\int_{\mathbb{R}^{N}}x^{T}A_{1}^{T}\left(A_{3}+A_{3}^{T}\right)B_{1}\nabla \varphi(x)\left|f(x)\right|^2\mathrm{d}x.\notag\\
\end{align}
By $A_{3}=A_{2}D_{1}^{T}-B_{2}C_{1}^{T}$, we have
\begin{equation}\label{J5}
    A_{1}^{T}A_{3}A_{1}
     =A_{1}^{T}A_{2}D_{1}^{T}A_{1}-A_{1}^{T}B_{2}C_{1}^{T}A_{1},
\end{equation}
\begin{equation*}
A_{1}^{T}A_{3}B_{1}
     =A_{1}^{T}A_{2}D_{1}^{T}B_{1}-A_{1}^{T}B_{2}C_{1}^{T}B_{1}
\end{equation*}
and
\begin{align}\label{J6}
    B_{1}^{T}A_{3}B_{1}
    =&B_{1}^{T}A_{2}D_{1}^{T}B_{1}-B_{1}^{T}B_{2}C_{1}^{T}B_{1}.
\end{align}
Since $A_1$, $B_1$, $C_1$ and $D_1$ satisfy the last equality of (\ref{conditions 1}), we have
\begin{align*}
A_{1}^{T}A_{3}^{T}B_{1}
=&A_{1}^{T}D_{1}A_{2}^{T}B_{1}-A_{1}^{T}C_{1}B_{2}^{T}B_{1}\notag\\
=&\left(\mathbf{I}_{N}+C_{1}^{T}B_{1}\right)A_{2}^{T}B_{1}-A_{1}^{T}C_{1}B_{2}^{T}B_{1}\notag\\
=&A_{2}^{T}B_{1}+C_{1}^{T}B_{1}A_{2}^{T}B_{1}-A_{1}^{T}C_{1}B_{2}^{T}B_{1}.
\end{align*}
Therefore we have
\begin{align}\label{J4}
   A_{1}^{T}\left(A_{3}-A_{3}^{T}\right)B_{1}=&A_{1}^{T}A_{2}D_{1}^{T}B_{1}-A_{1}^{T}B_{2}C_{1}^{T}B_{1}-A_{2}^{T}B_{1}\notag\\
   &-C_{1}^{T}B_{1}A_{2}^{T}B_{1}+A_{1}^{T}C_{1}B_{2}^{T}B_{1}
\end{align}
and 
\begin{align}\label{J7}
A_{1}^{T}\left(A_{3}+A_{3}^{T}\right)B_{1}=&A_{1}^{T}A_{2}D_{1}^{T}B_{1}-A_{1}^{T}B_{2}C_{1}^{T}B_{1}+A_{2}^{T}B_{1}\notag\\
   &+C_{1}^{T}B_{1}A_{2}^{T}B_{1}-A_{1}^{T}C_{1}B_{2}^{T}B_{1}.
\end{align}
Combining (\ref{J3})-(\ref{J7}), one has (\ref{the third lemma}). 
\end{Proof1}

\end{document}